\PassOptionsToPackage{dvipsnames}{xcolor}
\documentclass[onefignum,onetabnum,twoside]{siamart171218}

\usepackage[T1]{fontenc}
\usepackage[utf8]{inputenc}
\usepackage[american]{babel}

\usepackage{amsmath}
\usepackage{amssymb}
\usepackage[square,numbers,sort&compress]{natbib}
\usepackage{subfig}
\usepackage{hyperref}
\usepackage{mathtools}
\usepackage{graphics}
\usepackage{float}
\usepackage{nicefrac}
\usepackage{cancel}
\usepackage{extarrows}
\usepackage{booktabs}

\DeclareMathAlphabet{\mathpzc}{OT1}{pzc}{m}{it}
\usepackage{mathbbol}
\usepackage{mathrsfs}
\usepackage{stmaryrd}
\usepackage{upgreek}

\usepackage{pdflscape}

\usepackage{tikz}
\usetikzlibrary
  {arrows,calc,through,backgrounds,matrix,positioning,decorations.pathmorphing,shapes.misc}

\newsiamremark{remark}{Remark}
\newsiamremark{example}{Example}
\newsiamremark{assumption}{Assumption}

\usepackage{ulem}
\usepackage{Macros}

\newcommand{\ie}{i.\,e.}

\headers%
  {Structure-preserving schemes for the Euler-Poisson equations}%
  {Maier, Tomas, Shadid}

\title{Structure-preserving finite-element schemes for the Euler-Poisson equations}

\begin{document}

\author{%
  Matthias~Maier\footnotemark[1]
  \and John~N.~Shadid\footnotemark[2] \footnotemark[3]
  \and Ignacio~Tomas\footnotemark[4] \footnotemark[5]
}

\maketitle

\renewcommand{\thefootnote}{\fnsymbol{footnote}}

\footnotetext[1]{%
  Department of Mathematics, Texas A\&M University, 3368 TAMU,
  College Station, TX 77843, USA (\email{maier@math.tamu.edu})}

\footnotetext[2]{%
  Sandia National Laboratories$^{1}$, P.O. Box 5800, MS 1320, Albuquerque,
  NM 87185, USA (\email{jnshadi@sandia.gov})}

\footnotetext[3]{%
  Department of Mathematics and Statistics, University of New Mexico, MSC01
  1115, Albuquerque, NM 87131, USA}

\footnotetext[4]{%
  Department of Mathematics and Statistics, Texas Tech University, 2500
  Broadway, Lubbock, TX 79409, USE (\email{igtomas@ttu.edu})}

\footnotetext[5]{Corresponding author}

\renewcommand{\thefootnote}{\arabic{footnote}}

\footnotetext[1]{%
  Sandia National Laboratories is a multimission laboratory managed and
  operated by National Technology \& Engineering Solutions of Sandia, LLC,
  a wholly owned subsidiary of Honeywell International Inc., for the U.S.
  Department of Energy's National Nuclear Security Administration under
  contract DE-NA0003526. This document describes objective technical
  results and analysis. Any subjective views or opinions that might be
  expressed in the paper do not necessarily represent the views of the U.S.
  Department of Energy or the United States Government.}

\begin{abstract}
  We discuss structure-preserving numerical discretizations for repulsive
  and attractive Euler-Poisson equations that find applications in
  fluid-plasma and self-gravitation modeling. The scheme is fully discrete
  and structure preserving in the sense that it maintains a discrete energy
  law, as well as hyperbolic invariant domain properties, such as
  positivity of the density and a minimum principle of the specific
  entropy. A detailed discussion of algorithmic details is given, as well
  as proofs of the claimed properties. We present computational experiments
  corroborating our analytical findings and demonstrating the computational
  capabilities of the scheme.
\end{abstract}



%%%%%%%%%%%%%%%%%%%%%%%%%%%%%%%%%%%%%%%%%%%%%%%%%%%%%%%%%%%%%%%%%%%%%%%%%%%%%%%%
%%%%%%%%%%%%%%%%%%%%%%%%%%%%%%%%%%%%%%%%%%%%%%%%%%%%%%%%%%%%%%%%%%%%%%%%%%%%%%%%
%%%%%%%%%%%%%%%%%%%%%%%%%%%%%%%%%%%%%%%%%%%%%%%%%%%%%%%%%%%%%%%%%%%%%%%%%%%%%%%%

\section{Introduction}

In this manuscript we develop numerical schemes for the repulsive and
attractive Euler-Poisson equations. This is a system of equations that
combine the hyperbolic compressible Euler equations of gas dynamics that
describe the time evolution of a fluid state (consisting of pressure,
momentum and total energy) with the action of a scalar potential that in
turn depends on the time evolution of the density of the system. The
Euler-Poisson equations have found applications in the context of plasma
physics \cite{Sentis2014}, semiconductor device modeling \cite{Marko1990},
and vacuum electronics \cite{Wohl2005}. The equations are often used to model an
electron fluid subject to electrostatic forces. The Euler-Poisson system
is also routinely used in astrophysics \cite{Zeldovich1989} for modeling
large scale formation of galaxies due to self-gravitation.

Our goal is to develop a numerical method for the Euler-Poisson system that
is second order accurate and provably robust. By this we mean that the
fully discrete update procedure at each time-step is \emph{locally
well-posed}, implying: (a) that the numerical scheme is always able to
compute a new \emph{admissible} state (i.\,e., with positive density and
internal energy); (b) it preserves a discrete energy law; and (c) that
the linear algebra only involves symmetric positive definite problems,
while the time-step size is only subject to a \emph{hyperbolic CFL}
condition. Our approach is based on an operator splitting in order to
decouple the hyperbolic and elliptic subsystems. We consider a
fully-discrete analysis of the scheme, revealing the need of specific
choices of space and time discretization that we make precise in Section
\ref{sec:numerical_approach}.

The splitting approach allows us to relegate invariant domain preservation
entirely to the numerical scheme used for the hyperbolic system; see
Section~\ref{sec:fundamentals}. This leaves considerable freedom for the
specific choice of hyperbolic solver. In particular, one could choose a
numerical method that preserves all invariant sets, or a subset of the
invariant set properties, such as positivity of density and internal
energy, see \cite{Guer2018, GuePo2019}. The resulting scheme will be
invariant domain preserving if the hyperbolic solver preserves all
invariants. For the sake of completeness we briefly describe the hyperbolic
solver used in this manuscript in Appendix \ref{app:hyperbolic}.

%%%%%%%%%%%%%%%%%%%%%%%%%%%%%%%%%%%%%%%%%%%%%%%%%%%%%%%%%%%%%%%%%%%%%%%%%%%%%%%%
\subsection{The Euler-Poisson system}
\label{subse:intro:governing_equations}
We consider a general model problem derived by coupling the compressible
Euler equations of gas dynamics to a scalar potential:
\begin{subequations}
  \label{Model}
  \begin{align}
    \label{RhoEqEP}
    \partial_t \den + \diver{}\mom &= 0 , \\
    \label{pEqEP}
    \partial_t \mom + \diver{}\big(\den^{-1} \mom \mom^\transp + I
    p\big) &= - \den \nabla\Epot - \frac{1}{\chartime} \mom , \\
    \label{TotMeEqEP}
    \partial_t \totme + \diver{}\Big(\frac{\mom}{\den} (\totme + \pre) \Big)
    &= - \nabla\Epot \cdot \mom - \frac{1}{\den\chartime} |\mom|^{2} ,
    \\
    \label{EfieldEqEPElliptic}
    - \Delta \Epot &= \alpha (\rho + \rho_b).
  \end{align}
\end{subequations}
Here, $\den(\xcoord,t) \in \mathbb{R}^+$ is the mass density,
$\mom(\xcoord,t) \in \mathbb{R}^d$, the momentum,
$\totme(\xcoord,t)\in\mathbb{R}^+$ the total energy, $\pre \in
\mathbb{R}$ denotes the thermodynamic pressure, and 
$\rho_b(\xcoord,t)\in\mathbb{R}$ denotes a prescribed background
density that, in contrast to the mass density $\den$, might attain negative
values. The balance of momentum and total energy equations \eqref{pEqEP} and
\eqref{TotMeEqEP} include a force caused by a scalar potential
$\Epot(\xcoord,t)$ whose time evolution in turn is coupled to
the density $\den(\xcoord,t)$ of the system. The model includes a
relaxation term $-\tfrac1\chartime\mom$ in the momentum equation where
$\chartime>0$ is a so-called relaxation time of the system. For example,
this could model resistive or collisional effects with an electrically
neutral species. By slight abuse of notation we set $\chartime= +\infty$
for the case of a vanishing relaxation effects. In case of a positive
coupling constant, $\alpha>0$, system \eqref{Model} is said to be
\emph{repulsive}, meaning, the force term $-\den \nabla\Epot$ in
\eqref{pEqEP} repels density accumulations. The corresponding case of a
negative coupling constant, $\alpha <0$, leads to an \emph{attractive}
system \eqref{Model} where density accumulations attract each other. We
close system~\eqref{Model} by prescribing a simple polytropic equation of
state where the pressure follows the ideal gas law, \ie, $p = (\gamma -
1)(\totme - \frac{1}{2}\rho^{-1}|\bv{m}|_{\ell^2}^2)$, with $\gamma = 5/3$.

Examples \ref{ex:electron_fluid} and \ref{ex:gravity} illustrate two
prototypical applications of the Euler-Poisson equations. Generally
speaking, mathematical models describing the dynamics of a ``density''
subject to its own self-consistent field is a somewhat universal theme in
mathematical physics \cite{Bardos2002, Gasser1997, Cances2003,
LionsCoulomb1987}. In this regard, the Euler-Poisson system can be viewed
as the simplest prototypical model that contains all necessary mechanisms.
Thus, developing and understanding numerical methods for the Euler-Poisson
system may open up a path for new ideas and development for other
self-consistent models appearing in quantum, molecular, and statistical
physics.

\begin{example}[Electron fluid plasma]
  \label{ex:electron_fluid}
  An important example of a repulsive system that can be expressed with
  \eqref{Model} is an electron fluid subject to an electrostatic force. Here,
  $\den$ takes the meaning of an electron mass density, and $\Epot$ is the
  electrostatic potential. The coupling constant in this case is given by
  $\alpha = \frac{1}{\varepsilon_0}\frac{\qe^2}{\me^2}$ where $\varepsilon_0$
  is the vacuum permittivity, $\qe$ is the specific electron charge (charge per
  unit particle), and $\me$ is the specific electron mass (mass per unit
  particle), $\Epot \in \mathbb{R}$ is the electric potential, and $\chartime$
  denotes a characteristic collision-time.
\end{example}

\begin{example}[Euler-Poisson gravity model]
  \label{ex:gravity}
  An example of an attractive system is given by an \emph{Euler-Poisson
  gravity model}. Here, the density $\den(\xcoord,t)$, momentum
  $\mom(\xcoord,t)$, and total mechanical energy $\totme(\xcoord,t)$ describe
  the time evolution of the matter of a celestial body subject to
  self-gravitation. The latter is expressed by the (classical) gravitational
  potential $\Epot(\xcoord,t)$ governed by \eqref{EfieldEqEPElliptic} with the
  coupling constant
  \begin{align*}
    \alpha=-4\pi\,G,
  \end{align*}
  where $G$ is the graviational constant, and where the relaxation terms
  have been removed from system \eqref{Model} by formally setting
  $\chartime=\infty$.
\end{example}

%%%%%%%%%%%%%%%%%%%%%%%%%%%%%%%%%%%%%%%%%%%%%%%%%%%%%%%%%%%%%%%%%%%%%%%%%%%%%%%%
\subsection{Structure preservation and stability properties}
\label{subse:intro:motivation}

Developing numerical schemes for the Euler-Poisson system \eqref{Model} is
not a trivial task. It is the authors' impression that the current body of
mathematical literature on the topic for provably robust, fully discrete,
self-consistent schemes for the Euler-Poisson system is far from complete
at this point in time. In the mathematical literature, the repulsive
electrostatic case (of Example~\ref{ex:electron_fluid}) has been studied by
Degond and collaborators \cite{Crispel2005, Crispel2007, Degond2008,
Degond2012}. For the attractive gravitational case
(Example~\ref{ex:gravity}) we refer the reader to \cite{Vides2014,
Jiang2013}. It is worth mentioning the astrophysics literature reporting
computations of fluids subject to self-gravitational effects is vast.
However, precise descriptions of the \emph{fully discrete} numerical
schemes and the mathematical properties guaranteed by such schemes, in
general, is not provided. We refer the reader to \cite{Truelove1998,
Almgreen2010, Springel2010, Bryan2014} where schemes used in practice are
discussed with some detail.

The fundamental question we wish to pose and examine is: \emph{What
mathematical properties should a ``good'' numerical scheme of the
Euler-Poisson system possess?} To this end we collect a list of desirable
properties that the continuous PDE exhibits (at least formally; see
Section~\ref{sec:fundamentals}).

We start with two important properties that imply local well-posedness
regardless of the numerical scheme of choice. By this we mean that for each time
step, the discrete map $\{\den^n, \mom^{n}, \totme^n, \Epot^n\} \rightarrow
\{\den^{n+1}, \mom^{n+1}, \totme^{n+1}, \Epot^{n+1}\}$ is
computable\footnote{Implying that a catastrophic failure due to occurrence
of negative density, negative internal energy, or a lack of convergence of
iterative linear solvers does not occur. }:
\begin{itemize}
  \item[(i)]
    Preservation of invariant domain: positivity of the density and
    pointwise minimum principle of the specific entropy \cite{GuePo2019,
    Guer2021}. These conditions ensure that the hyperbolic subsystem is
    locally well-posed at each time step \cite{GuePo2019, Guer2021}.
  \item[(ii)]
    Linear systems arising due to semi-implicit time marching, and their
    discretization in space, shall be \emph{well-posed} for all
    physically admissible regimes, meaning that: all linear systems are
    invertible and/or do not exhibit any artificial rank-deficiency.
\end{itemize}
Properties (i) and (ii) ensure that all discrete subsystems of our
numerical scheme are well-posed. Assuming these conditions are met, we can
pursue additional properties that the local map
$\{\den^n,\mom^{n},\totme^n, \Epot^n\} \rightarrow
\{\den^{n+1},\mom^{n+1},\totme^{n+1}, \Epot^{n+1}\}$ should preserve:
\begin{itemize}
  \item[(iii)]
    Satisfaction of a discrete total-energy balance.
  \item[(iv)]
    Preservation of the Hamiltonian structure of the source-dominated
    regime; see Remark~\ref{rem:hamiltonian} in
    Section~\ref{sec:source_dominated}.
  \item[(v)]
    Satisfaction of a discrete Gauß law \eqref{EfieldEqEPElliptic} at the
    end of each time-step.
\end{itemize}
In addition a number of discrete properties arise from enforcing
compatibility with limiting equations \cite{Miller2019, Crispel2005,
Crispel2007, Jungel1999}, in particular in the repulsive case ($\alpha>0$).
The numerical scheme:
\begin{itemize}
  \item[(vi)]
    Should not require to resolve the electrostatic plasma-oscillation
    time-scale (see Remark~\ref{rem:plasma_frequency} in
    Section~\ref{sec:source_dominated}); or at least the plasma-oscillation
    should not be a source of numerical instabilities in the scheme
    \cite{Miller2019}.
  \item[(vii)]
    Should be \emph{asymptotic-preserving} with respect to the
    quasi-neutral regime \cite{Crispel2005,Crispel2007}.
  \item[(viii)]
    Should be \emph{asymptotic-preserving} with respect to the
    drift-diffusion model. The drift-diffusion model is formally achieved in
    the zero-relaxation limit ($\chartime \rightarrow 0^+$), see for instance
    \cite{Jungel1999} and references therein.
\end{itemize}
This is a very ambitious list of desirable mathematical properties for the
discrete scheme that, to the best of our knowledge, no single numerical
discretization can achieve simultaneously. For instance, we are not aware
of any fully discrete numerical scheme satisfying, both, properties
\textrm{(iii)} and \textrm{(v)} simultaneously. Similarly, we are not aware
of any fully discrete scheme satisfying properties \textrm{(vii)} and
\textrm{(iii)} at the same time, and we are not aware of the existence of any
self-consistent scheme satisfying \textrm{(viii)}. We briefly mention that
another very important property is the preservation of steady states. There is
indeed a very large body literature on well-balanced schemes for Euler equations
with gravitation, see for instance \cite{Kappe2014, Chandra2015} and references
therein, but the potential is assumed to be `given' (not computed in
a self-consistent manner) in these and related references.

In this manuscript we focus on developing numerical schemes for system
\eqref{Model} for which properties \textrm{(i)}-\textrm{(iv)} can always be
guaranteed. We also present a relaxation technique in order to enforce
property \textrm{(v)} at the expense of introducing a consistency error in
time. In every case the potential is computed self-consistently.

%%%%%%%%%%%%%%%%%%%%%%%%%%%%%%%%%%%%%%%%%%%%%%%%%%%%%%%%%%%%%%%%%%%%%%%%%%%%%%%%
\subsection{Background and related literature}
\label{subse:intro:reformulation}

System \eqref{Model} may be understood as a hyperbolic PDE subject to an
elliptic constraint \eqref{EfieldEqEPElliptic}. As such it can be reformulated
in various different ways \cite{Pen1998, Vides2014, Jiang2013, Crispel2005,
Crispel2007}. For the sake of discussion we focus on two modifications of
the original system \eqref{Model}. First, we can reformulate \eqref{Model}
by rewriting writing the gradient of the potential in nonlocal form, viz.
$\nabla\Epot = \alpha \nabla(- \Delta)^{-1} \den$, and then substituting
into the remaining equations of $\eqref{Model}$ thereby eliminating the
last equation. This leads to the following \emph{nonlocal} formulation:

\begin{subequations}
  \label{ModelNonLocal}
  \begin{align}
    \partial_t \den + \diver{}\mom &= 0 , \\
    \partial_t \mom + \diver{}\big(\den^{-1} \mom \mom^\transp + I
    p\big) &= \boldsymbol{\mathit{f}}
            - \frac{1}{\chartime} \mom , \\
    \partial_t \totme + \diver{}\Big(\frac{\mom}{\den} (\totme + \pre) \Big)
    &= \frac{1}{\den}\boldsymbol{\mathit{f}} \cdot \mom -
    \frac{1}{\den\chartime} |\mom|^{2} ,
  \end{align}
\end{subequations}
where we have introduced a force $\boldsymbol{\mathit{f}} := - \alpha \den
(\nabla(-\Delta)^{-1}(\den + \den_b)$. Clearly, \eqref{ModelNonLocal}
is not any easier to solve than \eqref{Model}. Secondly, equation
\eqref{EfieldEqEPElliptic} can be rewritten as an evolution equation by
taking the time derivative and substituting \eqref{RhoEqEP}:
\begin{align}
  \label{EfieldEqEP}
  -\Delta \partial_t\Epot &= - \alpha\,\diver{}\mom
  \;+\; \alpha\partial_t\rho_{b}.
\end{align}
In this paper we exploit some aspects of the simplicity of
\eqref{ModelNonLocal}, since it can be viewed as just being the Euler
equations evolving the triple $[\den,\mom,\totme]^\transp$ subject to an
external force $\boldsymbol{\mathit{f}}$; see
Section~\ref{sec:fundamentals}. We also use the time dependent formulation
\eqref{EfieldEqEP} of the potential in order to avoid any kind of numerical
non-locality in the resulting hyperbolic scheme. This allows us to develop
a numerical scheme satisfying properties (i)--(iv). We also present a
relaxation technique in order to address property (v)---the preservation of
a discrete Gauß law---which may introduce some purely numerical dissipation
into the total energy-balance.

\begin{remark}
  \label{rem:modifications}
  It is certainly possible to reformulate system \eqref{Model} even
  further. A number of variants have been explored in the literature in
  this regard. For example, one approach is to rewrite the source terms
  containing $\rho\nabla\Epot$ in divergence-form leading to a nonlocal
  conservation law with fluxes that are nonlocal in space and time
  \cite{Pen1998, Vides2014, Jiang2013}:
  \begin{align}
    \nonumber
    \partial_t \den + \diver{}\mom &= 0 ,    \\
    \label{eq:nonlocal_condensed}
    \partial_t \mom + \diver{}\Big(\den^{-1} \mom \mom^\transp + I p +
    \alpha^{-1} (\nabla\Epot \nabla\Epot^\transp - \tfrac{1}{2}
    |\nabla\Epot|^2\mathbb{I})
    \Big) &= - \frac{1}{\chartime} \mom ,    \\
    \nonumber
    \partial_t E + \diver{}\Big( \frac{\mom}{\den} (E + \pre)
     + \frac{1}{2\alpha} (\Epot \nabla\partial_t\Epot
     - \partial_t \Epot \nabla\Epot)
      \Big)
    &= - \frac{1}{\den\chartime} |\mom|^{2} .
  \end{align}
  with $E = \totme + \frac{1}{2}\rho\Epot = \totme +
  \frac{1}{2\alpha}|\nabla\Epot|^2$. The conservation law structure is
  indeed interesting, but it has to be noted that such nonlocal fluxes have
  very different mathematical properties from those found for example for
  the Euler equations. This leads to a number of difficulties that can call
  to question the benefits of this approach. Most importantly, the force
  $\boldsymbol{\mathit{f}} := - \alpha \den (\nabla(-\Delta)^{-1}(\den +
  \rho_b))$, even if rewritten in divergence-form, exhibits an infinite
  speed of propagation: in this context the new balance law is unlikely to be
  hyperbolic. This is a roadblock if we want to guarantee robustness, as it
  would require to develop either a Godunov-type solver or at the very
  least finding a guaranteed estimate of the maximum wavespeed across the
  Riemann fan \cite{Toro2009, GuePo2016, GuePo2019, Toro2020}. This is
  problematic, since the maximum wavespeed across the Riemann fan is, by
  construction, infinite in this context. We are not the first ones to
  point-out related concerns, see for instance \cite[p. 48]{FrankShu1992}.

  Even if the issues mentioned above can be properly addressed, on a
  practical note, the discrete counterparts of the nonlocal fluxes in
  system \eqref{eq:nonlocal_condensed} exhibit a scaling of the form
  $\mathcal{O}(h^{-2})$ where $h$ denotes the mesh size. This poses the
  issue that explicit time-marching (which we are using) might then be
  subject to a \emph{parabolic} CFL condition of the form $\dt \sim h^2$.
\end{remark}

\begin{remark}
  A different approach one could pursue is to follow the steps of the
  pioneering work of \cite{Crispel2005,Crispel2007} on
  asymptotic-preserving schemes in the quasi-neutral regime (Property (vi)
  discussed in Section~\ref{subse:intro:motivation}). Such approaches,
  however, require major commitment to semi-implicit time integration
  of the \emph{hydrodynamical} subsystem. A situation for which---to the best
  of our knowledge---rigorously guaranteeing invariant domain preservation is
  a challenging and open mathematical problem; see \cite{Perth1996, Zhang2010,
  GuePo2019} and references therein.

  In particular, there are no general-purpose time-implicit schemes for
  nonlinear hyperbolic \emph{systems} of conservation laws with mathematical
  guarantees of local solvability and pointwise stability for the
  shock-hydrodynamics regime. This is particularly important if we are
  interested in the full-Euler system and not just barotropic models. Progress
  in this direction can be found in \cite{Herbin2011, Chalons2016,
  Badia2020, Tang2000, Crockatt2022}.
\end{remark}

Finally, in the vast majority of references presenting computations and/or
schemes for Euler-Poisson system, see for instance \cite{Crispel2005,
Crispel2007, Degond2008, Degond2012, Pen1998, Vides2014, Jiang2013,
Truelove1998, Almgreen2010, Springel2010, Bryan2014} and references therein,
the primary space discretization techniques are finite volumes for the
hyperbolic sub-system, and either finite differences and/or integral methods for
the elliptic operators, in cartesian meshes. However, during the last two
decades, there has been an explosive growth of versatile finite element
frameworks/libraries, see for instance \cite{dealII93, dealIIcanonical,
Fenics2015, Anderson2021}, capable of supporting a large number of
space-discretizations, linear solvers, preconditioners, and adaptivity
among many other features. There has also been significant development of
collocation methods that can solve hyperbolic systems of conservation laws
\cite{Guer2018, GuePo2019} using finite-element infrastructure while also
providing mathematically guaranteed robustness. In this sense the current
work is a significant point of departure from pre-existing literature: we
present new schemes that are meant to be coded entirely within the scope of
a finite element library. In particular, we use the finite element
library \texttt{deal.II} \cite{dealII93, dealIIcanonical} and exploit the 
mathematical framework of numerical schemes for hyperbolic systems of 
conservation laws developed in \cite{GuePo2019}.

%%%%%%%%%%%%%%%%%%%%%%%%%%%%%%%%%%%%%%%%%%%%%%%%%%%%%%%%%%%%%%%%%%%%%%%%%%%%%%%%
\subsection{Paper organization}
\label{subse:intro:outline}

The remainder of the paper is organized as follows. We discuss fundamental
(formal) stability properties of the continuous PDE and limits of the PDE
in Section~\ref{sec:fundamentals}. The discretization approach of our
numerical scheme is discussed in Section~\ref{sec:numerical_approach}.
Section~\ref{sec:gauss} introduces postprocessing strategies to maintain
the discrete Gauß law. Section~\ref{sec:numerics} presents computational
results that illustrate the properties of the proposed numerical schemes.
We conclude with a short summary and outlook in
Section~\ref{sec:conclusion}.

%%%%%%%%%%%%%%%%%%%%%%%%%%%%%%%%%%%%%%%%%%%%%%%%%%%%%%%%%%%%%%%%%%%%%%%%%%%%%%%%
%%%%%%%%%%%%%%%%%%%%%%%%%%%%%%%%%%%%%%%%%%%%%%%%%%%%%%%%%%%%%%%%%%%%%%%%%%%%%%%%
%%%%%%%%%%%%%%%%%%%%%%%%%%%%%%%%%%%%%%%%%%%%%%%%%%%%%%%%%%%%%%%%%%%%%%%%%%%%%%%%

\section{Fundamental stability properties}
\label{sec:fundamentals}

In preparation for Sections~\ref{sec:numerical_approach} and
\ref{sec:gauss} in which structure-preserving numerical schemes are
introduced, we first discuss some fundamental stability properties of the
PDE. We start with the Euler equations without source terms and then
proceed to the system \eqref{ModelNonLocal} and \eqref{EfieldEqEP}.

%%%%%%%%%%%%%%%%%%%%%%%%%%%%%%%%%%%%%%%%%%%%%%%%%%%%%%%%%%%%%%%%%%%%%%%%%%%%%%%%

\subsection{Euler equations with forces}

We start by collecting a number of useful properties of the solution
$\state:=[\den,\mom,\totme]^T$ of \eqref{ModelNonLocal} when fixing the
force $\force$ and in the case of no relaxation, \ie, $\chartime=\infty$.
In order to simplify notation we rewrite \eqref{ModelNonLocal} in compact form
as follows:
\begin{align}
  \label{EulerWithSources}
  \frac{\partial}{\partial t}\state + \diver{}
  \flux(\state) = \bv{s}(\force),
\end{align}
where
\begin{align}\label{EulerForceSystem}
  \flux(\state) = \begin{bmatrix}
    \mom^\transp \\
    \den^{-1} \mom \mom^\transp + I p \\
    \den^{-1 } \mom^\transp (\totme + p)
  \end{bmatrix},
  \qquad
  \bv{s}(\force) =
  \begin{bmatrix}
    0 \\ \force \\ \den^{-1}\mom \cdot \force
  \end{bmatrix}.
\end{align}
Here, deviating from Section~\ref{subse:intro:reformulation}, the force
density $\force:= \force(\xcoord,t):\mathbb{R}^{d} \times \mathbb{R}
\rightarrow\mathbb{R}^{d}$ shall be an arbitrary prescribed vector field
possibly depending on the state $\state$, the space $\xcoord$ and time $t$.
Note that $\den^{-1}\mom \cdot \force$, acting on the energy equation, is
the power of the force $\force$ per unit volume.

\begin{lemma}[Tangent-in-time invariance of the density and internal energy]
  \label{lem:InteInvar} Let $\Psi(\state):\mathbb{R}^{d+2} \rightarrow
  \mathbb{R}$ be any arbitrary functional of the state satisfying the
  functional dependence $\Psi(\state):= \psi(\den, \inte (\state))$ where
  $\inte (\state) := \totme - \frac{|\mom|^2}{2 \den}$ is the internal
  energy per unit volume. Then we have that
  \begin{align}
    \label{resultinv}
    \nabla_{\state} \Psi(\state) \cdot \bv{s}(\force) \equiv 0 ,
  \end{align}
  where $\nabla_{\state}$ is the gradient with respect to the state, \ie,
  $\nabla_{\state} = \big[\frac{\partial}{\partial\den},
  \frac{\partial}{\partial\mom_1}, ...,
  \frac{\partial}{\partial\mom_d},
  \frac{\partial}{\partial\totme}\big]^\transp$.
\end{lemma}
\begin{proof}
  Using the chain rule we observe that $\nabla_{\state} \Psi(\state) =
  \frac{\partial\psi}{\partial\rho} \nabla_{\state}\rho +
  \frac{\partial\psi}{\partial\inte} \nabla_{\state}\inte$, where
  \begin{align*}
    \nabla_{\state}\rho &= [1,0, ..., 0]^\transp \in \mathbb{R}^{d+2},
    \\
    \nabla_{\state}\inte &=
    \Big[\frac{|\mom|^2}{\den^2},-\frac{\mom_1}{\rho}, ...,
    -\frac{\mom_d}{\rho}, 1\Big]^\transp \in \mathbb{R}^{d+2}.
  \end{align*}
  Taking the product with $\bv{s}(\force)$ we get
  \begin{align*}
    \nabla_{\state} \Psi(\state) \cdot \bv{s}(\force)
    &= \frac{\partial\psi}{\partial\rho} \underbrace{\nabla_{\state}\rho
    \cdot \bv{s}(\force)}_{=\, 0} + \frac{\partial\psi}{\partial\inte}
    \nabla_{\state}\inte \cdot \bv{s}(\force)
    \\
    &= \frac{\partial\psi}{\partial\inte} (- \rho^{-1} \mom \cdot \bv{f} +
    \rho^{-1} \mom \cdot \bv{f}) = 0.
  \end{align*}
\end{proof}

\begin{remark}[Colloquial interpretation]\label{Colloq}
  Lemma \ref{lem:InteInvar} is simply saying that the evolution in time
  of an arbitrary functional of the state $\Psi(\state)$
  satisfying the functional dependence $\Psi(\state):= \psi(\den, \inte
  (\state))$ is independent of the force $\bv{f}$. This follows directly by
  taking the dot-product of \eqref{EulerWithSources} with $\nabla_{\state}
  \Psi(\state)$ to get
  \begin{align*}
    \nabla_{\state}\Psi(\state)\cdot \tfrac{\partial}{\partial t}\state
    \;=\;
    \tfrac{\partial}{\partial t}\Psi(\state)
    \;=\;
    - \nabla_{\state} \Psi(\state) \cdot \diver{} \flux(\state) +
    \underbrace{\nabla_{\state} \Psi(\state) \cdot
    \bv{s}(\force)}_{\equiv\,0} .
  \end{align*}
  In particular, this holds true when $\Psi(\state):= \inte (\state)$.
  Similarly, we can apply Lemma \ref{lem:InteInvar} to the specific internal
  energy $\specinte(\bv{u}) = \rho^{-1} \inte(\bv{u})$ since
  $\specinte(\bv{u})$ satisfies the functional dependence $\specinte(\bv{u}) =
  \psi(\den, \inte  (\state))$ as well.
\end{remark}

Lemma \ref{lem:InteInvar} and Remark \ref{Colloq} lead us to the following
immediate corollary.
\begin{corollary}\label{UnivCorollary}
  \label{EntropyDecl}
  Let $\spe(\rho,\specinte(\state))$ denote the specific internal entropy, and
  $\eta(\rho,\spe) = -\den\,\spe(\rho,\specinte(\state))$ denote the
  mathematical entropy. Then we have that
  \begin{align}\label{InvarianceEntropy}
    \nabla_{\state} \spe(\rho,\specinte) \cdot \bv{s}(\force) \equiv 0
    \quad\text{therefore}\quad
    \nabla_{\state} \eta(\rho,\spe)  \cdot
    \bv{s}(\force) \equiv 0.
  \end{align}
\end{corollary}
The statement follows by using the chain rule, and identity
\eqref{resultinv} respectively. This implies that forces (more precisely
source terms $\bv{s}(\force)$ as described by
expression \eqref{EulerForceSystem}) can modify neither the specific,
nor the mathematical entropy. In this manuscript we are only concerned with
the case that $\spe(\rho,\specinte) = \ln (\specinte^{\frac{1}{\gamma -1}}
\rho^{-1})$, which corresponds with the ideal gas closure. However,
Corollary \ref{UnivCorollary} remains valid for any equilibrium equation of
state
\cite{Harten1998,Guermond2014}.

%%%%%%%%%%%%%%%%%%%%%%%%%%%%%%%%%%%%%%%%%%%%%%%%%%%%%%%%%%%%%%%%%%%%%%%%%%%%%%%%
\subsection{Formal balance equation of the Euler-Poisson system}
\label{sec:EPBCs}

\begin{lemma}[Formal balance equation]
  \label{lemm:EPbalance}
  For the sake of simplicity let's assume that $\partial_t\rho_b\equiv0$, then
  the Euler-Poisson system \eqref{Model} satisfies the formal balance:
  \begin{multline}
    \label{intStabEP}
    \frac{\partial}{\partial t}
    \int_{\domain} \Big\{\totme + \frac{1}{2\alpha} |\nabla\Epot|^2\Big\} \dx
    \;+\;
    \int_{\partial\domain}
    \Big\{ \frac{\mom}{\den}(\totme + p)
    + \Epot\,\Big(\mom - \frac1\alpha\nabla\partial_t\Epot
    \Big)\Big\}\cdot\normal \ds
    \\
    \;=\;- \int_{\domain} \frac{1}{\den \chartime} |\mom|^{2}\dx.
  \end{multline}
\end{lemma}
Note that \emph{only} for the repulsive case $\alpha>0$ the scalar $\{\totme +
\frac{1}{2\alpha}|\nabla\Epot|^2\}$ is an energy density, and thus, Equation
\eqref{intStabEP} represents an energy-flux balance. In case of an
attractive Euler-Poisson system, that is $\alpha<0$, the expression
$\{\totme + \frac{1}{2\alpha}|\nabla\Epot|^2\}$ is more closely related to
a \emph{Lagrangian}, the difference of kinetic/total energy and potential
energy, see also Definition \ref{Def:source_dom}.

We call \eqref{intStabEP} a formal balance equation because it is valid
only under the assumption that the solution (as function of time) remains
sufficiently integrable in space. The hyperbolic character of \eqref{Model}
poses a fundamental challenge for mathematically deriving such
integrability properties. For example, in order for \eqref{intStabEP} to
remain a valid mathematical expression we would want to require that
$\nabla\Epot$ remains $L^2$-integrable in space. Suppose now that the Gauß
law \eqref{EfieldEqEPElliptic}, viz., $-\Delta\Epot = \alpha \rho$, is
satisfied for all time, then this requires $\rho \in \Hmone$. From well-known
Sobolev imbedding theorems (see for instance \cite[Ch.
5]{Evans1998}) we know for instance that $\|\rho\|_{\Hmone} \lesssim
\|\rho\|_{L^p(\Omega)}$, provided that $p \geq \frac{6}{5}$ in three
spatial dimensions, $d=3$. This means that the density should be $\rho \in
L^{p}(\Omega)$ with $p \geq \frac{6}{5}$ in order to guarantee that the
$\nabla\Epot$ remains in $\Ltwo$. This would be a questionable assumption,
since to the best of the authors knowledge, at present it is not possible
to guarantee such integrability for the density. Even worse, there is a
growing body of scientific literature \cite{Makino1992, Chae2008, Deng2002}
indicating that finite-time blow-up, \ie, the loss of integrability, indeed
happens.

Current mathematical evidence does not align in favor of identity
\eqref{intStabEP}. We nevertheless take a pragmatic standpoint in this
manuscript and make (an inequality version of) identity \eqref{intStabEP} a
guiding principle for our numerical algorithm development. In this sense
\eqref{intStabEP} should be understood as a regularity assumption that we
wish to maintain on the discrete level.

%%%%%%%%%%%%%%%%%%%%%%%%%%%%%%%%%%%%%%%%%%%%%%%%%%%%%%%%%%%%%%%%%%%%%%%%%%%%%%%%
\subsection{Asymptotic source dominated regime}
\label{sec:source_dominated}

In order to simplify some arguments in the following discussion we restrict
ourselves to the case of a constant (in time) background density, i.\,e.,
$\partial_t\rho_b\equiv0$. A generalization of the numerical approach to
time-dependent background densities is discussed in
Section~\ref{subse:background_density}.

\begin{definition}[Source-dominated regime]\label{Def:source_dom}
  The \emph{source-dominated} regime
  of the repulsive ($\alpha>0$) Euler-Poisson system
  \eqref{RhoEqEP}-\eqref{TotMeEqEP} and \eqref{EfieldEqEP} is given by the
  limiting equations in which all fluxes are set to zero (assuming
  $\partial_t\rho_b\equiv0$):
  \begin{subequations}
    \label{ODEModel}
    \begin{align}
      \label{ODErhoEP}
      \partial_t \den &= 0 , \\
      \label{ODEpEP}
      \partial_t \mom &= - \den \nabla\Epot - \frac{1}{\chartime} \mom , \\
      \label{ODEmeEP}
      \partial_t \totme &= - \nabla\Epot \cdot \mom -
      \frac{1}{\den\chartime} |\mom|^{2} ,
      \\
      \label{ODEeEP}
      - \Delta \partial_t\Epot &= - \alpha\,\diver{}\mom.
    \end{align}
  \end{subequations}
\end{definition}
Note that in \eqref{ODEModel} only $\mom(\xcoord,t)$ and $\Epot(\xcoord,t)$
remain as coupled unknowns governed by \eqref{ODEpEP} and \eqref{ODEeEP}. A
direct computation shows that system \eqref{ODErhoEP}-\eqref{ODEmeEP}
satisfies the identity $\partial_t \big(\totme - \tfrac{|\mom|^2}{2\den}
\big) \equiv 0$, see also Lemma \ref{lem:InteInvar}. In addition, subsystem
\eqref{ODEpEP}-\eqref{ODEeEP} satisfies an integral balance equation, viz.,
\begin{multline}
  \label{intStabODEEP}
  \frac{\partial}{\partial t}
  \int_{\domain} \Big\{\frac{1}{2\rho}|\mom|^2 + \frac{1}{2\alpha}
  |\nabla\Epot|^2\Big\} \dx
  + \int_{\partial\domain}
  \Big\{
  \Epot\,\Big(\mom - \frac1\alpha\nabla\partial_t\Epot
  \Big)\Big\}\cdot\normal \ds
  \\
  \;=\;
  - \int_{\domain} \frac{1}{\den \chartime} |\mom|^{2}\dx.
\end{multline}
We note in passing that \eqref{intStabODEEP} is an energy-flux balance
equation \emph{only} for the repulsive case $\alpha>0$. In case of a
negative coefficient $\alpha<0$, the expression $\{\tfrac{1}{2\rho}|\mom|^2
+ \tfrac{1}{2\alpha}|\nabla\Epot|^2\}$ is a Lagrangian, the difference of
kinetic and potential energy of the system. For the case $\alpha>0$ we make
an important observation:

\begin{remark}[Plasma frequency of the repulsive system, $\alpha>0$]
  \label{rem:plasma_frequency}
  By neglecting relaxation terms (formally setting $\chartime=\infty$) and by
  taking the divergence of \eqref{ODEpEP} as well as the time derivative of
  \eqref{ODEeEP} we arrive at
  \begin{align*}
    \diver{}\partial_t \mom = - \diver{}(\den \nabla\Epot)
    \ \text{and} \
    - \Delta\partial_{tt}\Epot = - \alpha \diver{}\partial_t\mom.
  \end{align*}
  Substituting the first expression into the second one:
  \begin{align}
    \label{secondOrderSource}
    - \Delta\partial_{tt}\Epot - \alpha\,\diver{}(\den \nabla\Epot) = 0.
  \end{align}
  Further simplifying \eqref{secondOrderSource} by assuming a spatially
  uniform density $\den(\xcoord,t) = \den_0$, we arrive at a simple
  harmonic oscillator equation
  \begin{align*}
    -\Delta\big(\partial_{tt}\Epot - \omega_p^2\Epot\big) = 0,
  \end{align*}
  where we have introduced the \emph{plasma frequency}
  \begin{align*}
    \omega_p
    \;=\;
    \sqrt{\den_0\alpha}
    \;=\;
    \sqrt{\frac{\rho_0\qe^2}{\varepsilon_0\me^2}}
  \end{align*}
  for the case of an electron fluid as described in
  Example~\ref{ex:electron_fluid}. We note that the plasma frequency
  $\omega_p$ tends to be very large: $\omega_p$ typically takes values in
  the GHz regime for most high-energy density applications (see for
  instance \cite[p. 12]{Bitten2004} and \cite[p. 56]{Goed2004}).
\end{remark}

\begin{remark}
  \label{rem:hamiltonian}
  In the absence of boundary terms, e.g. $\mom \cdot \normal \equiv 0$ and
  $\nabla\Epot \cdot \normal \equiv 0$, source-system
  \eqref{ODEpEP}-\eqref{ODEeEP} is Hamiltonian, and as such it makes sense
  to use a scheme that preserves its Hamiltonian structure. A
  natural choice of time-integration scheme for such a task is the
  Crank-Nicolson scheme, i.e., the implicit mid-point rule. However,
  blindly using the Crank-Nicolson scheme in time combined with some ad-hoc
  discretization in space in general will not lead to a well-posed scheme. In
  the following section we introduce a space and time discretization for
  source-system \eqref{ODEpEP}-\eqref{ODEeEP} that is capable of preserving
  the proper local-dynamics and stability properties associated to this
  system.
\end{remark}

%%%%%%%%%%%%%%%%%%%%%%%%%%%%%%%%%%%%%%%%%%%%%%%%%%%%%%%%%%%%%%%%%%%%%%%%%%%%%%%%
%%%%%%%%%%%%%%%%%%%%%%%%%%%%%%%%%%%%%%%%%%%%%%%%%%%%%%%%%%%%%%%%%%%%%%%%%%%%%%%%
%%%%%%%%%%%%%%%%%%%%%%%%%%%%%%%%%%%%%%%%%%%%%%%%%%%%%%%%%%%%%%%%%%%%%%%%%%%%%%%%

\section{A structure-preserving numerical discretization}
\label{sec:numerical_approach}

In order to gain some insight for deriving \emph{structure-preserving}
numerical schemes for \eqref{Model} we wish to start by discussing some of
the obstacles that one encounters when discretizing
\eqref{ODEpEP}-\eqref{ODEeEP} and the strategies to avoid them. The goal is
to preserve a discrete counterpart of the stability properties of
source-system \eqref{intStabODEEP}. The challenge is to come up with a
discretization strategy (in space) that leads to linear algebra systems
that are always well-posed (meaning invertible and reasonably well
conditioned).

In order to avoid some subtle technicalities we assume that
$\partial_t\rho_b\equiv0$, as well as $\bv{m}\cdot\normal \equiv 0$,
and either $\nabla\Epot \cdot\normal\equiv 0$ or $\Epot \equiv 0$ on the
entirety of the boundary $\partial\domain$. This eliminates the boundary
term in \eqref{intStabODEEP} thus making \eqref{ODEModel} an energetically
isolated system.

%%%%%%%%%%%%%%%%%%%%%%%%%%%%%%%%%%%%%%%%%%%%%%%%%%%%%%%%%%%%%%%%%%%%%%%%%%%%%%%%
\subsection{Notation}
\label{subse:notation}

In the following we mainly consider a mesh of simplices (triangles or
tetrahedra) and focus our attention on continuous nodal finite elements
spaces $\FESpacePot$ for the potential, and nodal scalar-valued
discontinuous finite elements space $\FESpaceNodalScal$ for each component
of the hyperbolic systems:
\begin{subequations}
  \label{triangSpaces}
  \begin{align}
    \FESpacePot \;&=\; \big\{\EpotTest_h \in \mathcal{C}^0(\Omega) \;\big|\;
    \EpotTest_h\circ\locglobmap_\element \in
    \mathbb{P}^1(\widehat{\element})
    \;\forall\,\element \in \triangulation\big\} \, , \\
    \FESpaceNodalScal &= \big\{ \hyptest_h \in
    \Ltwo \;\big|\;
    \hyptest_h\circ\locglobmap_\element \in
    \mathbb{P}^1(\widehat{\element}) \;\forall \element \in \triangulation
     \big\}.
  \end{align}
\end{subequations}
Here, $\locglobmap_\element:\widehat{\element}\to\element$ denotes a
diffeomorphism mapping the unit simplex $\widehat{\element}$ to the
physical element $\element \in \triangulation$. For the momentum and
velocity we introduce a vector-valued finite element space
$\FESpaceNodalVec := [\FESpaceNodalScal]^d$. Alternatively, we may also
consider a mesh consisting of quadrilaterals (or hexahedra) with finite
elements spaces $\FESpacePot$ and $\FESpaceNodalScal$ defined by
\begin{subequations}
  \label{QuadSpaces}
  \begin{align}
    \FESpacePot \;&=\; \big\{\EpotTest_h \in \mathcal{C}^0(\Omega) \;\big|\;
    \EpotTest_h\circ\locglobmap_\element \in
    \mathbb{Q}^1(\widehat{\element})
    \;\forall\,\element \in \triangulation\big\},
    \\
    \FESpaceNodalScal &= \big\{ \hyptest_h \in
    \Ltwo \;\big|\;
    \hyptest_h\circ\locglobmap_\element \in
    \mathbb{Q}^1(\widehat{\element}) \;\forall \element \in \triangulation
     \big\}.
  \end{align}
\end{subequations}
Let $\potvertices$ denote the set of vertices and let $\{\EpotBasis_i\}_{i
\in \potvertices}$ denote the corresponding nodal basis of the finite
element space $\FESpacePot$. Similarly, let $\hypvertices$ denote the set
of standard (Gauß-Lobatto) support points of the discontinuous finite
element space $\FESpaceNodalScal$, with a corresponding nodal basis
$\{\HypBasisScal_i\}_{i \in \hypvertices}$ of $\FESpaceNodalScal$. That is,
for every $v_h \in \FESpaceNodalScal$ there is a unique set of scalar
coefficients $\{V_i\}_{i \in \hypvertices}$ such that $v_h = \sum_{i \in
\hypvertices} V_i \HypBasisScal_i$.

\begin{remark}
  For all finite element spaces the basis functions are generated using 
  the reference-to-physical map $\locglobmap_\element$. That is, Lagrangian
  shape functions are defined in the reference element satisfying the
  property $\widehat{\phi}_k(\widehat{\xcoord}_j) = \delta_{jk}$ where
  $\{\widehat{\xcoord}_k\}_{k \in \mathcal{N}}$ are the coordinates of the
  interpolation nodes in the reference element, and $\mathcal{N}$ denotes
  the set of integers used to identify such nodes (e.g. $\mathcal{N} =
  \{1:4\}$ for $\mathbb{Q}^1(\widehat{\element})$ elements in 2d). In each
  physical element $\element$, the shape functions can be defined using a
  local indexation $\phi_{\element,i}(\xcoord) :=
  \widehat{\phi}_i(\locglobmap_\element^{-1}(\xcoord))$ for all $i \in
  \mathcal{N}$. Note that, in general, due to the composition with the
  inverse map $\locglobmap_\element^{-1}$, mapped-finite elements are not
  polynomial in physical space.
\end{remark}

Let $\mathcal{C}^0(\triangulation)$ denote the space of scalar-valued piecewise
continuous functions on the triangulation, that is: functions with well-defined
point-values on each element. Similarly we define the space of piecewise
continuous vector-valued functions as $[\mathcal{C}^0(\triangulation)]^d$. Let
$f,g \in \mathcal{C}^0(\triangulation)$: we define the bilinear form $\langle f,
g\rangle:\mathcal{C}^0(\triangulation) \times \mathcal{C}^0(\triangulation)
\rightarrow \mathbb{R}$ as follows:
\begin{align}
    \label{LumpingDef}
    \langle f, g \rangle := \sum_{\element \in \triangulation} \sum_{i \in
    \mathcal{N}} f(\xcoord_i) g(\xcoord_i) w_{\element,i},
\end{align}
where $\xcoord_i = \locglobmap_\element(\widehat{\xcoord}_i)$ and
$w_{\element,i} := \int_{K} \phi_{\element,i}(\xcoord) \dx$, and with an
obvious extension when $\boldsymbol{f}, \boldsymbol{g} \in
[\mathcal{C}^0(\triangulation)]^d$. Whenever the bilinear form $\langle
\cdot, \cdot\rangle$ is applied to finite dimensional spaces
$\FESpaceNodalScal$, $\FESpaceNodalVec$, $\FESpacePot$, or
$\nabla\FESpacePot$ we call it a lumped inner product.
In addition, we introduce the short-hand notation $m_i := \int_{\Omega}
\phi_{i}(\xcoord) \dx$ for $i\in\hypvertices$ denoting the diagonal entries
of the lumped mass matrix of $\FESpaceNodalScal$.

%%%%%%%%%%%%%%%%%%%%%%%%%%%%%%%%%%%%%%%%%%%%%%%%%%%%%%%%%%%%%%%%%%%%%%%%%%%%%%%%
\subsection{The method of lines and its potential shortcomings}
\label{sec:MOL}

For the sake of discussion we neglect all damping terms in
\eqref{ODEModel} in this section by formally setting $\chartime = +\infty$.
We reintroduce damping terms again in the full scheme in Section
\ref{subse:procedure}. We may start discretizing \eqref{ODEModel} by
introducing appropriate finite element spaces for approximations
$\Epot_h\in\FESpacePot$ for the potential $\Epot$, and
$\mom_h\in\FESpaceNodalVec$ for the momentum $\mom$. Testing \eqref{ODEpEP}
and \eqref{ODEeEP} with test functions $\EpotTest\in\FESpacePot$ and
$\momtest_h\in\FESpaceNodalVec$ and integrating by parts, as well as
discretizing in time with the Crank-Nicolson scheme leads to the following
scheme: For given $\rho_h^n, \mom_h^{n}, \Epot_h^{n}$,
\begin{align}\label{FeFunctions}
  \rho_h^{n} = \sum_{i \in \hypvertices} \rhoVect_i^n \HypBasisScal_i,
  \quad
  \mom_h^{n} = \sum_{i \in \hypvertices} \momVect_i^n \HypBasisScal_i,
  \quad
  \Epot_h^{n} = \sum_{i \in \potvertices} \Phi_i^n \EpotBasis_i,
\end{align}
we define $\vel_h^{n} := \sum_{i \in \hypvertices} \velVect_i^n \HypBasisScal_i$
with $\velVect_i^n = \frac{\momVect_i^n}{\rhoVect_i^n}$ for all $i \in
\hypvertices$. We want to find $\vel_h^{n+1}$ and $\Epot_h^{n+1}$ for time
$t_{n+1}$ solving
\begin{subequations}
  \label{UpVariEP}
  \begin{align}
    \label{eUpVariEP}
    \big(\nabla\Epot_h^{n+1} -
\nabla\Epot_h^{n},\nabla\EpotTest_h\big)_{\Ltwo}
    &\;=\;
    \frac{\dtn \alpha}{2}
    \big(\mom_h^{n+1} + \mom_h^{n},\nabla\EpotTest_h\big)_{\Ltwo},
    \\
    \label{pUpVariEP}
    \langle\vel_h^{n+1} - \vel_h^{n},\momtest_h\big>
    &\;=\;
    - \frac{\dtn}{2}
    \big( \nabla\Epot_h^{n+1} + \nabla\Epot_h^{n}, \momtest_h\big)_{\Ltwo},
  \end{align}
\end{subequations}
for all $\EpotTest_h \in \FESpacePot$ and $\veltest \in \FESpaceNodalVec$,
where $\dtn = t^{n+1} - t^n$ is the time step size at time step $n$. The
new velocity $\vel_h^{n+1}$ and momentum $\bv{m}_h^{n+1}$ shall be related
by
\begin{align*}
  \vel_h^{n+1} = \sum_{i \in \hypvertices} \velVect_i^{n+1}
  \HypBasisScal_i,
  \
  \mom_h^{n+1} = \sum_{i \in \hypvertices} \momVect_i^{n+1}
  \HypBasisScal_i,
  \ \text{with} \
  \momVect_i^{n+1} := \rhoVect_i^{n} \velVect_i^{n+1} \ \text{for all} \ i \in
  \hypvertices.
\end{align*}
Therefore, the unknowns of the linear problem \eqref{UpVariEP} are
$\{\velVect_i\}_{i \in \hypvertices}$ and $\{\EpotVect_i^n\}_{i \in
\potvertices}$. Momentum $\mom_h^{n+1}$ does not introduce additional
unknowns since it is related linearly to the velocity. This is somewhat 
more evident in the linear algebra context, see \eqref{LACfirst}. We also
note that the density field $\rho_h = \sum_{i \in \hypvertices} \varrho_i^{n}
\HypBasisScal_i$ is just given data for problem \eqref{UpVariEP} and does not
evolve during the source-update scheme.

\begin{lemma}\label{EnergyStabS1}
  Scheme \eqref{UpVariEP} is energy preserving in the following discrete
  sense:
  \begin{align*}
    \frac{1}{2}\big\|\nabla\Epot_h^{n+1}\big\|_{\Ltwo}^2 +
    \alpha\,
    \sum_{i \in \hypvertices} \frac{m_i}{2\varrho_i}
    \big|\momVect_i^{n+1}\big|_{\ell^2}^2
    \;=\;
    \frac{1}{2}\big\|\nabla\Epot_h^{n}\big\|_{\Ltwo}^2 +
    \alpha\,
    \sum_{i \in \hypvertices}  \frac{m_i}{2\varrho_i}
    \big|\momVect_i^{n}\big|_{\ell^2}^2.
  \end{align*}
\end{lemma}
\begin{proof}
  The statement follows by taking $\EpotTest_h = \tfrac{1}{2} (\Epot_h^{n+1} +
  \Epot_h^{n})$ in \eqref{UpVariEP} and taking $\veltest =
  \tfrac{1}{2}{\alpha}(\mom_h^{n+1} + \mom_h^{n})$ in
  \eqref{UpVariEP}, and adding both lines.
\end{proof}
The good news is that we have found a discrete analogue of the energy
stability property \eqref{intStabODEEP} for system \eqref{UpVariEP}.
However, we have to bring some attention to the algebraic difficulties that
are encountered when trying to actually solve algebraic system
\eqref{UpVariEP}. Introducing matrices
\begin{alignat*}{2}
  [\mathcal{M}^{\mathrm{L}}]_{ij} &=
    \delta_{ij} \int_{\domain} \HypBasisScal_i \dx, \quad
  &&[\mathcal{K}]_{ij} =
  \int_{\domain} \nabla\EpotBasis_i\cdot\nabla\EpotBasis_j \dx, \\
  [\mathcal{B}]_{ij} &=
    \int_{\domain} \nabla\EpotBasis_i \HypBasisScal_j\dx, \quad
  &&[\mathcal{D}_{\rho}]_{ij} = \delta_{ij} \, \rhoVect_i^n,
\end{alignat*}
system \eqref{UpVariEP} can be written as follows:
\begin{align}
  \begin{split}\label{LACfirst}
    \mathcal{K}\,(\EpotVect^{n+1}-\EpotVect^{n})
    \;&=\;
    \frac{\dtn\alpha}{2}\,\mathcal{B}\,
    \mathcal{D}_{\rho} (\velVect^{n+1} + \velVect^{n}),
    \\
    \mathcal{M}^{\mathrm{L}}\,(\velVect^{n+1} - \velVect^{n})
    \;&=\;
    -\frac{\dtn}{2}\, \mathcal{B}^T(\EpotVect^{n+1} + \EpotVect^{n}).
  \end{split}
\end{align}
We may proceed to eliminate the velocity vector by block-substitution arriving
at the following matrix system for the potential:
\begin{align*}
  &\underbrace{\Big(\mathcal{K}+\frac{\dtn^2\alpha}{4}\mathcal{B}
  \mathcal{D}_{\rho}\big(\mathcal{M}^L\big)^{-1}\mathcal{B}^T\Big)}
  _{:=\;S}\EpotVect^{n+1} =\\
  & \ \ \ \ \ \ \ \
  = \Big(\mathcal{K}-\frac{\dtn^2\alpha}{4}\mathcal{B}\mathcal{D}_{\rho}
  \big(\mathcal{M}^L\big)^{-1}\mathcal{B}^T\Big)\EpotVect^{n}
  -\frac{\dtn\alpha}{2}\mathcal{B}\momVect^n.
\end{align*}
Subject to appropriate boundary conditions, the block $\mathcal{K}$ is
positive-definite. Similarly, the Schur complement $S$ is symmetric and
invertible, since it is just a symmetric positive perturbation of the block
$\mathcal{K}$. However, the fact that $S$ is invertible may not have much
computational significance: for instance, the choice of an
\emph{equal-order} continuous bilinear finite-element ansatz,
$\FESpaceNodalScal=\FESpacePot=\mathbb{Q}^1$, invariably produces a
rank-deficient block $\mathcal{B}$, see for instance
\cite{Girault1986,ErnGuermond2004}.
As a consequence, the Schur complement $\mathcal{B}\mathcal{D}_{\rho}
\big(\mathcal{M}^L\big)^{-1} \mathcal{B}^T$ has a non-trivial kernel
which manifests graphically as the well-known \emph{checkerboard} modes.
This is of particular concern for the electron-fluid model discussed in
Example~\ref{ex:electron_fluid}, where the coupling constant $\alpha\gg1$
can be very large. In this case the Schur complement matrix $S$ is entirely
dominated by the rank-deficient second block (with the exception of
low-density regimes where $\rhoVect_i^n \ll 1$). Similar spurious defects
can also appear in ad-hoc finite-difference, finite-volume, or
discontinuous finite-element constructions. In other words, the
source-update scheme \eqref{UpVariEP} is stable in the sense that it
preserves the energy, but without a careful choice of finite element spaces
it may lead to an ill-conditioned algebraic system.

\begin{remark}[Stable choice of finite element spaces]
   The ill-conditioning induced by the rank deficiency of the
  block $\mathcal{B}$ can in principle be cured by using compatible,
  inf-sup stable finite element space tuples
  $\{\FESpaceNodalScal,\FESpacePot\}$, see \cite{Boffi2013}.
  For instance, a  possible choice is using a curl-conforming finite
  element space $\FESpaceNodalVec$ for the momentum supplemented by a
  choice $\FESpacePot$, satisfying the inclusion
  $\bv{\nabla}\FESpacePot\subset\FESpaceNodalVec$. If, in addition, we
  assume an almost uniform density distribution ($\mathcal{D}_{\rho}
  \approx \mathcal{I}$) it is possible to show that
  \begin{align*}
    \mathcal{B}\mathcal{D}_{\rho} \mathcal{M}^{-1}\mathcal{B}^T
    \sim \mathcal{K},
  \end{align*}
  which is a well-conditioned full rank matrix. However, the introduction
  of curl-conforming elements for $\FESpaceNodalVec$ creates a new
  problem: pretty much all mathematically rigorous%
  \footnote{In the sense that the schemes admit some provable, strong
  guarantees of pointwise stability.}
  schemes for hydrodynamical systems found in the literature are either based
  on \emph{nodal} discretizations, or require a notion of \emph{pointwise}
  state \cite{Perth1996, Zhang2010, Guer2018}. In particular, our intention is
  to use invariant domain preserving approximation techniques
  \cite{Guer2018,GuePo2019} which offer significant mathematical assurances of
  pointwise stability in the shock-hydrodynamics regime. To the best of
  our knowledge there is no overarching mathematical framework of discrete
  differential forms \cite{Bochev2006, Arnold2006} (or related concepts),
  capable of preserving maximum principles, invariant domain properties in
  phase space, or any other form of pointwise stability---numerical properties 
  which are required to approximate zero-viscosity limits and entropy-solutions.
\end{remark}

%%%%%%%%%%%%%%%%%%%%%%%%%%%%%%%%%%%%%%%%%%%%%%%%%%%%%%%%%%%%%%%%%%%%%%%%%%%%%%%%
\subsection{Energy-stable source-update scheme: affine-simplicial mesh}
\label{subse:scheme_affine}

\begin{figure}[t]
  \subfloat[]{%
  \begin{tikzpicture}
    \node[
      draw,thick,rounded rectangle, minimum width=2cm,minimum
      height=2em](pde) {PDE};
    \node[
      draw,thick,rounded rectangle, minimum width=2cm,minimum height=2em,
      below=4.5em of pde](abs) {Algebraic block system};
    \node[
      draw,thick,rounded rectangle, minimum width=2cm,minimum height=2em,
      below=2.50em of abs](asc) {Algebraic Schur complement};
    \draw[thick,-stealth] (pde.south) -- (abs.north)
      node[midway,fill=white] {Discretize in time and space};
    \draw[thick,-stealth] (abs.south) -- (asc.north)
      node[midway,fill=white] {Eliminate variables};
  \end{tikzpicture}}
  \hspace{2em}
  \subfloat[]{%
  \begin{tikzpicture}
    \node[
      draw,thick,rounded rectangle,minimum width=2cm,minimum
      height=2em](pde) {PDE};
    \node[
      draw,thick,rounded rectangle,minimum width=2cm,minimum height=2em,
      below=4.5em of pde](abs) {Lower-triangular PDE system};
    \node[
      draw,thick,rounded rectangle,minimum width=2cm,minimum height=2em,
      below=2.50em of abs](asc) {PDE Schur complement};
    \draw[thick,-stealth,align=center] (pde.south) -- (abs.north)
      node[midway,fill=white] {Discretize in time\\[0.5em]Eliminate variables};
    \draw[thick,-stealth] (abs.south) -- (asc.north)
      node[midway,fill=white] {Discretize in space};
  \end{tikzpicture}}
  \centering
  \caption{The algebraic structure emerging from the method of lines,
    or Rothe's method (a) compared to the algebraic structure when the
    elimination step is performed on a semi-discrete level (b).}
  \label{fig:discretization_approaches}
\end{figure}
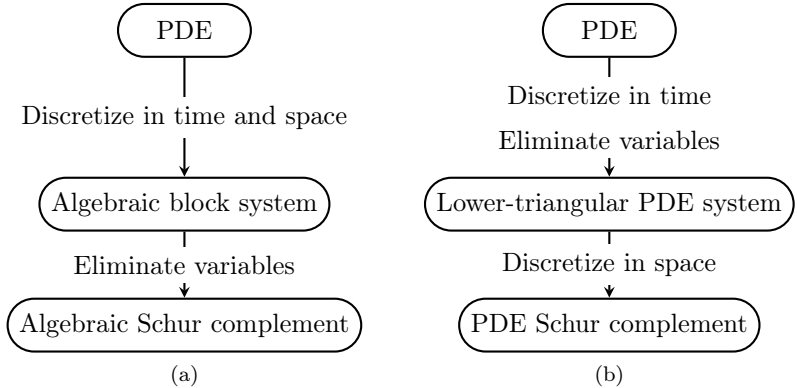
The algebraic structure of the scheme proposed in \eqref{UpVariEP} is a
consequence of the standard \emph{method of lines} discretization approach
(see Figure~\ref{fig:discretization_approaches}a) which transforms a PDE
into a coupled algebraic block system. Similarly, \emph{Rothe's method},
where the discretization in time is done first, leads to a similar
algebraic structure (Figure~\ref{fig:discretization_approaches}a) with the
same inherent difficulties as discussed in Section~\ref{sec:MOL}. In this
section we propose a different strategy in which we discretize in time
first and then eliminate variables on a semi-discrete level, see
Figure~\ref{fig:discretization_approaches}b.

We start by considering the following Crank-Nicolson semi-discretization of
the coupled system \eqref{ODEModel} augmented by an external force density
$\forcea(\xcoord, t)$ : given $\rho^{n}(\xcoord)$, $\mom^{n}(\xcoord)$ and
$\Epot^{n}(\xcoord)$ at time $t^n$, find $\mom^{n+1}(\xcoord)$ and
$\Epot^{n+1}(\xcoord)$ for time $t_{n+1}$ solving
\begin{align*}
  -\,\big(\Delta\Epot^{n+1} - \Delta\Epot^{n}\big) &\;=\; -\frac{\dtn\alpha}{2}
  \big(\diver{}\mom^{n+1} + \diver{}\mom^{n}\big),
  \\
  \mom^{n+1} - \mom^{n} &\;=\;
  -\frac{\dtn}{2} \rho^{n}\,\big(\nabla\Epot^{n+1} + \nabla\Epot^{n}\big)
  \;+\;\frac{\dtn}{2} \big(\forcea^{n+1}+\forcea^{n}\big).
\end{align*}
We now take the divergence of the second equation,
\begin{align*}
  \diver\mom^{n+1} \;=\; \diver\mom^{n}
  -\frac{\dtn}{2}\,\diver{}\Big(\rho^{n}\,\big(\nabla\Epot^{n+1} +
  \nabla\Epot^{n}\big)\Big)
  \;+\;\frac{\dtn}{2}\,\diver{}\big(\forcea^{n+1}+\forcea^{n}\big).
\end{align*}
and substitute into the first:
\begin{subequations}
  \label{ProtoSemi}
  \begin{multline}
    \label{ProtoSemiEpot}
    - \Delta\Epot^{n+1}
    - \frac{\dtn^2\alpha}{4} \diver{}\big(\rho^{n}\,\nabla\Epot^{n+1}\big)
    \;=\;
    -\dtn \alpha \diver \mom^n
    -\Delta\Epot^{n}
    \\
    +\frac{\dtn^2\alpha}{4} \diver{}\big(\rho^{n}\,\nabla\Epot^{n}\big)
    \;-\;\frac{\dtn^2\alpha}{4}
    \diver{}\big(\forcea^{n+1}+\forcea^{n}\big),
  \end{multline}
  \vspace{-1.5em}
  \begin{equation}
    \label{ProtoSemimom}
    \mom^{n+1} \;=\; \mom^{n} -\frac{\dtn}{2}
    \rho^{n}\,\big(\nabla\Epot^{n+1} + \nabla\Epot^{n}\big)
    + \frac{\dtn}{2} \big(\forcea^{n+1}+\forcea^{n}\big).
  \end{equation}
\end{subequations}
System~\eqref{ProtoSemi} is lower-triangular, in the sense that:
\eqref{ProtoSemiEpot} determines $\Epot^{n+1}$ and does not depend
on $\mom^{n+1}$. Once $\Epot^{n+1}$ is found \eqref{ProtoSemimom} determines
$\mom^{n+1}$.

Scheme \eqref{ProtoSemi} can be written in fully discrete, weak form as
follows: given $\vel_h^{n}\in\FESpaceNodalVec$ and
$\Epot_h^{n}\in\FESpacePot$ for time $t_{n}$ find
$\vel_h^{n+1}\in\FESpaceNodalVec$ and $\Epot_h^{n+1}\in\FESpacePot$ for
time $t_{n+1}$ solving
\begin{subequations}
  \label{ProtoSemiDiscrete}
  \begin{align}
    \label{ProtoSemiDiscreteEpot}
    a^+_{\dtn}(\Epot_h^{n+1}, \EpotTest_h)
    &\;=\;
    a^-_{\dtn}(\Epot_h^{n}, \EpotTest_h)
    \;+\;
    \dtn \alpha \langle\rho_h^n\vel_h^{n}, \nabla\EpotTest_h\rangle
    \;+\;\frac{\dtn^2\alpha}{4}
    \langle\forcea^{n+1}+\forcea^{n},\nabla\EpotTest_h\rangle,
    \\
    \label{ProtoSemiDiscretevel}
    \langle\rho_h^n\vel_h^{n+1},\momtest_h\rangle
    &\;=\;
    \langle\rho_h^n\vel_h^{n},\momtest_h\rangle
    -\frac{\dtn}{2}
    \langle\rho_h^n\{\nabla\Epot_h^{n+1} +
    \nabla\Epot_h^{n}\}-(\forcea^{n+1}+\forcea^{n}),\momtest_h\rangle.
  \end{align}
\end{subequations}
for all $\momtest_h\in\FESpaceNodalVec$ and $\EpotTest_h\in\FESpacePot$, where
the bilinear forms $a^+_{\dtn}(\Epot_h^{n+1}, \EpotTest_h)$ and
$a^-_{\dtn}(\Epot_h^{n}, \EpotTest_h)$ are defined by
\begin{align}
  \label{BilinearForm}
  a^\pm_{\dtn}(\Epot, \EpotTest) &\;:=\;
  ( \nabla\Epot,\nabla\EpotTest )
  \;\pm\; \frac{\dtn^2\alpha}{4}
  (\rho_h^n\nabla\Epot,\nabla\EpotTest) \, .
\end{align}
Note that in \eqref{ProtoSemiDiscrete} we have used a lumped inner product
product in \eqref{ProtoSemiDiscretevel} in order to recover a lumped
discrete kinetic energy (see Lemma~\ref{EnergyStabAffineLemma}).

This system is well posed provided that $\rho_h^n(\xcoord)>0$ and that the
inequality $\tfrac{\dtn^2\alpha}{4} \sup_{\xcoord\in\Omega}
\rho_h^n(\xcoord)>-1$ holds true. The latter imposes a mild restriction
on the step size for the self-gravitational case of
Example~\ref{ex:gravity} ($\alpha<0$), but none in the electrostatic case
of Example~\ref{ex:electron_fluid} ($\alpha>0$). Interestingly,
well-posedness holds true for \emph{any} choice of ansatz space
$\FESpaceNodalVec$ and $\FESpacePot$, thus, completely avoiding the
algebraic difficulties discussed in the previous Section~\ref{sec:MOL}.
Unfortunately, the discrete system \eqref{ProtoSemiDiscreteEpot} is in
general no longer energy preserving in the sense of \eqref{intStabODEEP}.
One way to remedy this shortcoming is by imposing very specific assumptions
on the finite-element ansatz spaces as detailed in the following lemma.

\begin{lemma}[Energy stability for affinely-mapped simplicial mesh]
  \label{EnergyStabAffineLemma}
  Consider the choice of finite element spaces described in
  \eqref{triangSpaces} with the additional restriction that the
  reference-to-physical map $\locglobmap_\element$ is affine for all
  elements $\element \in \triangulation$, then scheme
  \eqref{ProtoSemiDiscrete}  is energy-stable, more precisely, it satisfies
  the estimate:
  \begin{multline}
    \label{EnergyBalanceSourceUpdateStable}
    \frac{1}{2 \alpha} \big\|\nabla\Epot_h^{n+1}\big\|_{\bvLtwo}^2
    +\frac{1}{2} \sum_{i \in \hypvertices}m_i\varrho_i
    \big|\velVect_i^{n+1}\big|_{\ell^2}^2
    \;=\;
    \frac{1}{2 \alpha} \big\|\nabla\Epot_h^{n}\big\|_{\bvLtwo}^2
    +\frac{1}{2} \sum_{i \in \hypvertices}m_i\varrho_i
    \big|\velVect_i^{n}\big|_{\ell^2}^2
    \\
    \;+\;
    \frac{\dtn }{4}
    \sum_{i \in \hypvertices}m_i
    \big(\force_{\text{a},i}^{n+1}+\force_{\text{a},i}^{n},
    \velVect_i^{n+1}+\velVect_i^{n}\big)_{\ell^2}.
  \end{multline}
\end{lemma}

\begin{proof}
  Let $\Epot_h^{n+1}$ and $\vel^{n+1}_h$ be the solutions of the discrete
  system \eqref{ProtoSemiDiscrete}. From the assumptions of the lemma we
  have that the inclusion of the spaces $\nabla \FESpacePot \subseteq
  \FESpaceNodalVec$ holds true. Therefore we can set $\momtest_h =
  \frac{1}{2} \dtn \alpha \bv \nabla\EpotTest_h$ in
  \eqref{ProtoSemiDiscretevel} to get:
  \begin{multline}
    \label{clutter}
    \dtn\alpha \langle\rho_h^n\vel_h^{n},\nabla\EpotTest_h\rangle
    =
    \frac{\dtn\alpha}2
    \langle\rho_h^n\big\{\vel_h^{n+1}+\vel_h^{n}\big\},\nabla\EpotTest_h \rangle
    \\
    + \frac{\dtn^2\alpha}{4}
    \langle\rho_h^n\big\{\nabla\Epot_h^{n+1}+\nabla\Epot_h^{n}\big\}
    -(\forcea^{n+1}+\forcea^{n}),\nabla\EpotTest_h {\rangle}
  \end{multline}
  Substituting this identity into the right hand side of
  \eqref{ProtoSemiDiscreteEpot} and noting that%
  \footnote{%
    The identity follows from the exactness of the lumped quadrature for
    piecewice $\mathbb{P}^1(\element)$ functions, cf. \eqref{LumpingDef}.}
  \begin{align}
    \label{MagicP1property}
    \langle\rho_h^n\big\{\nabla\Epot_h^{n+1}+\nabla\Epot_h^{n}
    \big\},\nabla\EpotTest_h {\rangle}
    =
    (\rho_h^n\big\{\nabla\Epot_h^{n+1}+\nabla\Epot_h^{n}
    \big\},\nabla\EpotTest_h )
  \end{align}
  allows us to rewrite system \eqref{ProtoSemiDiscrete} as follows.
  \begin{subequations}
    \label{rewrite}
    \begin{align}
      \label{rewritePot}
      (\nabla\Epot_h^{n+1}, \nabla\EpotTest_h)
      & = (\nabla\Epot_h^{n}, \nabla\EpotTest_h)
      + \frac{\dtn\alpha}2
      \langle\rho_h^n\big\{\vel_h^{n+1}+\vel_h^{n}\big\},\nabla\EpotTest_h\rangle
      \\
      \label{rewriteMom}
      \langle\rho_h^n\vel_h^{n+1},\momtest_h\rangle
      &\;=\;
      \langle\rho_h^n\vel_h^{n},\momtest_h\rangle
      -\frac{\dtn}{2}
      \langle\rho_h\{\nabla\Epot_h^{n+1} + \nabla\Epot_h^{n}\}
      -(\forcea^{n+1}+\forcea^{n})
      ,\momtest_h \rangle.
    \end{align}
  \end{subequations}
  The statement now follows by testing \eqref{rewritePot} with
  $\EpotTest_h=\Epot_h^{n+1}+\Epot_h^{n}$, and \eqref{rewriteMom} with
  $\veltest=\alpha\,(\vel_h^{n+1}+\vel_h^n)$, adding both equations, and
  dividing both sides by $\alpha$.
\end{proof}

%%%%%%%%%%%%%%%%%%%%%%%%%%%%%%%%%%%%%%%%%%%%%%%%%%%%%%%%%%%%%%%%%%%%%%%%%%%%%%%%
\subsection{Energy-stable source-update scheme: non-simplicial mesh}
\label{subse:non_simplicial}

When considering the use of a non-simplicial mesh (for example by using
quadrilaterals, or hexahedra) one is faced with two main difficulties when
trying to repeat the steps of the proof of Lemma
\ref{EnergyStabAffineLemma}:
\begin{itemize}
  \item
    For affine meshes, while the inclusion property $\nabla\FESpacePot \subset
    \FESpaceNodalVec$ still holds true, the identity \eqref{MagicP1property}
    is no longer valid.
  \item
    For non-affine meshes, the situation is slightly worse: neither
    inclusion property $\nabla\FESpacePot \subset \FESpaceNodalVec$, nor
    property \eqref{MagicP1property} hold true anymore.
\end{itemize}
In order to get around these two obstacles we consider scheme
\eqref{ProtoSemiDiscrete} supplemented with the following \emph{lumped
version} of the bilinear form $a^\pm_{\dtn}(\Epot_h^{n+1}, \EpotTest_h)$:

\begin{align}
  \label{BilinearFormQuad}
  a^\pm_{\dtn}(\Epot, \EpotTest)
  &\;:=\;
  (\nabla\Epot,\nabla\EpotTest) \;\pm\; \frac{\dtn^2\alpha}{4}
  \langle\rho_h\nabla\Epot,\nabla\EpotTest\rangle.
\end{align}
Let $\mathcal{I}_{\FESpaceNodalVec}:\mathcal{C}^0(\triangulation)
\rightarrow \FESpaceNodalVec$ denote the nodal interpolant for the
piece-continuous space $\FESpaceNodalVec$. For every $\EpotTest_h \in
\FESpacePot$ the function
$\mathcal{I}_{\FESpaceNodalVec}[\nabla\EpotTest_h]$ is a valid test
function for \eqref{ProtoSemiDiscretevel}. In addition, we have that
\begin{align}
  \label{SemiCommuteProp}
  \langle \mathcal{I}_{\FESpaceNodalVec}[\nabla\EpotTest_h], \bv{z} \rangle =
  \langle \nabla\EpotTest_h, \bv{z} \rangle
\end{align}
for every $\bv{z} \in [\mathcal{C}^0(\triangulation)]^d$. Equipped with
definition \eqref{BilinearFormQuad} and identity \eqref{SemiCommuteProp},
it is now possible to establish the following lemma.

\begin{lemma}
  \label{substitutability}
  Consider the choice of non-simplicial finite element spaces described in
  \eqref{QuadSpaces}. Then, the scheme described by
  \eqref{ProtoSemiDiscrete} with the modified bilinear form
  \eqref{BilinearFormQuad} satisfies the energy estimate
  \eqref{EnergyBalanceSourceUpdateStable} as well.
\end{lemma}

\begin{proof}
  By setting $\momtest_h=\frac{1}{2} \dtn \alpha
  \mathcal{I}_{\FESpaceNodalVec}[\bv  \nabla\EpotTest_h]$ in
  \eqref{ProtoSemiDiscretevel} and using \eqref{SemiCommuteProp} we recover
  again identity \eqref{clutter}. Substituting this identity into
  \eqref{ProtoSemiDiscreteEpot} and using definition
  \eqref{BilinearFormQuad} allows us to rewrite \eqref{ProtoSemiDiscrete}
  again into \eqref{rewrite}. The proof now concludes in the same way as
  the proof of Lemma~\ref{EnergyStabAffineLemma}.
\end{proof}

\begin{remark}[Consistency error]
  Rigorously establishing that \eqref{BilinearFormQuad} possesses full
  second-order consistency on arbitrary meshes is a considerable challenge;
  see \cite[Ch. 13]{GuermondErnVolI}. In Section
  \ref{subse:num:convergence} we provide some numerical evidence that we
  indeed recover second order consistency for the class of an
  asymptotically affine mesh sequence.
\end{remark}

%%%%%%%%%%%%%%%%%%%%%%%%%%%%%%%%%%%%%%%%%%%%%%%%%%%%%%%%%%%%%%%%%%%%%%%%%%%%%%%%
\subsection{An operator-splitting scheme}
\label{subse:procedure}

We are now in a position to integrate the source update scheme developed in
the previous sections into a complete update scheme for the Euler-Poisson
equations by using either a first-order \emph{Yanenko} or a second-order
\emph{Strang} operator-splitting approach (see
\cite[Ch.~5]{Quarteroni1994}). We split the Euler-Poisson
equations into three operators: a hyperbolic update described by
\eqref{EulerWithSources} (without external forces),
\begin{align*}
  \partial_t\state + \diver{} \flux(\state) = 0,
\end{align*}
an undamped source update, \ie, $\chartime=\infty$, possibly including a
prescribed external force $\force_a$ and governed by \eqref{ODEModel},
viz.,
\begin{align}
  \begin{split}
    \label{sourceUp}
    \partial_t \den &= 0,
    \\
    \partial_t \mom &= - \den \nabla\Epot + \force_a,
    \\
    \partial_t \totme &= - \nabla\Epot \cdot \mom -
    \den^{-1}\force_{\text{a}} \cdot \mom ,
    \\
    - \Delta \partial_t\Epot &= - \alpha\,\diver{}\mom.
  \end{split}
\end{align}
and a pure damping operator,
\begin{align}
  \label{Damping}
  \partial_t \den = 0,
  \quad
  \partial_t \mom = - \frac{1}{\chartime} \mom,
  \quad
  \partial_t \totme = - \frac{1}{\den\chartime} |\mom|_{\ell^2}^{2},
  \quad
  - \Delta \partial_t\Epot = 0.
\end{align}
The source update \eqref{sourceUp} is now discretized using scheme
\eqref{ProtoSemiDiscrete}, which has to be combined with either definition
\eqref{BilinearForm} or \eqref{BilinearFormQuad} depending on the chosen finite
element space. A slight modification for the velocity update is necessary to
accommodate the additional external force $\force_{\text{a}}$;
see Algorithm~\ref{alg:source_update}. Damping effects in case of
$1\ll\chartime<\infty$ are incorporated using an additional layer of
operator splitting, where we have to solve system \eqref{Damping}; see
Algorithm~\ref{alg:damping}.

\begin{algorithm}
  \caption{\texttt{hyperbolic\_update}($\{\rho_h^n, \mom_h^n, \totme_h^n\}$)}
  \label{alg:hyperbolic_update}
  \begin{itemize}
    \item[1.]
      Compute largest feasible time-step size $\dt_n$ subject to CFL
      condition.
    \vspace{0.5em}
    \item[2.]
      Use initial data $\{\rho_h^n, \mom_h^n, \totme_h^n\}$ to compute
      update $\{\rho_h^{n+1}, \mom_h^{n+1}, \totme_h^{n+1}\}$.
  \end{itemize}
  \vspace{0.5em}
  Return $(\{\rho_h^{n+1}, \mom_h^{n+1}, \totme_h^{n+1}\}, \dt_n)$.
\end{algorithm}

\begin{remark}[Minimal assumptions for the hyperbolic solver]
  \label{MinimalAssumpt}
  For the remainder of the manuscript it is tacitly assumed that the
  hyperbolic solver invoked in the call to $\texttt{hyperbolic\_update}$,
  see Algorithm \ref{alg:hyperbolic_update}, is such that the inequalities
  \begin{align}
    \label{PositivityProp}
    \rho_i^{n+1} > 0,\quad
    \totme_i^{n+1} - \tfrac{\big|\momVect_i^{n+1}\big|_{\ell^2}^2}{2
    \varrho_i^{n+1}} > 0\quad \text{for all } i \in \hypvertices,
  \end{align}
  are guaranteed provided that \eqref{PositivityProp} already holds true
  for index $n$. Here, $\rho_i^n$ and $\momVect_i^{n}$ represent the
  density and momentum at the nodes, see also \eqref{FeFunctions}, while
  $\totme_i^n$ represents the total mechanical energy of Euler's system. We
  also assume that the total mechanical energy is preserved on discrete level,
  \begin{align}\label{ConsProp}
    \sum_{i \in \hypvertices} m_i \totme_i^{n+1} =
    \sum_{i \in \hypvertices} m_i \totme_i^{n},
  \end{align}
  whenever periodic or reflecting boundary conditions are used. Here, $m_i
  = \int_{\Omega} \phi_i(\xcoord) \dx$ is the lumped mass matrix, see
  Section \ref{subse:notation} for more details.
  We note in passing that the above (minimal) assumptions are primarily
  introduced to prove the following technical lemmas. In practice, it is
  often necessary to enforce significantly stricter pointwise stability
  properties to achieve a desired numerical fidelity; see \cite{Guer2018,
  GuePo2019}. For the convenience of the reader, we have succinctly
  summarized the hyperbolic solver used in all the computations of this
  paper in Appendix~\ref{app:hyperbolic} for which both properties
  \eqref{PositivityProp} and \eqref{ConsProp} are guaranteed.
\end{remark}

\begin{algorithm}
  
  \caption{\texttt{source\_update}($\{\rho_h^n, \mom_h^n, \totme_h^n,
  \Epot_h^n\}, \dt_n$)}
  \label{alg:source_update}
  \begin{itemize}
    \item[1.]
      Set $\velVect_i^n=\frac{1}{\rhoVect_i^n}{\momVect_i^{n}}$ and solve
      for the unknown $\Epot_h^{n+1}$, solution of the linear problem
      \begin{multline*}
        a^+_{\dt_n}(\Epot_h^{n+1}, \EpotBasis_i) =
        a^-_{\dt_n}(\Epot_h^{n}, \EpotBasis_i)
        \;+\; \dt_n \alpha \langle\rho_h^n\vel_h^{n},
        \nabla\EpotBasis_i\rangle
        \\
        \;+\;\frac{\dtn^2\alpha}{4}
        \langle\forcea^{n+1}+\forcea^{n},\nabla\EpotBasis_i\rangle
        \quad \text{for all } i \in \potvertices.
      \end{multline*}
    \item[2.] Update the velocity according to
      \begin{align*}
        \rhoVect_i^{n}\velVect_i^{n+1} =\;
        \rhoVect_i^{n}\velVect_i^{n} \;-\;
        \frac{\dt_n}{2}\rhoVect_i^{n}
        \,\big\{\nabla\Epot_h^{n+1}+\nabla\Epot_h^{n}\big\}(\xcoord_i)
        \;+\;
        \frac{\dt_n}{2}\big(\force^{n+1}_{\text{a},i} +
        \force^n_{\text{a},i}\big).
      \end{align*}
      Here, $\xcoord_i$ denotes the coordinate of the $i$-th interpolation
      point; see Section~\ref{subse:notation}.
    \item[3.] Set
      \begin{align*}
        \rhoVect_i^{n+1} = \rhoVect_i^n, \quad
        \momVect_i^{n+1}=\rhoVect_i^n\velVect_i^{n+1}\quad
        \text{for all }i\in\hypvertices
      \end{align*}
    \item[4.] Update of total mechanical energy.
      \begin{align}
        \label{eq:energy_update_proof_2_first}
        \totme_i^{n+1} \;&=\;
        \totme_i^{n} \;+\;
        \frac{1}{2\,\varrho^{n+1}_i}
        \Big( |\momVect_i^{n+1}|^2 - |\momVect_i^{n}|^2 \Big)
        \quad\text{for all }i\in\hypvertices.
      \end{align}
  \end{itemize}
  Return $(\{\rho_h^{n+1}, \mom_h^{n+1}, \totme_h^{n+1}, \Epot_h^{n+1}\},
  \dt_n)$.

\end{algorithm}

\begin{lemma}[Properties preserved by Algorithm \ref{alg:source_update}]
  \label{lem:balance_internal_energy}
  The source update (Algorithm \ref{alg:source_update}) does not modify the
  density and internal energy, i.\,e.,
  \begin{align}
    \label{InternalEnergyBalance}
    \rho_i^{n+1} = \rho_i^n \ , \ \
    \totme_i^{n+1} - \frac{\big|\momVect_i^{n+1}\big|_{\ell^2}^2}{2
    \varrho_i^{n+1}}
    =
    \totme_i^{n} - \frac{\big|\momVect_i^{n}\big|_{\ell^2}^2}{2
    \varrho_i^{n}} \ \ \text{for all } i \in \hypvertices.
  \end{align}
  Furthermore, it maintains a global energy-balance
  \begin{multline}
    \label{TotalEnergyBalanceSource}
    \sum_{i \in \hypvertices} m_i\,\totme_{i}^{n+1} \;+\;
    \frac{1}{2\,\alpha} \big\|\nabla\Epot_h^{n+1}\big\|_{\bvLtwo}^2
    =
    \sum_{i \in \hypvertices} m_i\,\totme_{i}^n \;+\; \frac{1}{2\,\alpha}
    \big\|\nabla\Epot_h^{n}\big\|_{\bvLtwo}^2
    \\
    \;+\;\sum_{i \in \hypvertices} m_i\,\frac{\dt_n}{4}\,
    \big(\force_{\text{a},i}^{n+1}+\force_{\text{a},i}^{n}\,,
    \velVect_i^{n}+\velVect_i^{n+1}\big)_{\ell^2}.
  \end{multline}
\end{lemma}
\begin{proof}
  Invariance of the density follows directly from the fact that Algorithm
  \ref{alg:source_update} does not modify the density, see Step 3.
  Invariance of the internal energy follows from a reorganization of
  \eqref{eq:energy_update_proof_2_first}. By adding for all nodes in the
  mesh, from \eqref{InternalEnergyBalance} it follows directly that
  \begin{align}\label{GlobInte}
    \sum_{i \in \hypvertices} m_i \Big( \totme_i^{n+1} -
    \frac{\big|\momVect_i^{n+1}\big|_{\ell^2}^2}{2 \varrho_i^{n+1}}\Big)
    =
    \sum_{i \in \hypvertices} m_i \Big( \totme_i^{n} -
    \frac{\big|\momVect_i^{n}\big|_{\ell^2}^2}{2 \varrho_i^{n}}\Big)
  \end{align}
  On the other hand, Steps 1 and 2 from Algorithm \ref{alg:source_update}
  implement formula \eqref{ProtoSemiDiscrete}, which was designed to comply with
  Lemma \ref{EnergyStabAffineLemma}, therefore they satisfy identity
  \eqref{EnergyBalanceSourceUpdateStable} by construction.
  Finally, \eqref{TotalEnergyBalanceSource} follows by adding
  \eqref{EnergyBalanceSourceUpdateStable} to \eqref{GlobInte}.
\end{proof}

We describe the relaxation update in Algorithm \ref{alg:damping}, which is
used to implement the effects of damping as described by
\eqref{Damping}. Finally, the entire update procedure using Yanenko
splitting is summarized in Algorithm~\ref{alg:yanenko}.

\begin{algorithm}

  \caption{\texttt{relaxation\_update}($\{\rho_h^n, \mom_h^n, \totme_h^n,
  \Epot_h^n \}, \dtn , \chartime $) }
  \label{alg:damping}
  \vspace{0.5em}
  \begin{enumerate}
    \item[1.]
      \emph{Relaxation update.} Operator \eqref{Damping} is discretized as
      follows.
      \begin{align*}
        \momVect_i^{n+1} \;&=\;
        e^{-\dtn/\chartime}\momVect_i^{n}
        \quad\text{for }i\in\hypvertices.
      \end{align*}
    \item[2.]
      \emph{Update of total mechanical energy.}
      \begin{align}
        \label{eq:energy_update_proof_2_second}
        \totme_i^{n+1} &= \totme_i^{n}  +
        \frac{1}{2\,\varrho^{n}_i}\Big(
        |\momVect_i^{n+1}|^2 - |\momVect_i^{n}|^2 \Big)
        \quad\text{for }i\in\hypvertices.
      \end{align}
    \item[3.]
      Set $\Epot_h^{n+1} = \Epot_h^{n}$, and $\rho_i^{n+1} = \rho_i^{n}$
      for all $i \in \hypvertices$.
  \end{enumerate}
  Return $[\rho_h^{n+1},\mom_h^{n+1},\totme_h^{n+1},
  \Epot_h^{n+1}]^\transp$.
\end{algorithm}

\begin{algorithm}

  \caption{First-order Yanenko splitting.}
  \label{alg:yanenko}
  Input discrete state
  $\bv{U}_h^{n} = [\rho_h^{n},\mom_h^{n},\totme_h^{n}, \Epot_h^n]^\transp$
  and initial time $t^n$, then:
  \vspace{0.5em}
  \begin{enumerate}
    \item[1.] \emph{Hyperbolic update.}
      Perform an update using Algorithm \ref{alg:hyperbolic_update}:
      \begin{align*}
        \{\rho_h^{n+1,1},\mom_h^{n+1,1},\totme_h^{n+1,1},\dt_n\} \leftarrow
        \texttt{hyperbolic\_update}(\{\rho_h^{n},\mom_h^{n},\totme_h^{n}\}).
      \end{align*}
    We set $\Epot_h^{n+1,1} = \Epot_h^{n}$.
    \item[2.]
      \emph{Source update.} Then perform a source update using
      Algorithm~\ref{alg:source_update}:
      \begin{align*}
      \{\rho_h^{n+1,2}, \mom_h^{n+1,2},\totme_i^{n+1,2},\Epot_h^{n+1,2}\}
      \leftarrow
      \texttt{source\_update}
      (\{\rho_h^{n+1,1},\mom_h^{n+1,1},
      \totme_i^{n+1,1},\Epot_h^{n+1,1}\}, \dt_n)
      \end{align*}
  \item[3.]
    \emph{Relaxation update.} Perform a source update using Algorithm
    \ref{alg:damping}
    \begin{align*}
      \{\rho_h^{n+1}, \mom_h^{n+1},\totme_i^{n+1},\Epot_h^{n+1}\}
      \leftarrow
      \texttt{relaxation\_update}
      (\{\rho_h^{n+1,2}, \mom_h^{n+1,2},
      \totme_h^{n+1,2},\Epot_h^{n+1,2}\}, \dt_n, \chartime)
    \end{align*}
  \end{enumerate}
  Return $[\rho_h^{n+1},\mom_h^{n+1},\totme_h^{n+1},
  \Epot_h^{n+1}]^\transp$ for final time $t_{n+1}:=t+\dt_n$.

\end{algorithm}
\begin{lemma}[Properties preserved by Algorithm \ref{alg:yanenko}]
  \label{lem:balance_total_energy} For the sake of simplicity we neglect
  damping mechanisms (i.e. $\chartime = + \infty$), we use the assumptions
  outlined in Remark \ref{MinimalAssumpt} for the hyperbolic solver, and
  assume periodic boundary conditions. Then the full update procedure described
  by Algorithms \ref{alg:yanenko} preserves positivity of density and internal
  energy, as well as the following energy balance:
  \begin{multline}
    \label{TotalEnergyBalance}
    \sum_{i \in \hypvertices} m_i\,\totme_{i}^{n+1}
    \;+\; \frac{1}{2\,\alpha} \big\|\nabla\Epot_h^{n+1}\big\|_{\bvLtwo}^2
    \;\le\;
    \sum_{i \in \hypvertices} m_i\,\totme_{i}^n
    \;+\; \frac{1}{2\,\alpha} \big\|\nabla\Epot_h^{n}\big\|_{\bvLtwo}^2
    \\
    \;+\;\sum_{i \in \hypvertices} m_i\,\frac{\dt_n}{4}\,
    \big(\force_{\text{a},i}^{n+1}+\force_{\text{a},i}^{n}\,,
    \velVect_i^{n+1,1}+\velVect_i^{n+1,2}\big)_{\ell^2},
  \end{multline}
  which is a discrete counterpart of the continuous energy balance
  \eqref{intStabEP} (neglecting boundary terms).
\end{lemma}

\begin{proof}
  
  The proof follows by analyzing the properties of the solution at the end
  of each step of Algorithms \ref{alg:yanenko}. We start with Step 1. Regarding
  positivity properties, we  use the assumption  \eqref{PositivityProp},
  described in Remark \ref{MinimalAssumpt}, in order to get that
  \begin{align}
    \label{Step1loc}
    \rho_i^{n+1,1} > 0,\quad
    \totme_i^{n+1,1} - \frac{\big|\momVect_i^{n+1,1}\big|_{\ell^2}^2}{2
    \varrho_i^{n+1,1}} > 0
    \quad \text{for all } i \in \hypvertices.
  \end{align}
  Using the assumptions on the hyperbolic solver \eqref{ConsProp}, and the
  invariance of the potential during Step 1, we have that
  \begin{align}
    \label{Step1glob}
    \sum_{i \in \hypvertices} m_i\,\totme_{i}^{n+1,1} + \frac{1}{2\,\alpha}
    \big\|\nabla\Epot_h^{n+1,1}\big\|_{\bvLtwo}^2
    =
    \sum_{i \in \hypvertices} m_i\,\totme_{i}^{n} + \frac{1}{2\,\alpha}
    \big\|\nabla\Epot_h^{n}\big\|_{\bvLtwo}^2 \, .
  \end{align}
  Expressions \eqref{Step1loc} and \eqref{Step1glob} conclude our analysis
  for the properties preserved at the end of Step 1.

  We proceed with the analysis of Step 2. From Lemma
  \eqref{lem:balance_internal_energy} we know that
  \begin{align}\label{PositFinal}
    \rho_i^{n+1,2} = \rho_i^{n+1,1}
    \quad \text{and}\quad
    \totme_i^{n+1,2} - \frac{\big|\momVect_i^{n+1,2}\big|_{\ell^2}^2}{2
    \varrho_i^{n+1,2}} = \totme_i^{n+1,1} -
    \frac{\big|\momVect_i^{n+1,1}\big|_{\ell^2}^2}{2
    \varrho_i^{n+1,1}}
  \end{align}
  for all $i \in \hypvertices$ at the end of Step 2. Regarding
  energy-balance properties, using property
  \eqref{TotalEnergyBalanceSource} of the source-update scheme we get
  \begin{multline}
    \label{Step2glob}
    \sum_{i \in \hypvertices} m_i\,\totme_{i}^{n+1,2} +
    \frac{1}{2\,\alpha} \big\|\nabla\Epot_h^{n+1,2}\big\|_{\bvLtwo}^2
    =
    \sum_{i \in \hypvertices} m_i\,\totme_{i}^{n+1,1} +
    \frac{1}{2\,\alpha} \big\|\nabla\Epot_h^{n + 1,1}\big\|_{\bvLtwo}^2
    \\ + \sum_{i \in \hypvertices} m_i\,\frac{\dt_n}{4}\,
    \big(\force_{\text{a},i}^{n+1}+\force_{\text{a},i}^{n}\,,
    \velVect_i^{n+1,1}+\velVect_i^{n+1,2}\big)_{\ell^2}.
  \end{multline}
  Finally under the assumption of $\chartime = +\infty$ we have that
  \begin{align*}
    [\rho_h^{n+1},\mom_h^{n+1},\totme_h^{n+1}, \Epot_h^{n+1}]^\transp =
    [\rho_h^{n+1,2},\mom_h^{n+1,2},\totme_h^{n+1,2},
    \Epot_h^{n+1,2}]^\transp.
  \end{align*}
  Adding \eqref{Step1glob} and \eqref{Step2glob} we recover identity
  \eqref{TotalEnergyBalance}. Finally, \eqref{PositFinal} implies
  positivity properties of the final solution. 
\end{proof}

Analogously, we can derive a second-order operator splitting. For the sake of
simplicity, we neglect damping effects and summarize a second-order
scheme in Algorithm~\ref{alg:strang}. Lemma \ref{lem:balance_total_energy}
also holds true for Algorithm~\ref{alg:strang} (combined with the relaxation
update Algorithm~\ref{alg:damping}). Only minimal adjustments of the
proofs are necessary to accommodate the second-order operator splitting. For
the sake of brevity we omit the details.
\begin{remark}[Total mechanical energy update]
  \label{RemarkEupdate}
  We note in passing that property \eqref{InternalEnergyBalance} is a
  discrete counterpart of property \eqref{resultinv}. More generally:
  forces should not affect the evolution of internal energy, see Lemma
  \ref{lem:InteInvar}. That is why formulas updating the total mechanical
  energy \eqref{eq:energy_update_proof_2_first} and
  \eqref{eq:energy_update_proof_2_second} are identical.
\end{remark}
\begin{algorithm}[t]
  \caption{Second-order Strang splitting.}
  \label{alg:strang}
  Let $\mathcal{S}_{\text{1}}(t^{n}+\dt_n,t^{n})$ represent the solution
  operator associated to the hyperbolic update (Step 1 of
  Algorithm~\ref{alg:yanenko}, and let $\mathcal{S}_{\text{2}}(t^{n}+ 2
  \dt_n,t^{n})$ denote the solution operator associated to the
  source-update (Step 2 of Algorithm~\ref{alg:yanenko}). For a given time
  $t^n$  and given  discrete state $[\rho_h^{n},\mom_h^{n},\totme_h^{n},
  \Epot_h^n]^\transp$ we  compute a second-order approximation as
  follows,
  \begin{multline}
    \label{eq:strang}
    \mathcal{S}_{\text{1}}(t^{n}+\dt_n,t^{n}+\dt_n)
    \circ
    \mathcal{S}_{\text{2}}(t^{n}+2\,\dt_n,t^{n})
    \circ
    \mathcal{S}_{\text{1}}(t^{n}+\dt_n,t^{n})
    \Big([\rho_h^{n},\mom_h^{n},\totme_h^{n}, \Epot_h^n]\Big).
  \end{multline}
\end{algorithm}

%%%%%%%%%%%%%%%%%%%%%%%%%%%%%%%%%%%%%%%%%%%%%%%%%%%%%%%%%%%%%%%%%%%%%%%%%%%%%%%%
\subsection{Incorporating a time-dependent background density}
\label{subse:background_density}

As discussed in Section~\ref{subse:intro:governing_equations}
the Euler-Poisson equations \eqref{Model} are often augmented by an
additional, prescribed background density $\rho_b$ driving the system.
Examples include incorporating an electrostatic potential into the system
caused by positive ions (in context of Example~\ref{ex:electron_fluid}), or
modeling a background density of \emph{dark matter} (in context of
Example~\ref{ex:gravity}) \cite{Zeldovich1989}.
The case of a time-independent background density is readily
incorporated into our numerical approach---it only requires to account for
the background density when computing the initial potential, no further
changes to Algorithm~\ref{alg:source_update} are required.

The case of a \emph{time-dependent} background density
$\rho_{b}(\xcoord,t)$ on the other hand requires a slight modification.
Here, the original evolution equation \eqref{EfieldEqEP} that include the
time derivative of the background potential have to be discretized. This
then leads to a modified linear system that has to be solved in
Algorithm~\ref{alg:source_update} reading
\begin{multline*}
  a^+_{\dt_n}(\Epot_h^{n+1}, \EpotBasis_i) =
  a^-_{\dt_n}(\Epot_h^{n}, \EpotBasis_i)
  \;+\; \dt_n \alpha \langle\rho_h^n\vel_h^{n},
  \nabla\EpotBasis_i\rangle
  \\
  \;+\; \alpha\, \langle \rho_{b}(.,t_{n+1}) - \rho_{b}(.,t_{n})
  ,\EpotBasis_i\rangle
  \quad \forall i \in \potvertices.
\end{multline*}

%%%%%%%%%%%%%%%%%%%%%%%%%%%%%%%%%%%%%%%%%%%%%%%%%%%%%%%%%%%%%%%%%%%%%%%%%%%%%%%%
%%%%%%%%%%%%%%%%%%%%%%%%%%%%%%%%%%%%%%%%%%%%%%%%%%%%%%%%%%%%%%%%%%%%%%%%%%%%%%%%
%%%%%%%%%%%%%%%%%%%%%%%%%%%%%%%%%%%%%%%%%%%%%%%%%%%%%%%%%%%%%%%%%%%%%%%%%%%%%%%%

\section{Gauß law restart}
\label{sec:gauss}

The combined method introduced in Section~\ref{subse:procedure} has one
notable defect in that it does not maintain Gauß's law, meaning a discrete
counterpart of
\begin{align}
  \label{eq:gauss}
  - \Delta \Epot = \alpha\, \den + \alpha\,\den_b
\end{align}
does not necessarily hold true. The reason for this is the fact that we
have chosen to use \eqref{EfieldEqEP} as a starting point for our
discretization instead of \eqref{eq:gauss}.%
\footnote{The validity of \eqref{EfieldEqEP} is a direct
consequence of \eqref{eq:gauss} combined with the conservation of mass
property \eqref{RhoEqEP}.}
This choice was motivated by the desire to maintain an energy inequality
\eqref{TotalEnergyBalance} on the discrete level.

We now discuss a family of postprocessing procedures aimed at
recovering a discrete counterpart of \eqref{eq:gauss} while maintaing the
discrete energy inequality \eqref{TotalEnergyBalance}.

%%%%%%%%%%%%%%%%%%%%%%%%%%%%%%%%%%%%%%%%%%%%%%%%%%%%%%%%%%%%%%%%%%%%%%%%%%%%%%%%
\subsection{Full restart of the potential}
\label{subse:full_restart}

The simplest possible procedure consists of simply resetting the computed
potential $\Epot_h^{n+1}$ at the end of a Yanenko step
(Algorithm~\ref{alg:yanenko}) or a Strang step
(Algorithm~\ref{alg:strang}). In light of the lumping strategy employed in
Section~\ref{sec:numerical_approach} for \eqref{ProtoSemiDiscrete} we
propose to first solve for $\tilde\varphi_h^{n+1}\in\FESpacePot$ given by
\begin{align}
  \label{eq:discrete_gauss}
  (\nabla\tilde\Epot_h^{n+1},\nabla\EpotTest)
  \;=\; \alpha\,\langle\rho_h^{n+1}+\rho_b(.,t_{n+1}),\EpotTest\rangle
  \quad\text{for all }\EpotTest\in\FESpacePot,
\end{align}
and then setting $\Epot_h^{n+1}\leftarrow\tilde\Epot_h^{n+1}$. With this
postprocessing strategy the discrete Gauß law (property (v)) is maintained
at the expense of possibly violating the discrete total energy balance (property
(iii), Lemmas~\ref{EnergyStabAffineLemma} and
\ref{EnergyBalanceSourceUpdateStable}). More precisely, unless the inequality
$\|\nabla\tilde\Epot_h^{n+1}\|_{L^2(\Omega)} \leq
\|\nabla\Epot_h^{n+1}\|_{L^2(\Omega)}$ holds true, simply re-setting the
potential cannot lead to an energy stable scheme. The following lemma
establishes error bounds on the consistency error, as well as the total energy
balance violation.
\begin{lemma}
  \label{lem:a_new_hope}
  Let $\bv{U}_h^{n} = [\rho_h^{n},\mom_h^{n},\totme_h^{n},
  \Epot_h^n]^\transp$ be a given discrete state satisfying
  \eqref{eq:discrete_gauss}. Let $\bv{U}_h^{n+1} =
  [\rho_h^{n+1},\mom_h^{n+1},\totme_h^{n+1}, \Epot_h^{n+1}]^\transp$ be the
  update computed by Algorithm~\ref{alg:yanenko}, or~\ref{alg:strang}.
  \\
  (a) Let $\tilde\Epot_h^{n+1}$ be the solution to
  \eqref{eq:discrete_gauss}. Then, under the CFL condition \eqref{eq:cfl},
  further assuming that the hyperbolic update procedure
  (Appendix~\ref{app:hyperbolic}) and the source update procedure
  (Section~\ref{sec:numerical_approach}) is second order accurate in time
  and space, and assuming they produce a sequence of solutions globally
  bounded in the $L^\infty$ norm, then the following estimates hold true:
  \begin{align*}
    \|\nabla\widetilde{\Epot}_h^{n+1} - \nabla\Epot_h^{n+1}\|_{\Ltwo} \;&=\;
    \alpha\, \mathcal{O}(\tau_n\,h),
    \\
    \big|\|\nabla\widetilde{\Epot}_h^{n+1}\|^2_{\Ltwo}
    - \|\nabla\Epot_h^{n+1}\|^2_{\Ltwo}\big| \;&=\;
    \alpha\, \mathcal{O}(\tau_n\,h).
  \end{align*}
  (b) Define a Gauß law residual $\mathcal{R}_h^{n+1}\in\FESpacePot'$ as
  follows
  \begin{align}
    \label{eq:gauss_law_residual}
    \mathcal{R}_h^{n+1}[\EpotTest]
    \;=\; \alpha\,
    \langle\rho_h^{n+1}+\rho_b(.,t_{n+1}),\EpotTest\rangle
    -(\nabla\Epot_h^{n+1},\nabla\EpotTest).
  \end{align}
  Then under the same assumptions as in (a) we have that
  $\|\mathcal{R}_h^{n+1}\|_{\FESpacePot'} \;=\;\mathcal{O}(\tau_n\,h)$.
\end{lemma}
\begin{proof}
  We outline a proof for the Yanenko split. The corresponding result for
  Strang splitting follows analogously. A key observation for proving
  energy stability of the source update scheme was identity
  \eqref{rewrite}, in particular \eqref{rewritePot} reads when accounting
  for the previous hyperbolic update step (in the notation of
  Algorithm~\ref{alg:yanenko} and reintroducing the background charge
  as discussed in Section \ref{subse:background_density}):
  \begin{multline*}
    (\nabla\Epot_h^{n+1}, \nabla\EpotTest_h)
    = (\nabla\Epot_h^{n}, \nabla\EpotTest_h)
    + \frac{\dtn\alpha}2
    \langle\rho_h^{n+1}\big\{\vel_h^{n+1,2}+\vel_h^{n+1,1}\big\},\nabla\EpotTest_h\rangle
    \\
    \;+\;\alpha\,\langle \rho_{b}(.,t_{n+1}) - \rho_{b}(.,t_{n})
    ,\EpotTest_h\rangle.
  \end{multline*}
  Using the assumption on $\nabla\Epot_h^n$ to satisfy the discrete Gauß
  law implies
  \begin{multline*}
    (\nabla\Epot_h^{n+1}, \nabla\EpotTest_h)
    = \alpha\,\Big\{\langle\rho_h^{n}+\rho_b(.,t_{n}),\EpotTest_h\rangle
    + \frac{\dtn}2
    \langle\rho_h^{n+1}\big\{\vel_h^{n+1,2}+\vel_h^{n+1,1}\big\},\nabla\EpotTest_h\rangle
    \Big\}
    \\
    +\alpha\,\langle \rho_{b}(.,t_{n+1}) - \rho_{b}(.,t_{n})
    ,\EpotTest_h\rangle.
  \end{multline*}
  The stated assumption of convergence of the hyperbolic and source update
  steps combined with the assumption that the sequence of approximants
  remains bounded in the $L^\infty$-norm now implies that
  \begin{align*}
    \rho_h^{n+1}v_h^{n+1,2}
    \;=\; \rho_h^{n+1}v_h^{n+1,1} + \mathcal{O}(\tau_n)
    \;=\; \rho_h^{n}v_h^{n} + \mathcal{O}(\tau_n),
  \end{align*}
  where the constant in front of the $\tau_n$ is uniform in $n$.
  Substituting:
  \begin{multline*}
    (\nabla\Epot_h^{n+1}, \nabla\EpotTest_h)
    = \alpha\,\Big\{
    \underbrace{\langle\rho_h^{n},\EpotTest_h\rangle + \frac{\dtn}2
    \langle\rho_h^{n}\vel_h^{n}+\rho_h^{n+1}\vel_h^{n+1,1},
    \nabla\EpotTest_h\rangle}_{\text{(I)}} \Big\}
    \\
    + \alpha\,\Big\{\langle\rho_{b}(.,t_{n+1}),\EpotTest_h\rangle\Big\}
    +\alpha\,\mathcal{O}(\tau_n^2)\|\nabla\omega_h\|.
  \end{multline*}
  We now stipulate that (I) has the formal structure of a Crank-Nicolson
  time step approximating the balance of mass equation, viz.,
  $\partial_t\rho+\diver{}\mom=0$. Using that the hyperbolic update
  procedure itself is second order accurate in time and space we infer that
  this implies
  $\text{(I)}=\langle\rho_h^{n+1},\omega_h\rangle+
  \mathcal{O}(\tau_n\,h^2+\tau_n^3)\|\omega_h\|$. We conclude that
  \begin{align*}
    (\nabla\Epot_h^{n+1}, \nabla\EpotTest_h)
    & =
    \alpha\langle\rho_h^{n+1}+\rho_{b}(.,t_{n+1}),\omega_h\rangle
    +\alpha\,\mathcal{O}(\tau_n^2)\|\nabla\omega_h\|.
    \\
    & =
    (\nabla\Epot_h^{n}, \nabla\EpotTest_h)
    +\alpha\,\mathcal{O}(\tau_n\,h)\|\nabla\omega_h\|.
  \end{align*}
  Here, we have used the fact that CFL condition \eqref{eq:cfl} combined
  with our assumptions ensures that $\tau_n \lesssim h$. Substituting
  $\omega_h=\Epot_h^{n+1}\pm\Epot_h^{n}$ shows (a). Similarly, (b) follows
  from the fact that
  \begin{align*}
    \|\mathcal{R}_h^{n+1}\|_{\FESpacePot'}
    = \sup_{\EpotTest_h \in \FESpacePot}
    \frac{|\alpha\,\langle\rho_h^{n+1}+\rho_{b}(.,t_{n+1}),\EpotTest\rangle
    -(\nabla\Epot_h^{n+1},\nabla\EpotTest)|}{\|\EpotTest_h\|_{\Hone}}
    \le \|\nabla\tilde\Epot_h^{n+1}-\nabla\Epot_h^{n+1}\|,
  \end{align*}
  which concludes the proof.
\end{proof}
In summary this implies that restarting the potential after every time step
by solving \eqref{eq:discrete_gauss} does not degrade the approximation
order of the scheme, but does introduce an energy balance violation of
the same order.

%%%%%%%%%%%%%%%%%%%%%%%%%%%%%%%%%%%%%%%%%%%%%%%%%%%%%%%%%%%%%%%%%%%%%%%%%%%%%%%%
\subsection{Energy balance via artificial relaxation}
\label{subse:relax}
The energy balance violation introduced by the full restart is not
desirable. As a remedy we propose an additional relaxation that
reestablished the energy balance. Given $\Epot_h^{n+1}$ computed with the
update procedure (either Algorithm~\ref{alg:yanenko} or \ref{alg:strang}
respectively) and $\tilde\Epot_h^{n+1}$ computed by solving
\eqref{eq:discrete_gauss} we introduce an instantaneous artifical
relaxation to the system in order to preserve a discrete energy balance
\eqref{TotalEnergyBalance}. In other words, we artificially lower the
kinetic energy of the system in order to balance (local) overshoots of the
potential energy introduced by the restart procedure at the end of each
time step. In this context it is easiest to reuse
Algorithm~\ref{alg:damping}, but instead of using a physical relaxation
time $\chartime$, we invoke the algorithm by passing an appropriate,
numerical relaxation time $\chartime_P$. To this end let
$\breve{\mathcal{T}}$ be a partition of $\Omega$ consisting of patches
$P\in\breve{\mathcal{T}}$ with $P\subset\mathcal{T}_h$. In the numerical
tests reported in Section~\ref{sec:numerics} we use $\breve{\mathcal{T}}
=\{\Omega\}$, as well as $\breve{\mathcal{T}} =\mathcal{T}_h$. We then
proceed as follows:

\begin{itemize}
  \item[1.] Perform the hyperbolic and source update by using
    Algorithms~\ref{alg:yanenko}, or~\ref{alg:strang}.
  \item[2.] Compute $\tilde\varphi^{n+1}_h$ given by
    \eqref{eq:discrete_gauss} and compute damping parameters:
    For each patch $P\in\breve{\mathcal{T}}$, compute the quantities
    \begin{align*}
      \delta E_P \;&:=\; \frac{1}{2\,\alpha} \Big(
      \big\|\nabla\tilde\Epot_h^{n+1}\big\|_{\bv L^2(P)}^2 -
      \big\|\nabla\Epot_h^{n+1}\big\|_{\bv L^2(P)}^2 \Big)
      \\
      E_{\text{k},P} \;&:=\; \sum_{\text{``}i\in P\text{''}}
      m_i\frac{\big|\momVect_i^{n+1}\big|^2_{\ell^2}}{2\,\varrho_i^{n+1}},
      \\\text{and define}\qquad
      \lambda^{n+1}_P \;&:=\; - \frac{1}{2\,\dt_n}\ln\,\Big[1-\Big(\frac{\delta
      E_P}{\,E_{\text{k},P}}\Big)_+\Big]\qquad\text{for each }
      P\in\breve{\mathcal{T}}.
    \end{align*}
  \item[3.]
    Perform the relaxation update and total mechanical energy update using
    Algorithm~\ref{alg:damping}, that is, we set
    \begin{align*}
      \{ \widetilde{\rho}_h^{n+1},
      \widetilde{\mom}_h^{n+1},
      \widetilde{\totme}_i^{n+1},
      \widetilde{\Epot}_h^{n+1}\}
      \leftarrow
      \texttt{relaxation\_update}
      (\{\rho_h^{n+1}, \mom_h^{n+1},
      \totme_h^{n+1},\Epot_h^{n+1}\}, \dt_n, \chartime_P)
    \end{align*}
    with an artificial \emph{local relaxation time} of $\chartime_P :=
    \min\big(\chartime\,,\, \big(\lambda^{n+1}_P\big)^{-1}\big)$ on each
    $P\in\breve{\mathcal{T}}$.
\end{itemize}
We note that the application of above algorithm is actually independent of
the chosen $\dt_n$: The factor $\dt_n$ in the definition of
$\lambda^{n+1}_P$ will cancel with the time-step in
Algorithm~\ref{alg:damping}.

\begin{lemma}[Balance of total energy]
  \label{lem:balance_total_energy_gauss}
  Let $\partial_t\rho_b\equiv0$. Provided that $\delta E_P <
  \,E_{\text{k},P}$ the update procedure described above preserves an
  energy inequality:
  \begin{align}
    \label{TotalEnergyBalanceGauss}
    \sum_{i \in \hypvertices} m_i\,\widetilde{\totme}_{i}^{n+1}
    +
    \frac{1}{2\,\alpha} \big\|\nabla\widetilde\Epot_h^{n+1}\big\|^2
    \leq
    \sum_{i \in \hypvertices} m_i\,\totme_{i}^n
    + \frac{1}{2\,\alpha} \big\|\nabla\Epot_h^{n}\big\|^2
  \end{align}
\end{lemma}

\begin{proof}
  The relaxation step of Algorithm~\ref{alg:damping}, implies that
  \begin{align*}
    \widetilde{\momVect}_i^{n+1} \;&=\; e^{-\dt_n/\chartime_P} \momVect_i^{n+1}
    \quad\text{for }i\in P.
  \end{align*}
  Squaring this expression, multiplying by $m_i /\varrho_i^{n+1}$ and
  summing over all indices $i$ yields a kinetic energy balance:
  \begin{align}
      \label{RelaxProof1}
      \sum_{i\in\hypvertices}
      m_i \frac{\big|\widetilde{\momVect}_i^{n+1}\big|^2_{\ell^2}}{
      2\,\varrho_i^{n+1}}
      \;&=\;
      e^{-2\dt_n/\chartime_P}\,
      \sum_{i\in\hypvertices}
      m_i\frac{\big|\momVect_i^{n+1}\big|^2_{\ell^2}}{2\,\varrho_i^{n+1}}
      \\\notag
      \;&=\;
      \sum_{P\in\breve{\mathcal{T}}}
      \Big\{\Big(1-\Big(\frac{\delta E_P}{E_{\text{k},P}}\Big)_+ \Big)
      \underbrace{\sum_{i\in P}
      m_i\frac{\big|\momVect_i^{n+1}\big|^2_{\ell^2}}{2\,\varrho_i^{n+1}}
      \Big\}}_{=\;E_{\text{k},P}}
      \\\notag
      \;&\le\;
      \sum_{i\in\hypvertices}
      m_i\frac{\big|\momVect_i^{n+1}\big|^2_{\ell^2}}{2\,\varrho_i^{n+1}}
      \;-\; \sum_{P\in\breve{\mathcal{T}}}\delta E_P
      \\\notag
      \;&=\;
      \sum_{i\in\hypvertices}
      m_i\frac{\big|\momVect_i^{n+1}\big|^2_{\ell^2}}{2\,\varrho_i^{n+1}}
      \;-\;
      \frac{1}{2\,\alpha}
      \big\|\nabla\widetilde\Epot_h^{n+1}\big\|^2 +
      \frac{1}{2\,\alpha}
      \big\|\nabla\Epot_h^{n+1}\big\|^2
  \end{align}
  Now, since we are using formula \eqref{eq:energy_update_proof_2_second} for
  the total mechanical energy update we have that (see Remark
\ref{RemarkEupdate})
\begin{align}\label{RelaxProof2}
\sum_{i \in \hypvertices} m_i \Big(
\widetilde{\totme}_i^{n+1} -
\frac{\big|\widetilde{\momVect}_i^{n+1}\big|_{\ell^2}^2}{2
\varrho_i^{n+1}} \Big) =
\sum_{i \in \hypvertices} m_i
\Big( \totme_i^{n+1} - \frac{\big|\momVect_i^{n+1}\big|_{\ell^2}^2}{2
\varrho_i^{n + 1}}  \Big)
\end{align}
Finally, energy inequality \eqref{TotalEnergyBalanceGauss} now follows by adding
\eqref{RelaxProof1} and \eqref{RelaxProof2}.
\end{proof}

\begin{remark}
  The algorithm can be modified to preserve the energy balance exactly by
  dropping the positive part from the definition of $\lambda^{n+1}_P$,
  \begin{align*}
    \lambda^{n+1}_P \;&:=\; - \frac{1}{2\,\dt_n}\ln\,\Big[1-\frac{\delta
    E_P}{\,E_{\text{k},P}}\Big].
  \end{align*}
  In this case an exact equality is maintained in
  \eqref{TotalEnergyBalanceGauss}. This is particularly desirable for
  maintaining the Lagrangian structure in the attractive formulation
  ($\alpha<0$). With either definition of $\lambda^{n+1}_P$, the internal energy
  will remain invariant because of the properties satisfied by
  \eqref{eq:energy_update_proof_2_second}, see Remark \ref{RemarkEupdate}.
\end{remark}

\begin{remark}
An alternative approach to the artificial relaxation discussed in this
section is to perform a line search blending $\Epot_h^{n+1}$ computed
with the update procedure (Algorithm~\ref{alg:yanenko}, or
\ref{alg:strang} respectively) and $\tilde\Epot_h^{n+1}$ computed by
solving \eqref{eq:discrete_gauss} such that the final update is as close
to the restarted potential $\tilde\Epot_h^{n+1}$ as possible while
maintaining the energy inequality. We summarize this approach in some
detail in Appendix~\ref{app:line_search}.
\end{remark}

%%%%%%%%%%%%%%%%%%%%%%%%%%%%%%%%%%%%%%%%%%%%%%%%%%%%%%%%%%%%%%%%%%%%%%%%%%%%%%%%
%%%%%%%%%%%%%%%%%%%%%%%%%%%%%%%%%%%%%%%%%%%%%%%%%%%%%%%%%%%%%%%%%%%%%%%%%%%%%%%%
%%%%%%%%%%%%%%%%%%%%%%%%%%%%%%%%%%%%%%%%%%%%%%%%%%%%%%%%%%%%%%%%%%%%%%%%%%%%%%%%

\section{Numerical illustrations}
\label{sec:numerics}

We now present a number of numerical computations demonstrating convergence
for the case of smooth solutions (Section~\ref{subse:num:convergence}),
cold plasma oscillation with contact-like discontinuity
(Section~\ref{subse:num:oscillation}), and a numerical simulation of an
electrostatic implosion (Section~\ref{subse:num:implosion}).

The numerical algorithms discussed in Sections~\ref{sec:numerical_approach}
and \ref{sec:gauss} have been implemented using the finite element library
\texttt{deal.II} \citep{dealII93,dealIIcanonical} using mapped
$\mathbb{Q}^1(\widehat{\element})$ finite elements as defined in
\eqref{QuadSpaces}.

%%%%%%%%%%%%%%%%%%%%%%%%%%%%%%%%%%%%%%%%%%%%%%%%%%%%%%%%%%%%%%%%%%%%%%%%%%%%%%%%
\subsection{Grid convergence on quadrilateral meshes}
\label{subse:num:convergence}

A finite element is said to be \emph{distorted} if it cannot be mapped to
the reference element by an affine diffeomorphism. Establishing optimal
convergence rates for finite element approximations on distorted
quadrilaterals and hexahedrons is a delicate issue. Analytical convergence
results for distorted meshes often hinge on the assumption that the
constructed mesh sequence is \emph{asymptotically affine}, meaning that the
mesh distortion of the mesh sequence converges to zero when measured in a
suitable metric; see for example \cite{Arnold2002, Botti2012, Chan2017}.
These results are not just hypothetical: violating the asymptotically
affine property often leads to a reduced convergence order
\cite{Arnold2002}. In our case, the necessity to introduce inexact
quadrature in order to preserve energy stability of the source update
further complicates matters and might be a source of suboptimal convergence
rates. In particular we introduced inexact nodal-point quadrature in
various bilinear forms; see \eqref{LumpingDef}, \eqref{ProtoSemiDiscrete}
and \eqref{BilinearFormQuad}.

In order to assess the dependence of the convergence rate of our scheme on
mesh distortion, quadrature, and Gauß law restart we consider two
different procedures of mesh generation:
\begin{enumerate}
  \item[(i)]
    Asympotically affine, nested mesh sequence: We first create a base coarse
    mesh where each element can be affinely mapped to the reference element
    and apply 5\% random noise to the location of the vertices of the coarse
    mesh. We then use a uniform \emph{bisection} refinement strategy. Upon
    refinement, no further distortion is added to the element shapes,
  \item[(ii)]
    Non affine, non-nested mesh sequence: In this case, for each refinement
    level, we produce an affine mesh and apply random noise to the node
    coordinates. The resulting mesh sequence is neither affine, nor
    asymptotically affine, nor nested.
\end{enumerate}
and three different strategies for enforcing the Gauß law: (a) no Gauß law
restart, (b) Gauß law restart with artificial relaxation
(Section~\ref{subse:relax}), and (c) full Gauß law restart
(Section~\ref{subse:full_restart}).

\begin{remark}
  \label{rem:it_is_fine_for_triangles}
  The case of affinely mapped $\mathbb{P}^1$ simplices is simpler. Here,
  our scheme is guaranteed to deliver second order convergence rates. See
  for instance \cite[\S 33.3]{GuermondErnVolII} for a related discussion of
  recovering optimal convergence rates with lumping by using one quadrature
  point per element at the barycenter (or one quadrature point per vertex).
\end{remark}

We manufacture an analytic solution for the Euler-Poisson equations by
starting with the isentropic vortex
(Definition~\ref{defi:isentropic_vortex}), which is an exact solution of
the Euler equations, and adding a background density
$\rho_{\text{a}}(\xcoord, t) = -\rho(\xcoord,t)$, where $\rho(\xcoord,t)$
is the density field of the isentropic vortex; see
Section~\ref{subse:background_density} for algorithmic details. We enforce
inhomogeneous Dirichlet boundary conditions on the state
$\state:=[\den,\mom,\totme]^T$ of the Euler subsystem enforcing the exact
solution throughout. Similarly, we enforce homogeneous Dirichlet boundary
conditions for the potential during the source update. The initial data of
the simulation is given by (an interpolation of) the exact solution of the
Euler subsystem with the potential set to zero. We set $\alpha = 1$, and
the final time is $\tf = 2$.
\begin{definition}[Isentropic vortex \cite{YeeSand1999}]
  \label{defi:isentropic_vortex}
  For given parameters $\xcoord_0$, $M$, $\gamma$, $\beta$ defined below we
  introduce functions
  \begin{align*}
    \boldsymbol{r}(\xcoord, t) &:= \xcoord - \xcoord_0 - M t,
    \quad
    f(\xcoord, t) := \tfrac{\beta}{2 \pi} e^{\frac{1}{2} (1 -
    |\boldsymbol{r}|^2)},
    \quad
    T(\xcoord,t) := 1 - \tfrac{\gamma - 1}{2 \gamma} f^2,
  \end{align*}
  depending on position $\xcoord \in \mathbb{R}^2$ and time $t \in
  \mathbb{R}^+$. Then, the state $u(\xcoord, t)$ defined by the primitive
  quantities (density, velocity, pressure),
  \begin{align*}
    \rho(\xcoord,t) := T^{\tfrac{1}{\gamma - 1}},
    \quad
    u(\xcoord,t) := M - f \boldsymbol{r}_2,
    \quad
    v(\xcoord,t) := M + f \boldsymbol{r}_1,
    \quad
    p(\xcoord,t) := \rho^{\gamma},
  \end{align*}
  is a solution of the compressible Euler equations. We use the following
  choice of free parameters throughout: $\xcoord_0 = [4,4]^\transp$,
  $\gamma = 5/3$, $M = 2$, $\beta = 5$, in the computational domain $\Omega
  = [-5,15] \times [-5,15]$.
\end{definition}

\begin{table}[p!]
  \caption{\label{tab:convergence_affine}%
    Convergence rates for the second order scheme
    Algorithm~\ref{alg:strang} with semi-lumped source update
    \eqref{ProtoSemiDiscrete} for
    case (i) the asympotically affine, nested
    mesh sequence. Error norms $\delta_{\text{euler},h}$ and
    $\delta_{\text{pot},h}$,(see \eqref{eq:error_norms}, for three
    different restart strategies: (a) no restart, (b) relaxation, (c) full
    restart.}
  \begin{center}
    \begin{tabular}{rclclclcl}
      \toprule
      & \multicolumn{4}{c}{\bfseries (a) no restart}
      & \multicolumn{4}{c}{\bfseries (b) relaxation}\\
      \cmidrule(lr){2-5}
      \cmidrule(lr){6-9}
      & $\delta_{\text{euler},h}$ & rate & $\delta_{\text{pot},h}$ & rate
      & $\delta_{\text{euler},h}$ & rate & $\delta_{\text{pot},h}$ & rate
      \\[0.3em]
      1 & 1.561e-01 & --   & 2.723e-02 & --   & 7.844e-02  & --   & 1.071e-02 & --   \\
      2 & 6.827e-02 & 1.19 & 8.622e-03 & 1.66 & 3.912e-02  & 1.00 & 5.280e-03 & 1.02 \\
      3 & 2.133e-02 & 1.68 & 2.370e-03 & 1.86 & 1.372e-02  & 1.51 & 1.747e-03 & 1.59 \\
      4 & 6.152e-03 & 1.79 & 6.180e-04 & 1.94 & 4.086e-03  & 1.75 & 4.781e-04 & 1.87 \\
      5 & 1.637e-03 & 1.91 & 1.561e-04 & 1.99 & 1.087e-03  & 1.91 & 1.210e-04 & 1.98 \\
      \bottomrule
    \end{tabular}
    \vspace{0.5em}

    \begin{tabular}{rclcl}
      \toprule
      & \multicolumn{4}{c}{\bfseries (c) full restart} \\
      \cmidrule(lr){2-5}
      & $\delta_{\text{euler},h}$ & rate & $\delta_{\text{pot},h}$ & rate
      \\[0.3em]
      1 & 7.844e-02 & --   & 1.071e-02 & --   \\
      2 & 3.912e-02 & 1.00 & 5.279e-03 & 1.02 \\
      3 & 1.372e-02 & 1.51 & 1.747e-03 & 1.60 \\
      4 & 4.086e-03 & 1.75 & 4.781e-04 & 1.87 \\
      5 & 1.086e-03 & 1.91 & 1.210e-04 & 1.98 \\
      \bottomrule
    \end{tabular}
  \end{center}

  \vspace{4em}

  \caption{\label{tab:convergence_non_affine}%
    Convergence rates for the second order scheme
    Algorithm~\ref{alg:strang} with fully lumped source update
    \eqref{ProtoSemiDiscrete} with modified bilinear form
    \eqref{BilinearFormQuad} for case (ii) the non-affine, unnested mesh
    sequence. Error norms $\delta_{\text{euler},h}$ and
    $\delta_{\text{pot},h}$,(see \eqref{eq:error_norms}, for three
    different restart strategies: (a) no restart, (b) relaxation, (c)
    full restart.}
  \begin{center}
    \begin{tabular}{rclclclcl}
      \toprule
      & \multicolumn{4}{c}{\bfseries (a) no restart}
      & \multicolumn{4}{c}{\bfseries (b) relaxation}\\
      \cmidrule(lr){2-5}
      \cmidrule(lr){6-9}
      & $\delta_{\text{euler},h}$ & rate & $\delta_{\text{pot},h}$ & rate
      & $\delta_{\text{euler},h}$ & rate & $\delta_{\text{pot},h}$ & rate
      \\[0.3em]
      1 & 1.561e-01 & --   & 2.723e-02 & --   & 7.844e-2 & --   & 1.071e-02 & --   \\
      2 & 1.004e-01 & 0.64 & 1.452e-02 & 0.91 & 3.955e-2 & 0.99 & 5.235e-03 & 1.03 \\
      3 & 5.167e-02 & 0.96 & 7.362e-03 & 0.98 & 1.386e-2 & 1.51 & 1.728e-03 & 1.60 \\
      4 & 2.622e-02 & 0.98 & 3.779e-03 & 0.96 & 4.080e-3 & 1.76 & 4.713e-04 & 1.87 \\
      5 & 1.469e-02 & 0.84 & 2.073e-03 & 0.87 & 1.074e-3 & 1.93 & 1.195e-04 & 1.98 \\
      \bottomrule
    \end{tabular}
    \vspace{0.5em}

    \begin{tabular}{rclcl}
      \toprule
      & \multicolumn{4}{c}{\bfseries (c) full restart} \\
      \cmidrule(lr){2-5}
      & $\delta_{\text{euler},h}$ & rate & $\delta_{\text{pot},h}$ & rate
      \\[0.3em]
      1 & 7.844e-02 & --   & 1.071e-02 & --   \\
      2 & 3.912e-02 & 1.00 & 5.279e-03 & 1.02 \\
      3 & 1.372e-02 & 1.51 & 1.747e-03 & 1.60 \\
      4 & 4.086e-03 & 1.75 & 4.781e-04 & 1.87 \\
      5 & 1.086e-03 & 1.91 & 1.210e-04 & 1.98 \\
      \bottomrule
    \end{tabular}
  \end{center}
\end{table}
The computational results summarized in Tables~\ref{tab:convergence_affine}
and~\ref{tab:convergence_non_affine} were computed using the a mesh
sequence of $20 \times20$, $40 \times 40$, $80 \times 80$, $160 \times
160$, and $320 \times 320$ elements and $\text{cfl} = 0.5$. Two separate
error norms are reported:%
\footnote{We observe essentially the same convergence rates in density,
  momentum, and total energy. We have thus chosen to consolidate the
  individal error components of density, momentum, and total energy into a
  combined quantity $\delta_{\text{euler},h}$.}
\begin{subequations}
  \label{eq:error_norms}
  \begin{align}
    \delta_{\text{euler},h} \;&:=\;
    \|\den - \den_h\|_{L^\infty(L^1)} \,+\, \|\mom - \mom_{h}\|_{L^\infty(L^1)}
    \,+\, \|\totme - \totme_h\|_{L^\infty(L^1)},
    \\[0.2em]
    \delta_{\text{pot},h} \;&:=\; \|\Epot - \Epot_h\|_{L^\infty(L^2)}.
  \end{align}
\end{subequations}
We observe that we recover second order convergence for all three choices
of restart, strategies (a) -- (c), for case (i), the asymptotically affine
mesh sequence; cf. Table~\ref{tab:convergence_affine}. Both restart
approaches (strategies (b) and (c)) reduce the error constant compared to no
restart (strategy (a)) whereas the difference in error rates between the
two restart techniques is negligible. The situation changes for the non
affine, unnested mesh sequence of case (ii), cf.
Table~\ref{tab:convergence_non_affine}. Here, the case of no restart ,
strategy (a), shows a reduced convergence order. An enforcement of the Gauß
law (i.e. strategies (b) and (c)) is necessary to recover optimal
convergence rates.

From these results, we may conjecture that Gauß law restart is not
necessary in order to recover optimal rates of convergence for the case of
asymptotically affine mesh sequences. However, for the case of general
meshes, some form of restart of the Gauß law, e.g. strategies (b) and (c),
is necessary in order to recover optimal convergence rates.

%%%%%%%%%%%%%%%%%%%%%%%%%%%%%%%%%%%%%%%%%%%%%%%%%%%%%%%%%%%%%%%%%%%%%%%%%%%%%%%%
\subsection{Perturbed electron gas column: pure plasma oscillation}
\label{subse:num:oscillation}
\begin{figure}[t!]
  \begin{center}
    \setlength\fboxsep{0pt}
    \setlength\fboxrule{0.5pt}
    \subfloat[$t=0$]{\fbox{\includegraphics[width=40mm]{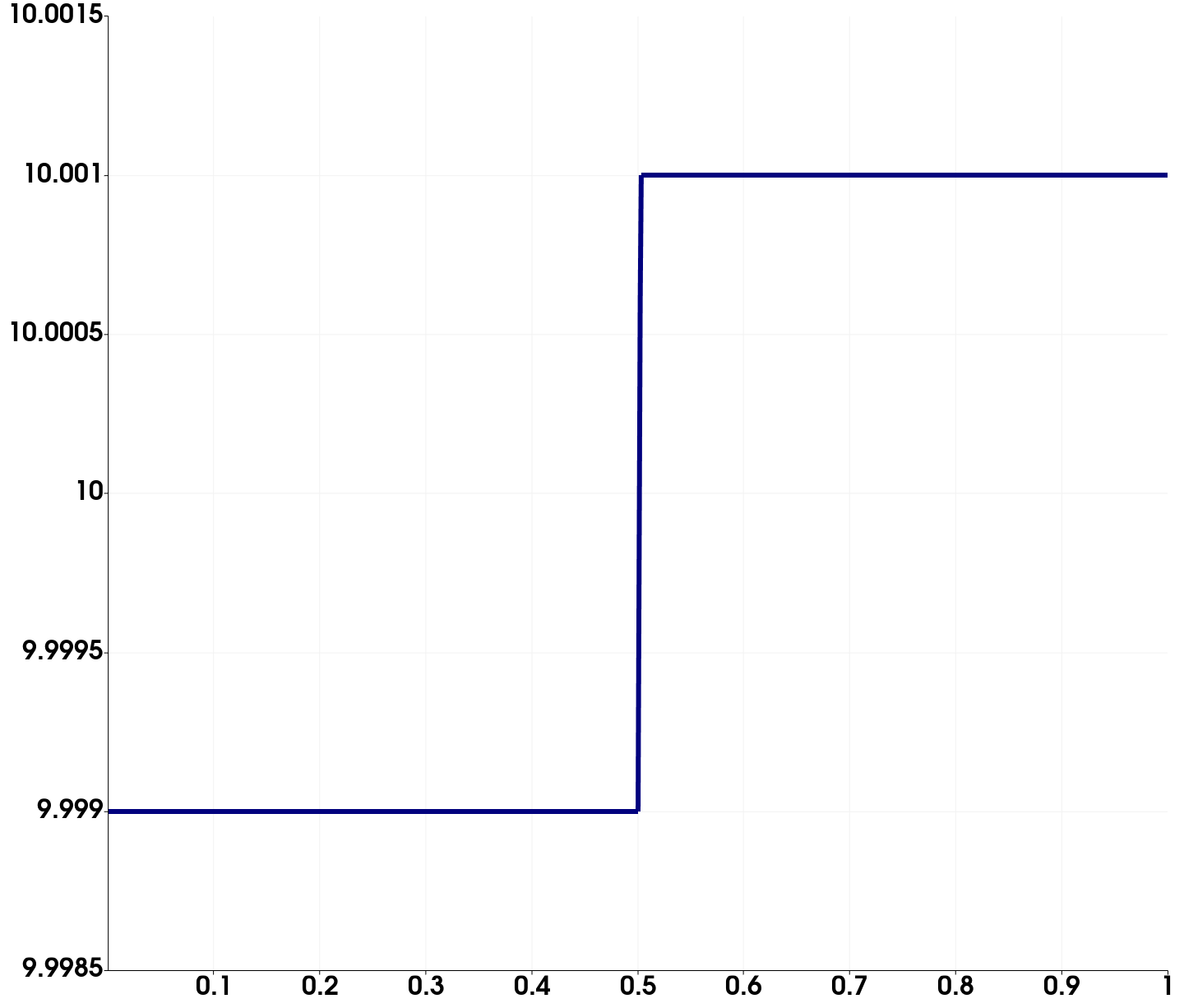}}}\;
    \subfloat[$t=1/8\,\tp$]{\fbox{\includegraphics[width=40mm]{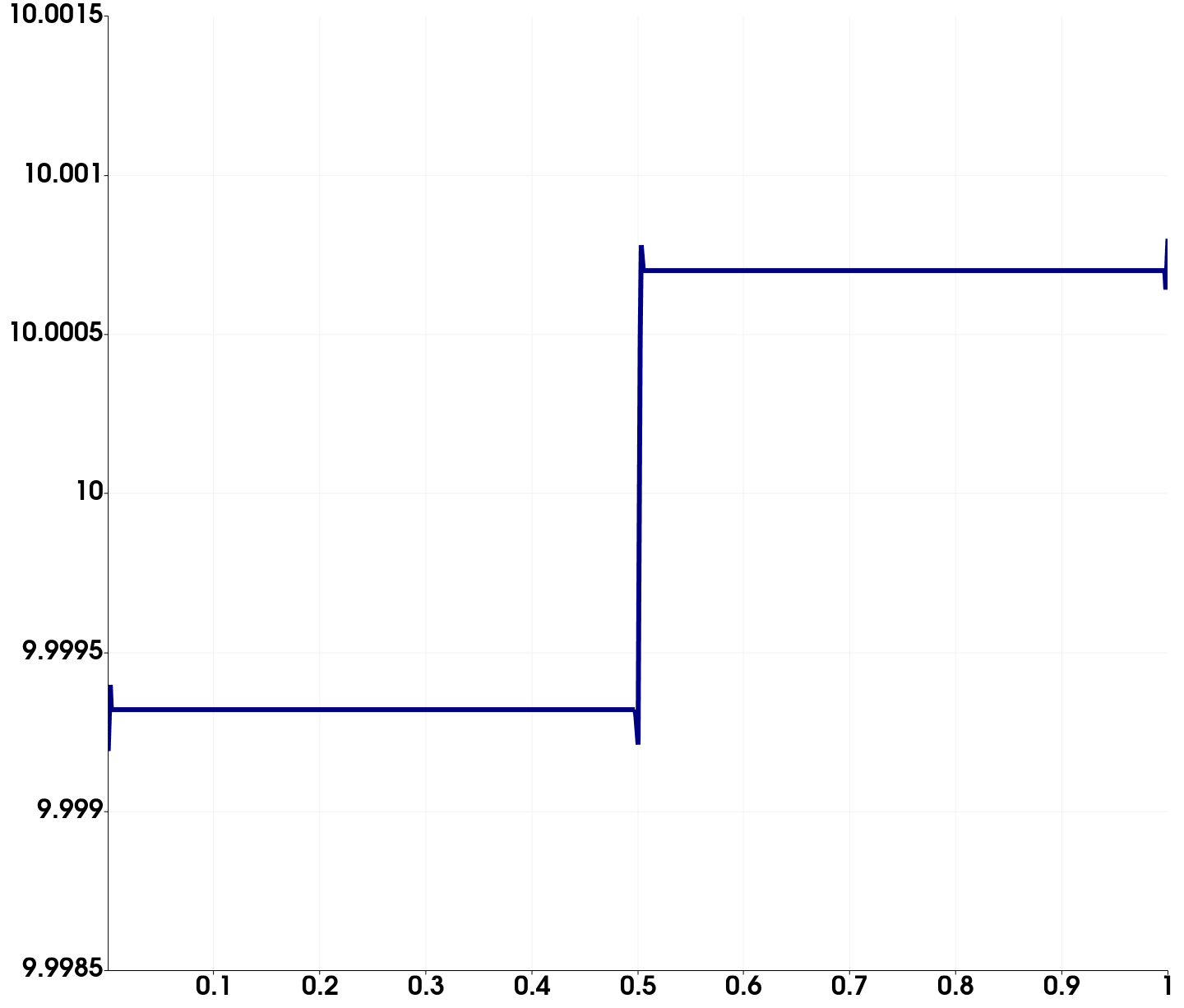}}}\;
    \subfloat[$t=2/8\,\tp$]{\fbox{\includegraphics[width=40mm]{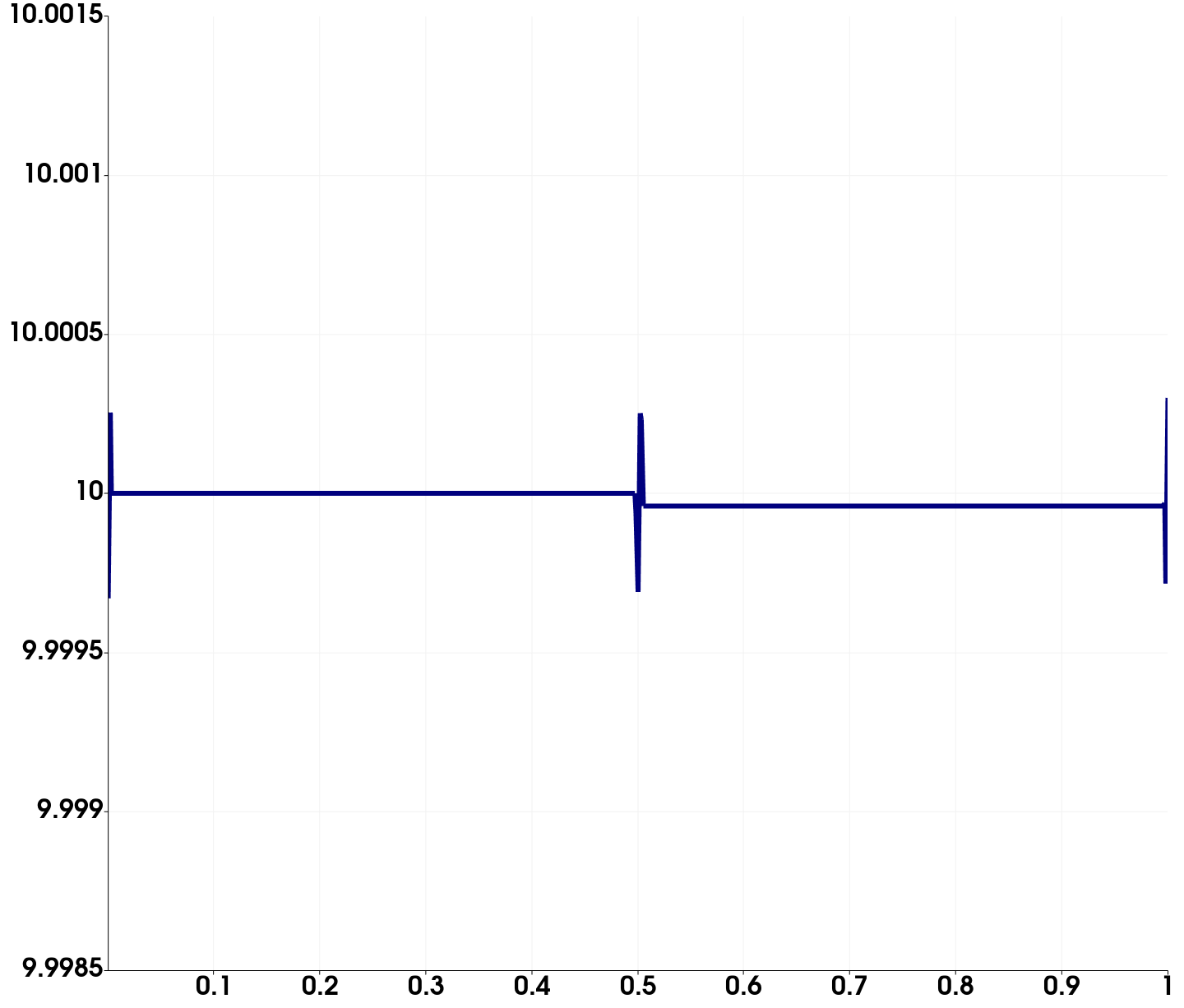}}}

    \subfloat[$t=3/8\,\tp$]{\fbox{\includegraphics[width=40mm]{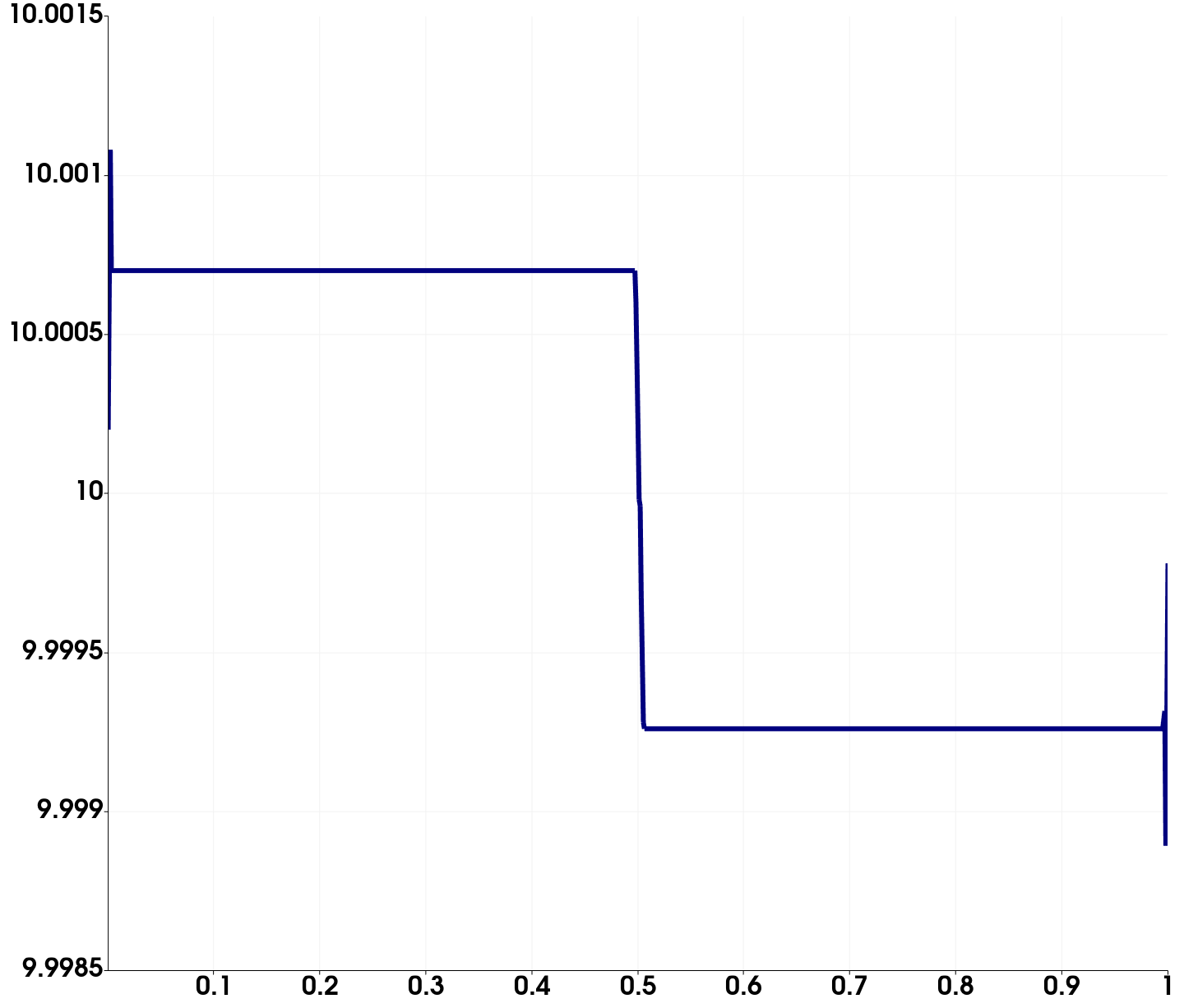}}}\;
    \subfloat[$t=4/8\,\tp$]{\fbox{\includegraphics[width=40mm]{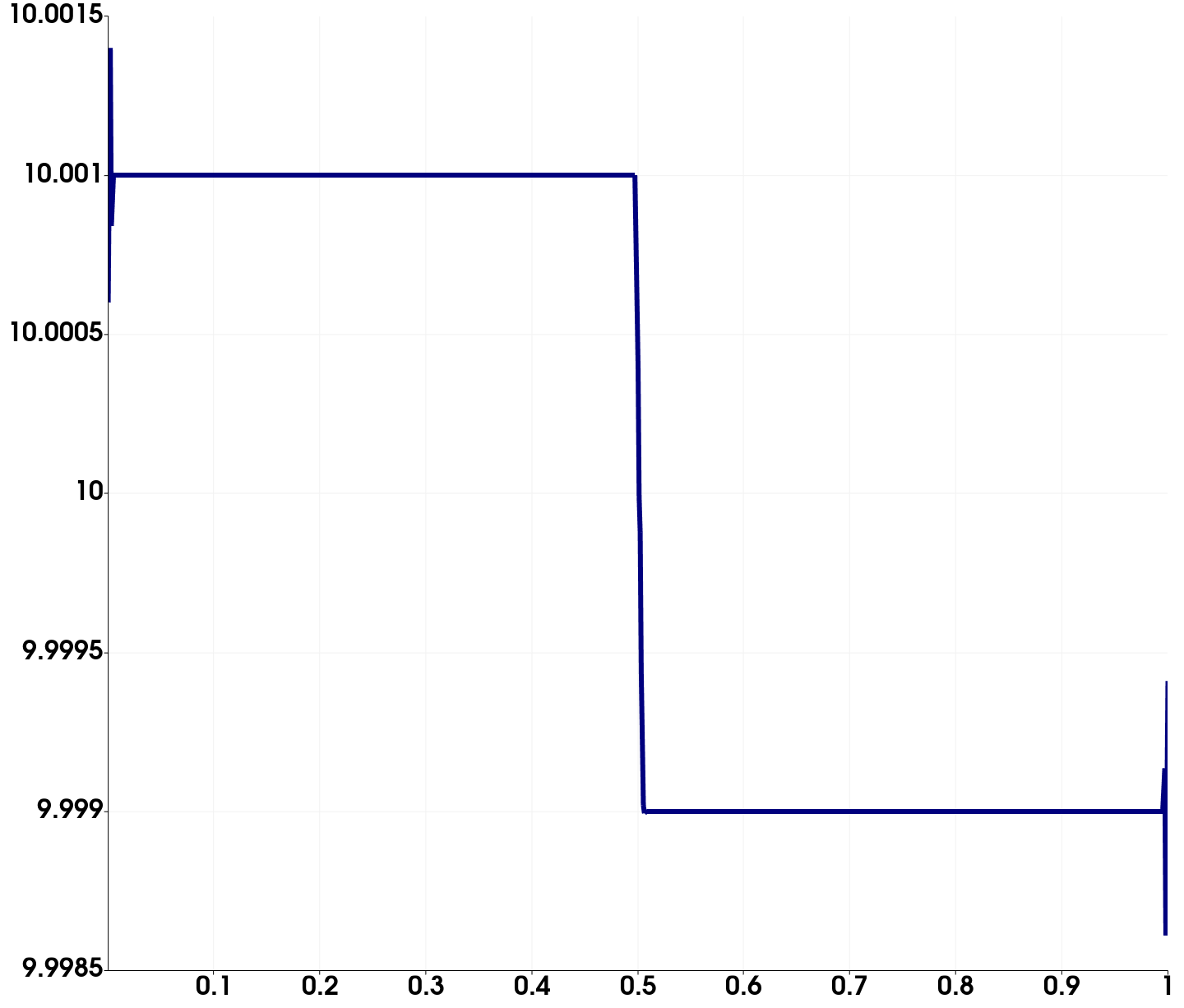}}}\;
    \subfloat[$t=5/8\,\tp$]{\fbox{\includegraphics[width=40mm]{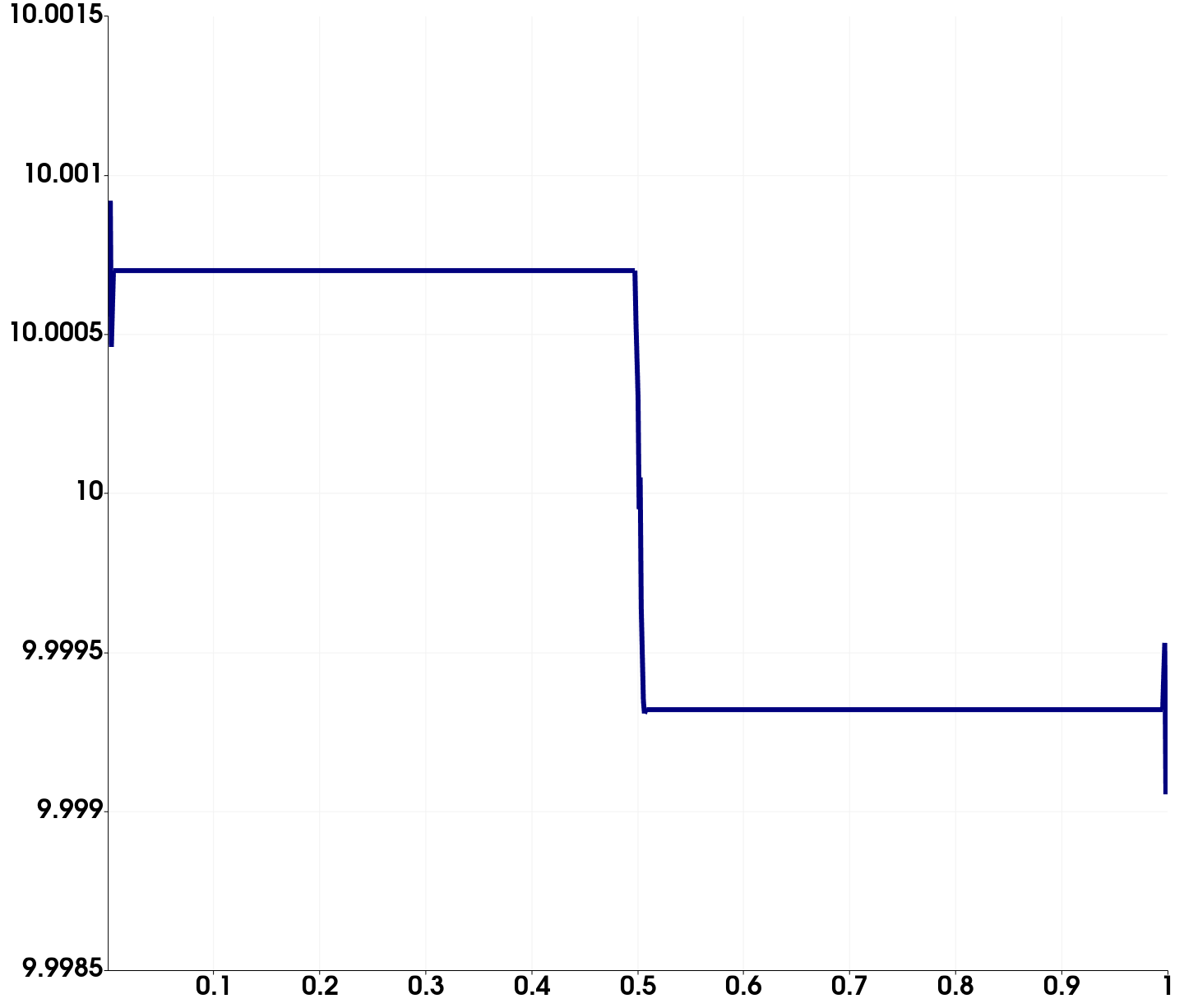}}}

    \subfloat[$t=6/8\,\tp$]{\fbox{\includegraphics[width=40mm]{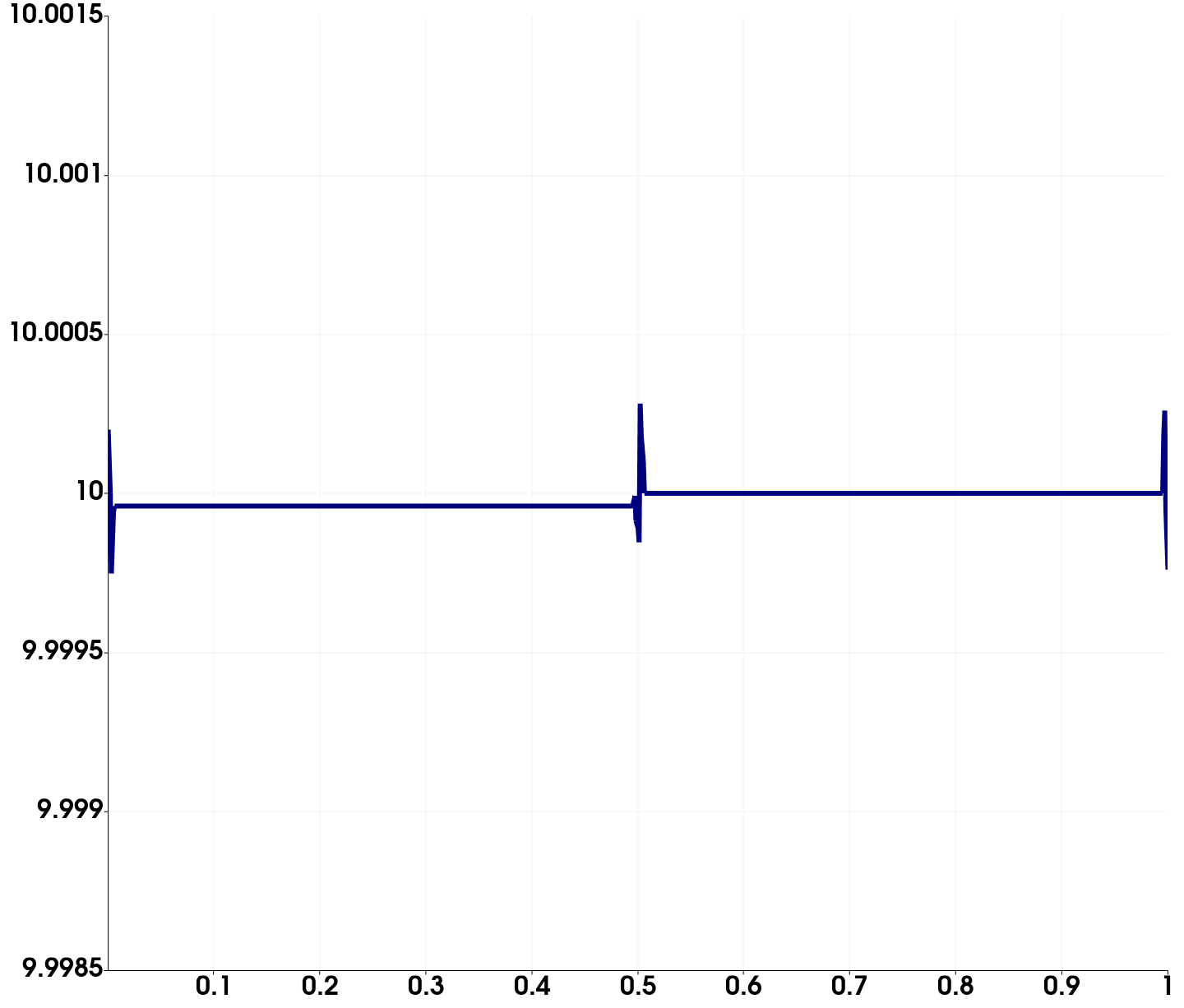}}}\;
    \subfloat[$t=7/8\,\tp$]{\fbox{\includegraphics[width=40mm]{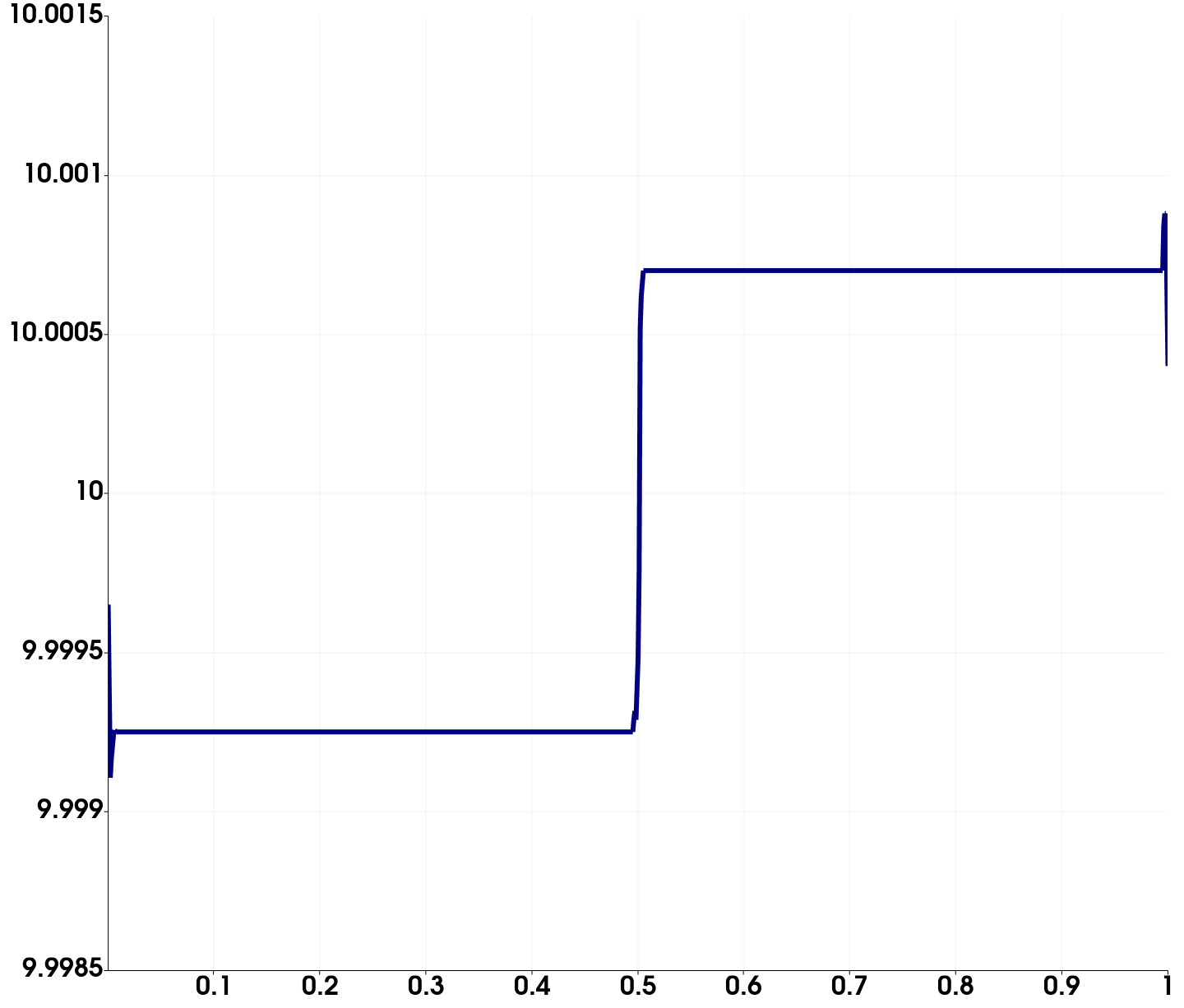}}}\;
    \subfloat[$t=\tp$]{\fbox{\includegraphics[width=40mm]{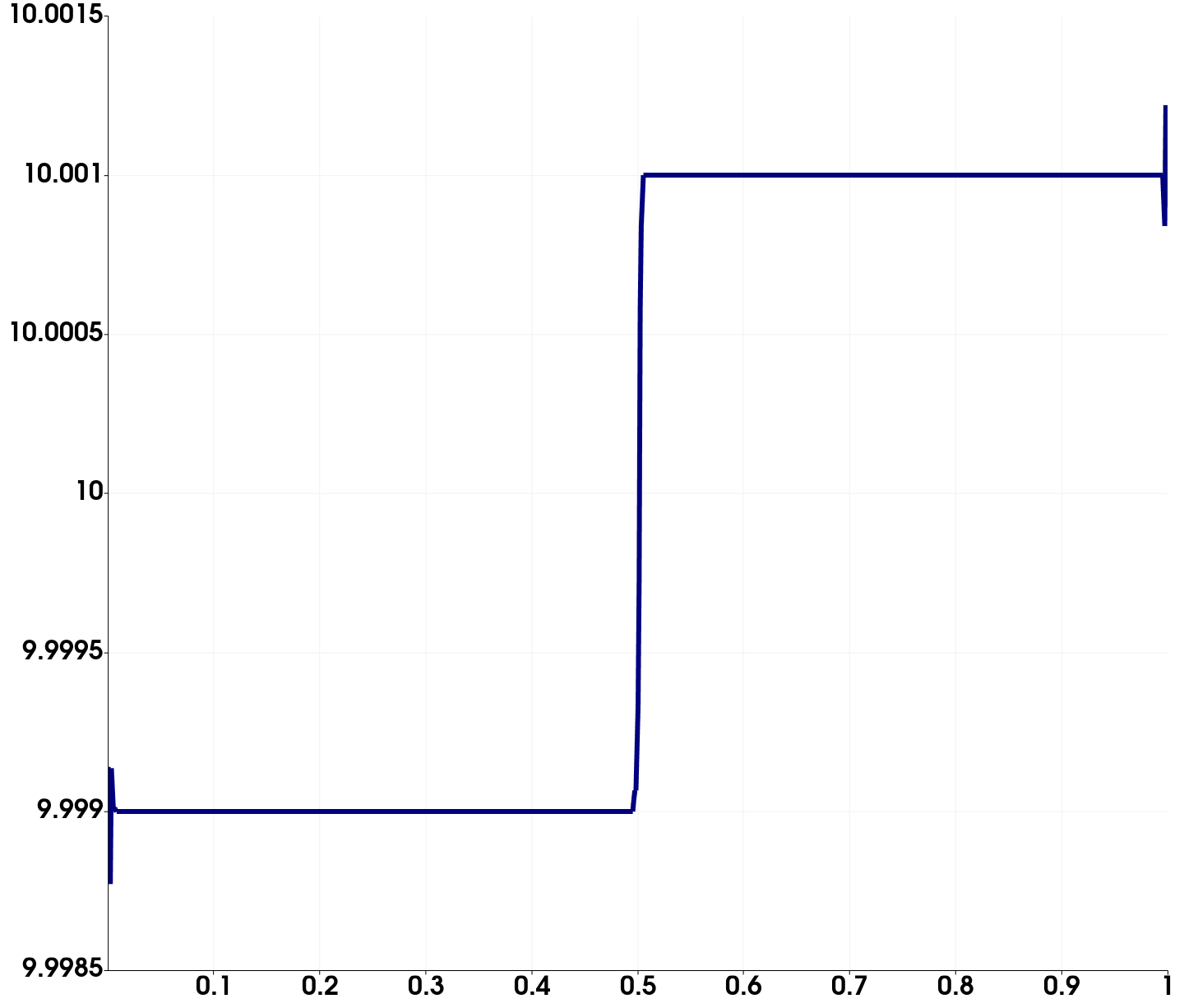}}}
  \end{center}
  \caption{\label{PlasmaOscResolved}%
    Temporal snapshots of the density profile of the pure plasma
    oscillation test case (Definition~\ref{defi:pure_plasma_oscillation})
    for a full period of plasma oscillation. The snapshots are taken
    approximately at times, from (a) to (i), $t = \tfrac{n-1}{8}\tp$,
    $n=1,\ldots,9$. The period of the first and final snapshot coincide.
    The dynamics of the system are almost purely Hamiltonian, meaning the
    time evolution is almost entirely dictated by
    \eqref{ProtoSemiDiscreteEpot}. The stationary contact at $x=0.5$
    remains very well preserved.}
\end{figure}
\begin{figure}[t!]
  \begin{center}
    \setlength\fboxsep{0pt}
    \setlength\fboxrule{0.5pt}
    \subfloat[no restart]{\begin{tabular}[b]{c}
        \fbox{\includegraphics[width=38mm]{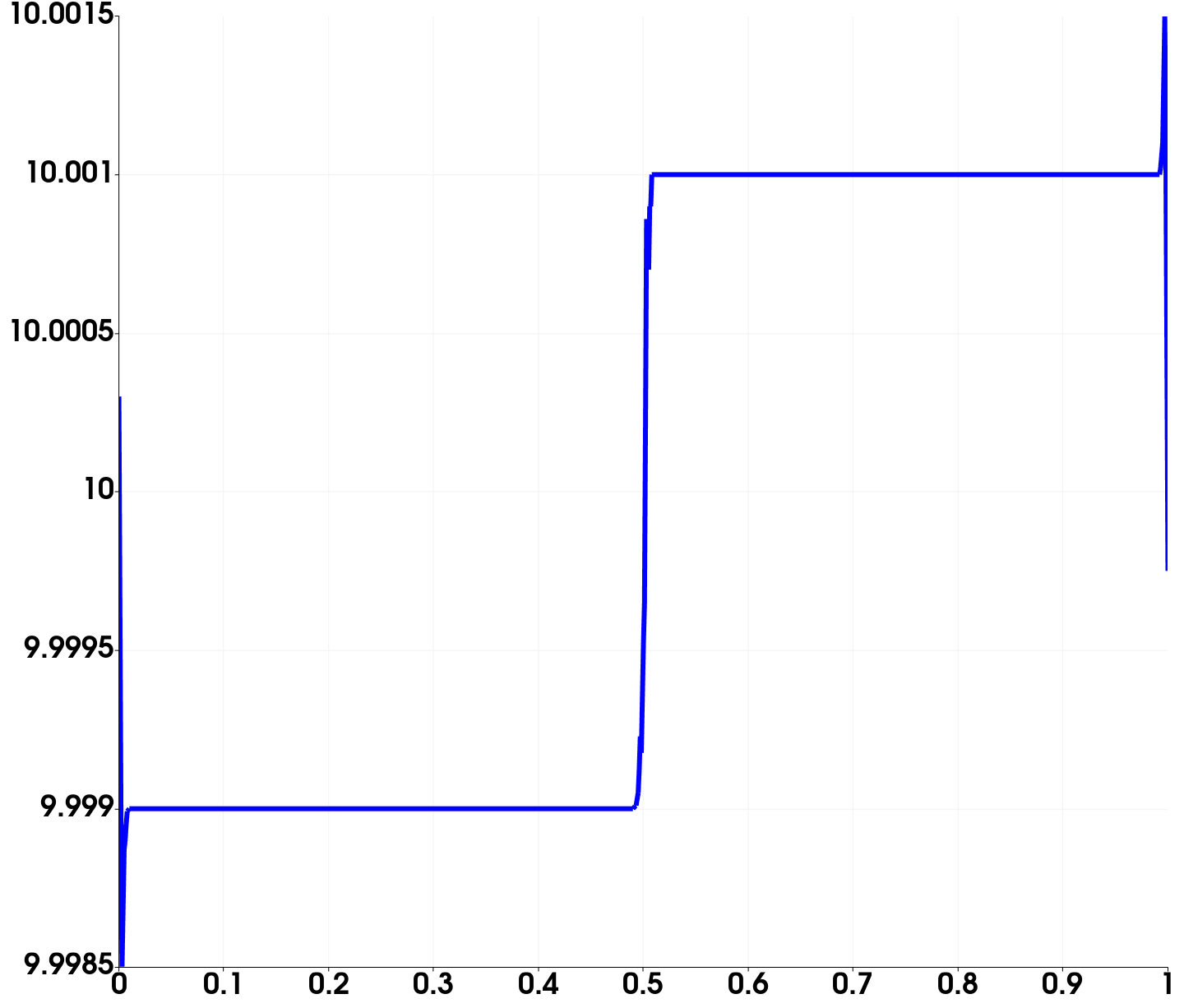}}\\
        \fbox{\includegraphics[width=38mm]{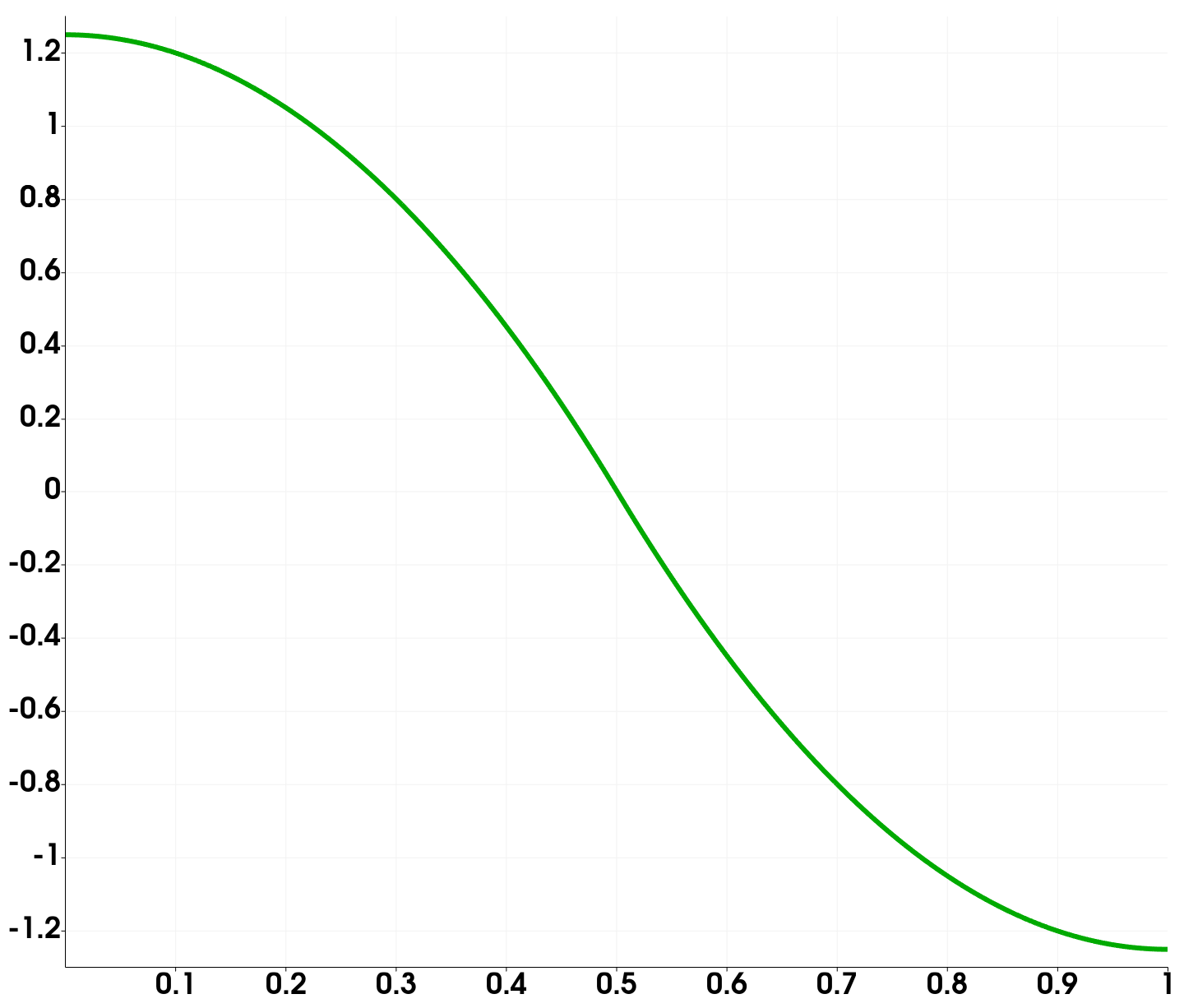}}
      \end{tabular}}
    \subfloat[relaxation]{\begin{tabular}[b]{c}
        \fbox{\includegraphics[width=38mm]{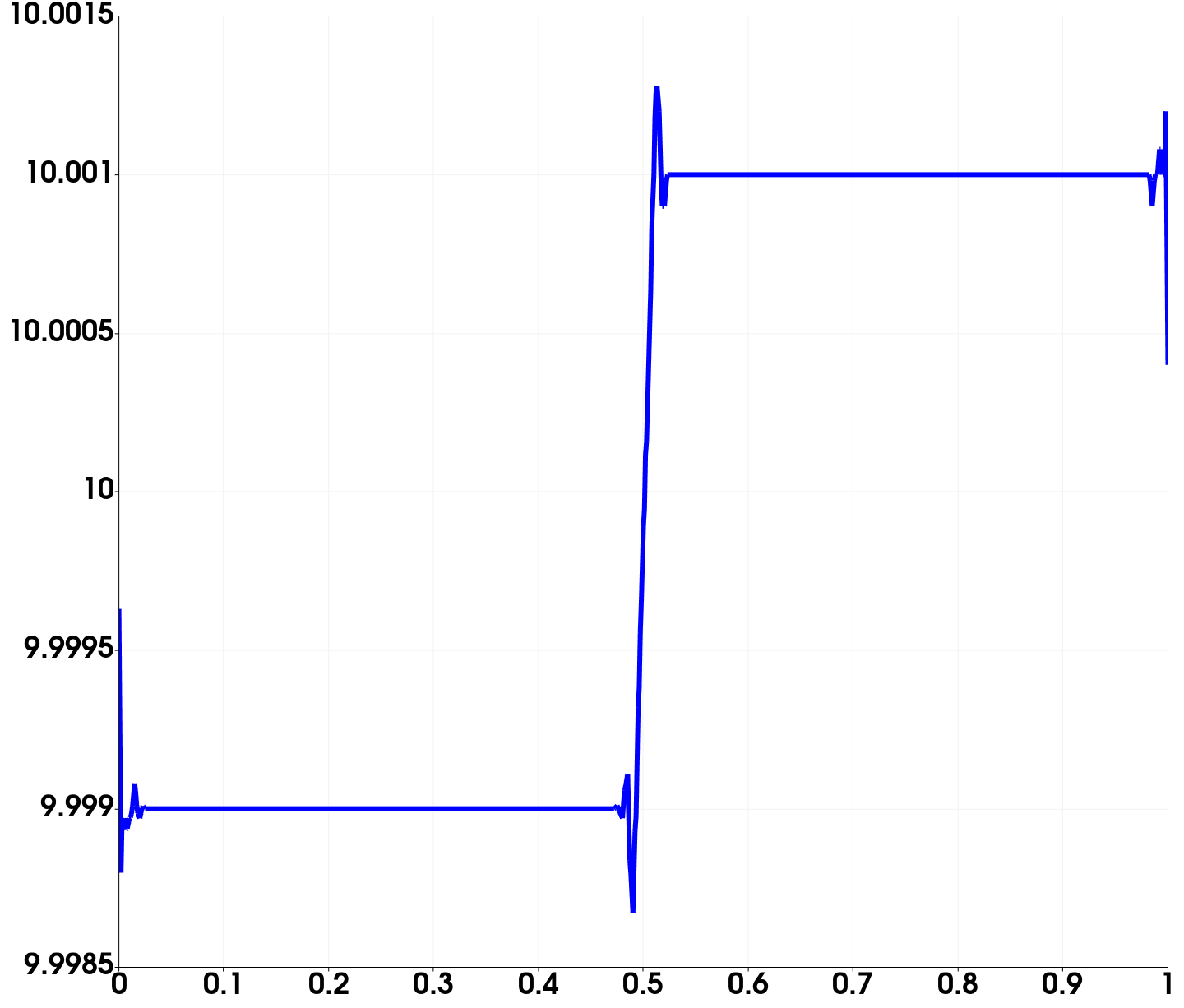}}\\
        \fbox{\includegraphics[width=38mm]{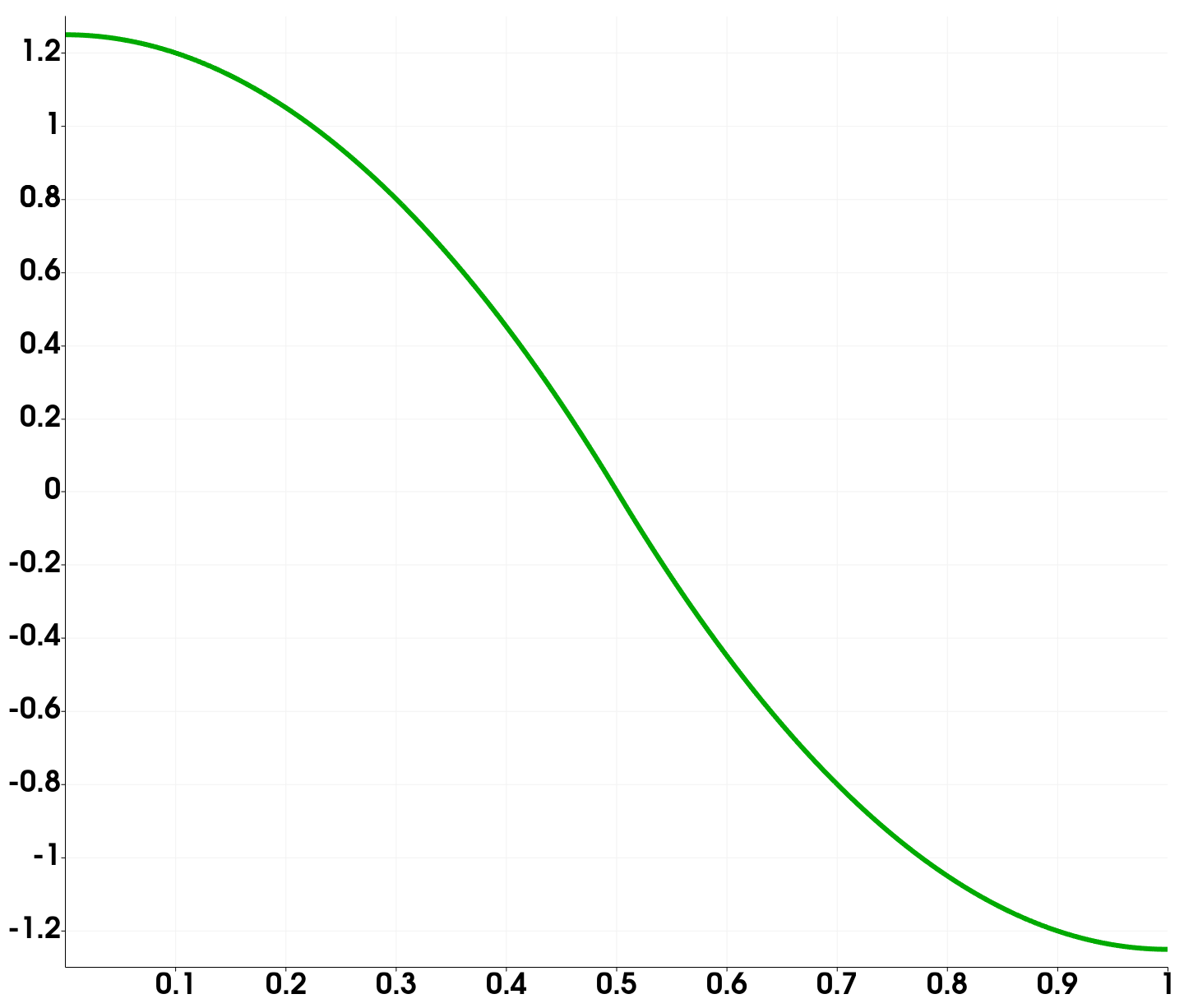}}
      \end{tabular}}
    \subfloat[full restart]{\begin{tabular}[b]{c}
        \fbox{\includegraphics[width=38mm]{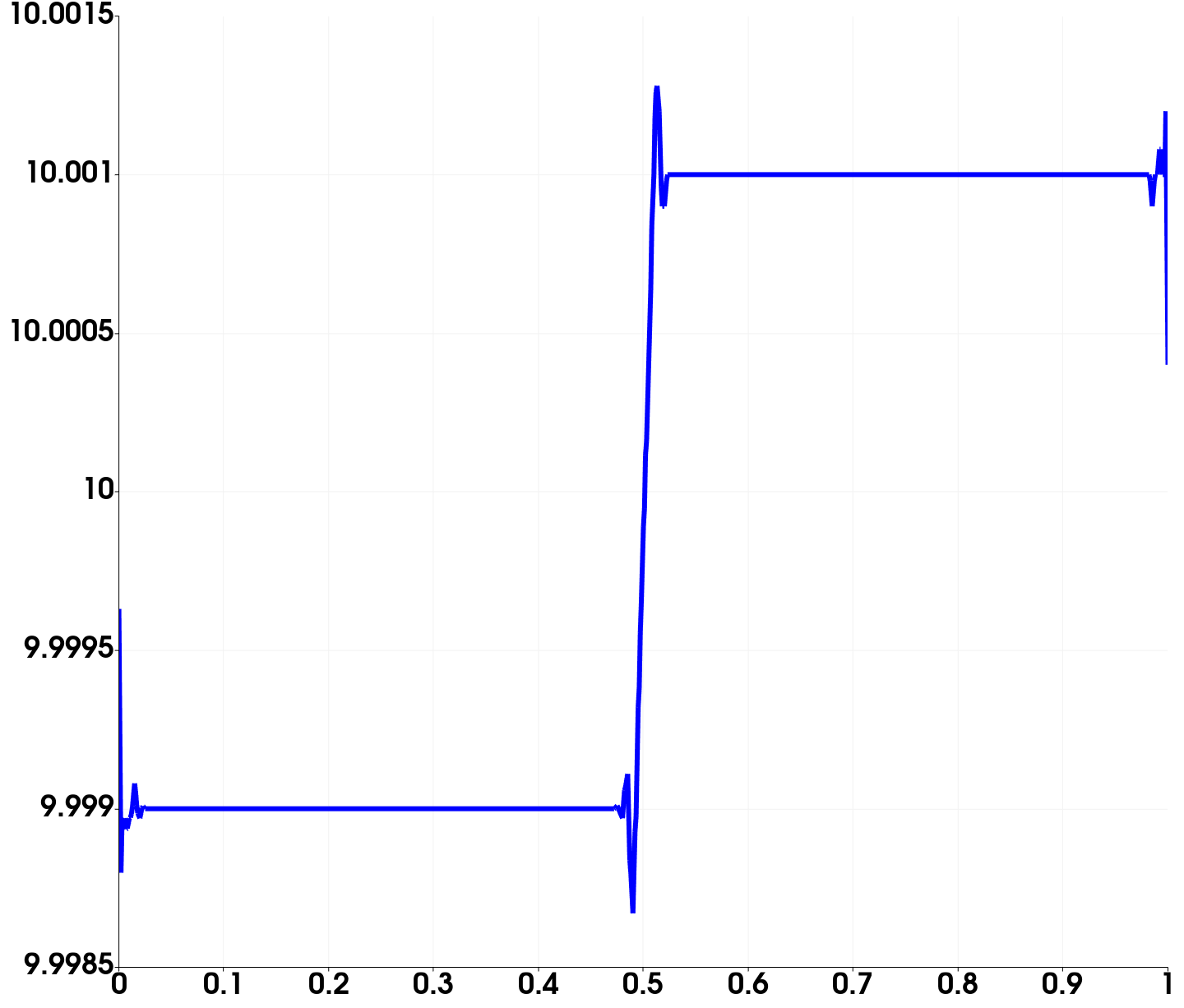}}\\
        \fbox{\includegraphics[width=38mm]{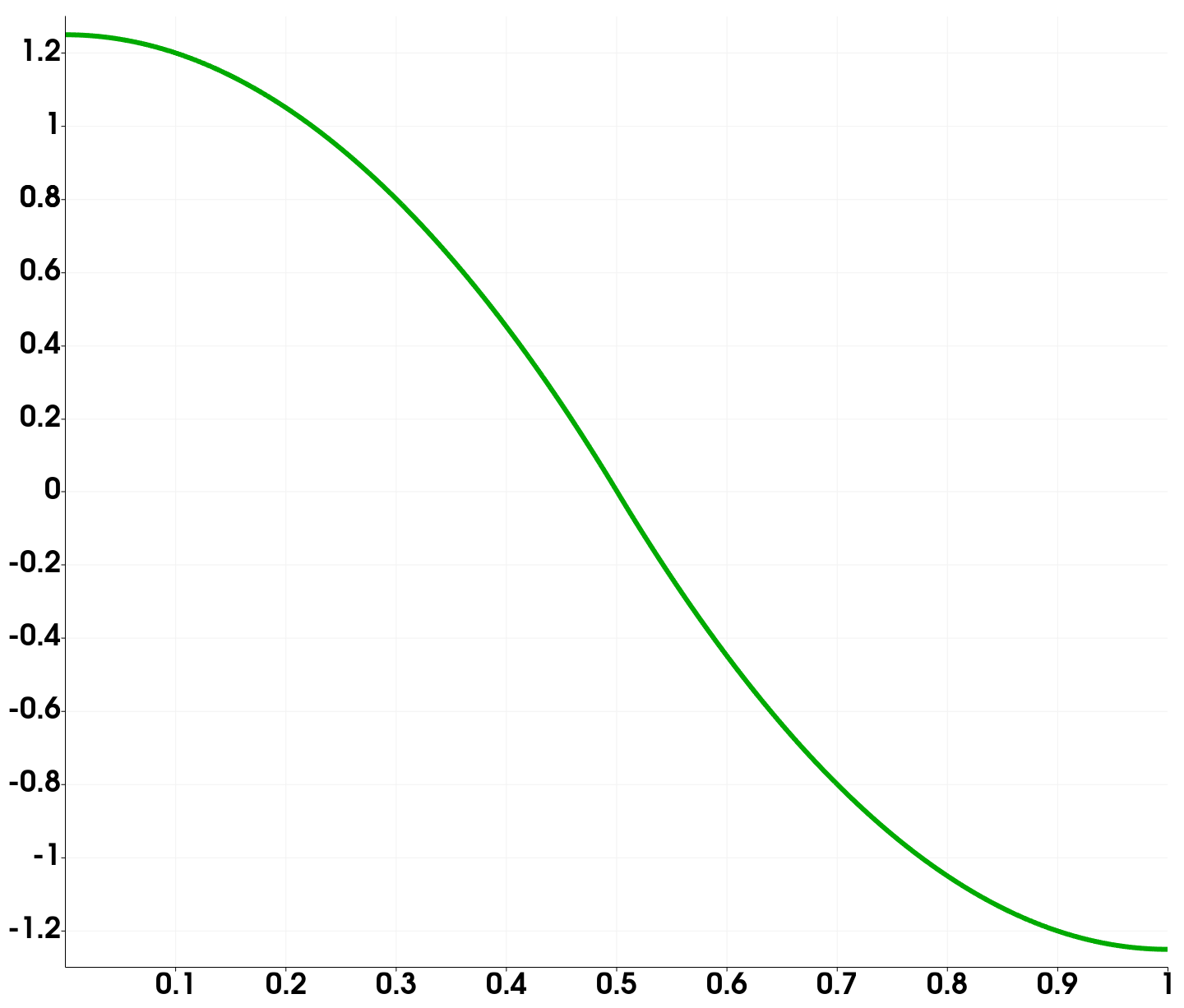}}
      \end{tabular}}
  \end{center}
  \caption{\label{PlasmaOscResolved2}
    Comparison of the effects of different choices of restart. The three
    top figures display the density at the final time $\tf = 5 \tp$, the
    three figures on the bottom show the potential at the final time.
    From left to right we use: (a) no restart, (b) relaxation, and (c)
    full restart. The difference between the full restart and relaxation
    restart is again minimal.}
\end{figure}

We now consider a numerical test case for capturing pure plasma
oscillations. The initial setup is given in Definition
\ref{defi:pure_plasma_oscillation}, and consists of a quasi one-dimensional
plasma with discontinuous density, no velocity, low pressure, and a
positive background charge density that neutralizes the mean value of the
electron charge in the domain. The purpose of prescribing a low pressure is
to ensure that pressure forces $-\nabla p$ are negligible in comparison to
the electric force $-\rho\nabla\Epot$.
\begin{definition}[Plasma column]
  \label{defi:pure_plasma_oscillation}
  Given the following rectangular domain and repulsive coupling constant
  $\alpha$,
  \begin{align*}
    \Omega = [0,1]\times[0,h] \ \ \text{with} \ \ h = 400^{-1},
    \qquad
    \alpha = 10^{4}\,>\,0,
  \end{align*}
  we introduce an initial state $u_0(\xcoord)$ defined by the primitive
  quantities (density, velocity, pressure),
  \begin{align}\label{plasmaOscInit}
    \rho_0 = \begin{cases}
      \overline{\rho} - \delta  & \text{if } x < 0.5,\\
      \overline{\rho} + \delta  & \text{if } x \ge 0.5,
    \end{cases}
    \quad
    \text{with }\overline{\rho} = 10.0 \text{ and } \delta = 0.001,\;
    \bv{v}_0 = \boldsymbol{0},\;
    p_0 = 0.01.
  \end{align}
  We further introduce a constant background charge density (see
  Section~\ref{subse:background_density}) with numerical value $\rho_b = -
  \overline{\rho}$.
\end{definition}
With this setup the approximate value of the plasma frequency $\omega_p =
\sqrt{\rho\,\alpha} \approx \sqrt{\overline{\rho}\alpha} = \sqrt{10^5}$;
therefore the plasma period is $\tp \approx 0.01986$. We consider a final
simulation time of $\tf = 5\,\tp$ and enforce slip boundary conditions on
the momentum, viz., $\mom \cdot \normal = 0$ at the boundary
$\partial\Omega$, as well as homogenenous Neumann boundary conditions on
the potential, $\nabla\Epot \cdot\normal =0$ at the boundary
$\partial\Omega$. Slip boundary conditions ensure that the total mass,
$\int_\Omega\rho(\xcoord,t)\dx$  of the system is conserved. This ensures
that the total charge density remains compatible with the Neumann boundary
conditions,
\begin{align*}
  0
  \;=\;\int_{\partial\Omega}\nabla\Epot\cdot\normal\ds
  \;=\;\int_{\Omega}\Delta\Epot\dx
  \;=\;\int_{\Omega}-\alpha\big(\rho(\xcoord,t) + \rho_b\big)\dx
  \;=\;0.
\end{align*}

We briefly comment on an implementational detail.
\begin{remark}[Rank deficiency due to Neumann boundary conditions]
  The stiffness matrices of the discrete Poisson problem
  \eqref{eq:discrete_gauss} and the source update
  \eqref{ProtoSemiDiscreteEpot} (either with bilinear form
  \eqref{BilinearForm}, or \eqref{BilinearFormQuad}) are rank deficient due
  to our choice of homogeneous Neumann boundary conditions for the
  potential. The kernel of the stiffness matrices is one dimensional and
  contains all constant functions in $\FESpacePot$. In order to deal with
  this defect we use a \emph{mean value filter} $\mathcal{P}$ to eliminate
  all constant modes; a detailed discussion of such filtering techniques
  can be found in \cite{BochevLehoucq}. More precisely, for a given vector
  $\EpotVect=\{\EpotVect_i\}_{i\in\potvertices}$ we set
  \begin{align*}
    (\mathcal{P}\EpotVect)_i &= \EpotVect_i - \mu(\EpotVect),
    \quad
    \text{for }i\in\potvertices,
    \text{ where}
    \quad
    \mu(\EpotVect) = \frac{1}{|\Omega|}
    \sum_{j \in \potvertices} m_j \EpotVect_j.
  \end{align*}
  The mean value filter is now applied to all right hand sides to ensure
  that the right hand side vector is in the column space of the stiffness
  matrix. Moreover, the mean value filter should also be applied to
  intermediate updates after every multiplication with the stiffness matrix
  when solving with a Krylov space method, such as conjugate gradient
  method.
\end{remark}

We note that our choice of source update scheme does not require a
CFL condition; see Sections~\ref{subse:scheme_affine}
and~\ref{subse:non_simplicial}. This implies that the full algorithm is
only subject to a hyperbolic CFL condition; see
Apppendix~\ref{app:hyperbolic}. In fact, using a (relative) CFL number of
$0.75$ we encounter only a very mild hyperbolic CFL constraint resulting in
time step sizes of about $\dt\,\sim\,1.25\,\tp$. While this is an ideal
situation---resolving the plasma frequency is not required---we
nevertheless want to fully resolve plasma oscillations in this case. We
thus choose a very small relative CFL number of $0.005$.

Figure~\ref{PlasmaOscResolved} shows nine temporal snapshots of the density
of a simulation covering a full plasma period (out of the total of five
periods simulated). We used the source update scheme
\eqref{ProtoSemiDiscreteEpot} with bilinear form \eqref{BilinearFormQuad},
and strategy (a), no restart of the Gauß law. Most noticeably the
stationary contact at $x=0.5$ is very well preserved.

Figure~\ref{PlasmaOscResolved2} shows a comparison of the effect of our
three different choices of restart, (a) no restart, (b) relaxation, and
(c) full restart. We note that the difference between the full restart and
relaxation restart are minimal. A slightly more pronounced Gibbs
phenomenon at the stationary contact is visible for both restart strategies
in comparison to strategy (a), no restart.

%%%%%%%%%%%%%%%%%%%%%%%%%%%%%%%%%%%%%%%%%%%%%%%%%%%%%%%%%%%%%%%%%%%%%%%%%%%%%%%%
%%%%%%%%%%%%%%%%%%%%%%%%%%%%%%%%%%%%%%%%%%%%%%%%%%%%%%%%%%%%%%%%%%%%%%%%%%%%%%%%
%%%%%%%%%%%%%%%%%%%%%%%%%%%%%%%%%%%%%%%%%%%%%%%%%%%%%%%%%%%%%%%%%%%%%%%%%%%%%%%%

\subsection{Electrostatic Implosion}
\label{subse:num:implosion}

We now consider an electrostatic implosion configuration in a circular
domain $\Omega$ of radius $r_3 = 16$, that is, $\Omega = \{\xcoord \in
\mathbb{R}^2 \;:\; |\xcoord|_{\ell^2} \leq r_3 \}$, with boundary
conditions $\bv{p}\cdot\normal = 0$ and $\Epot = 0$ on the entirety of
$\partial\Omega$ and the parameters $\alpha=10^{3}$, and $\gamma=5/3$. The
initial state is uniform with density $\rho_0 = 1$, velocity $\bv{v}_0 =
\bzero$, and pressure $p_0 = 10^{-4}$. Given radii $r_2 = 6$ and $r_1 = 4$
we consider a constant background charge density $\rho_{b}$ defined as follow:
\begin{align}
  \rho_{b}(\xcoord)\;=\;
  \begin{cases}
    10000 & \text{if } r_1 < |\xcoord|_{\ell^2}\leq r_2, \\
    0     & \text{otherwise}.
  \end{cases}
\end{align}
The final time is set to $\tf = \tfrac{3}{64} \tp$, where $\tp =
\tfrac{2\pi}{\omega_p}$, and $\omega_p=\sqrt{\rho_0\alpha}=\sqrt{10^3}$.

This geometric setup is similar to considering a configuration with two
concentric cylindrical electrodes, with the outer electrode grounded and
the inner electrode having a very high positive voltage pulling the
electron gas inwards. As the bulk of the electron fluid is accelerated
towards the center a cylindrical outer region with very low density and low
pressure is left behind. Such a configuration with strong implosion and
compression and an emerging near-vacuum region are well-known to be
hydrodynamically highly unstable, see for instance \cite{Sitenko1995,
Drake2010}.
The configuration is an excellent starting point to judge the merits of the
scheme in relationship to its ability to work in the shock hydrodynamics
regime (discontinuous solutions and strong expansions). The configuration
leads to large material velocities which necessitates small time-step
sizes due to the hyperbolic CFL condition \eqref{eq:cfl}. On the other hand
the plasma frequency is moderate. This implies that the smallest time-scale
(that has to be resolved) is dominated by the hydrodynamic subsystem and
not by electrostatic effects. A reference computation with 1M
quadrilaterals visualizing the dynamics is shown in
Figure~\ref{fig:implosion-resolved}. Temporal snapshots of three
different computations with different choices of restart (no restart,
relaxation, full restart) are given in
Figure~\ref{fig:implosion-comparison}. All in all, the qualitative difference 
of all three computations is minimal. All of them seem to capture the dynamics 
accurately and are close to the reference computation 
\ref{fig:implosion-resolved}\,(h).
\begin{figure}[p!]

  \begin{center}
    \setlength\fboxsep{0pt}
    \setlength\fboxrule{0.5pt}
    \subfloat[$t=1/16\,\tf$]{\fbox{\includegraphics[width=40mm]{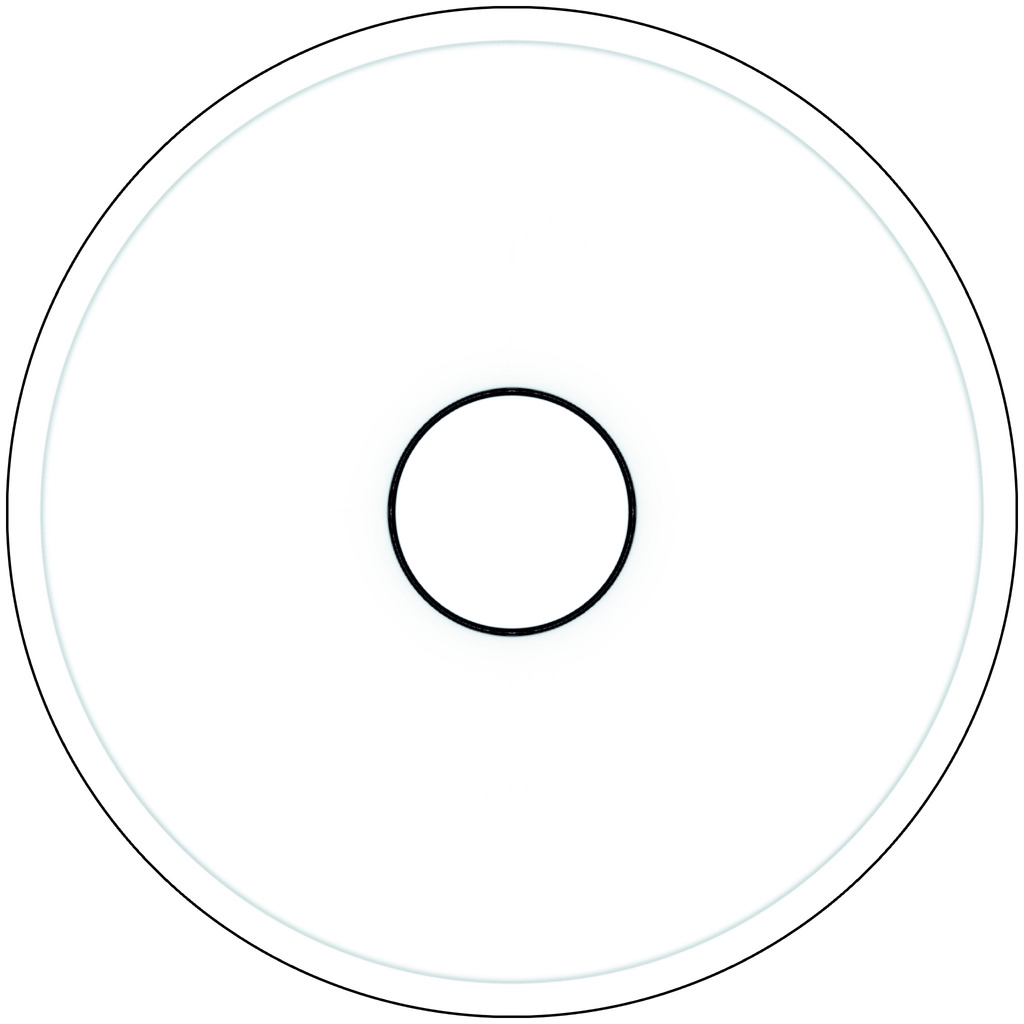}}}\;
    \subfloat[$t=1/8\,\tf$] {\fbox{\includegraphics[width=40mm]{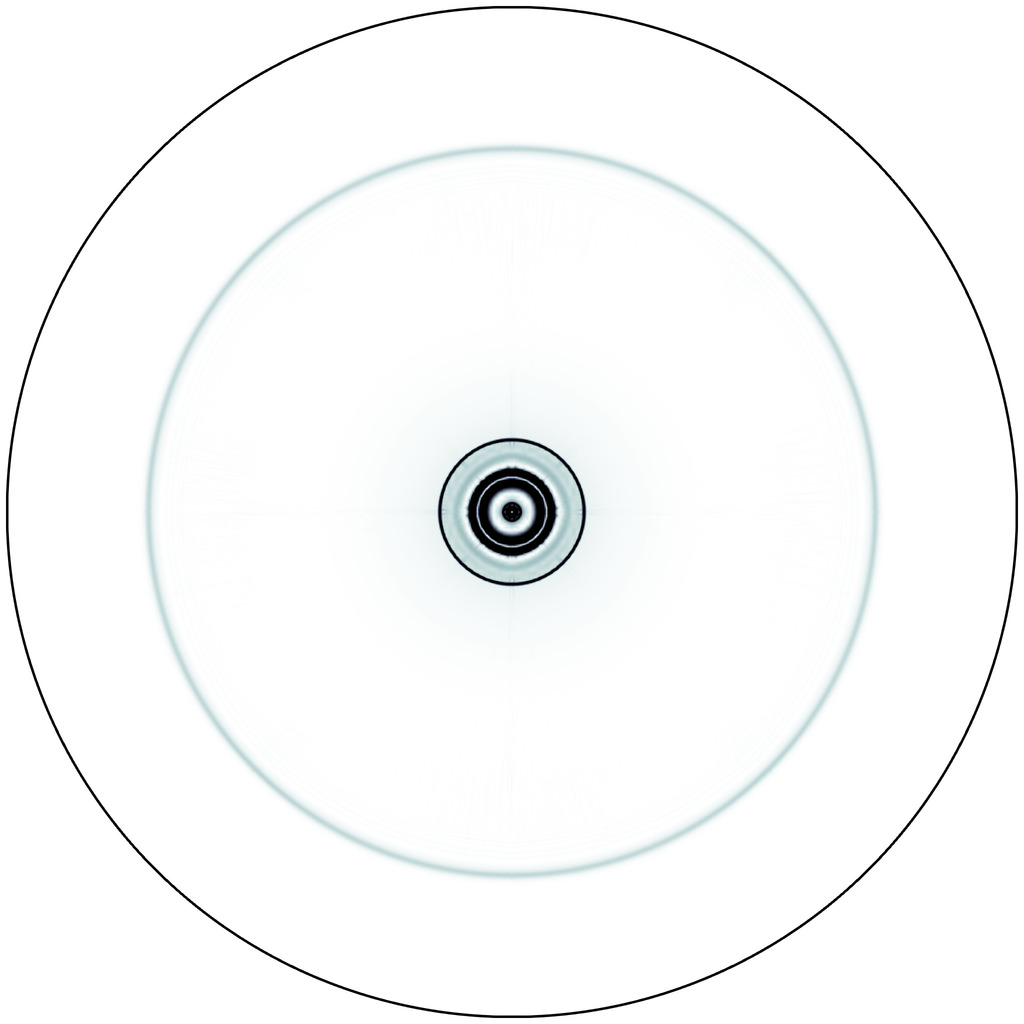}}}\;
    \subfloat[$t=2/8\,\tf$] {\fbox{\includegraphics[width=40mm]{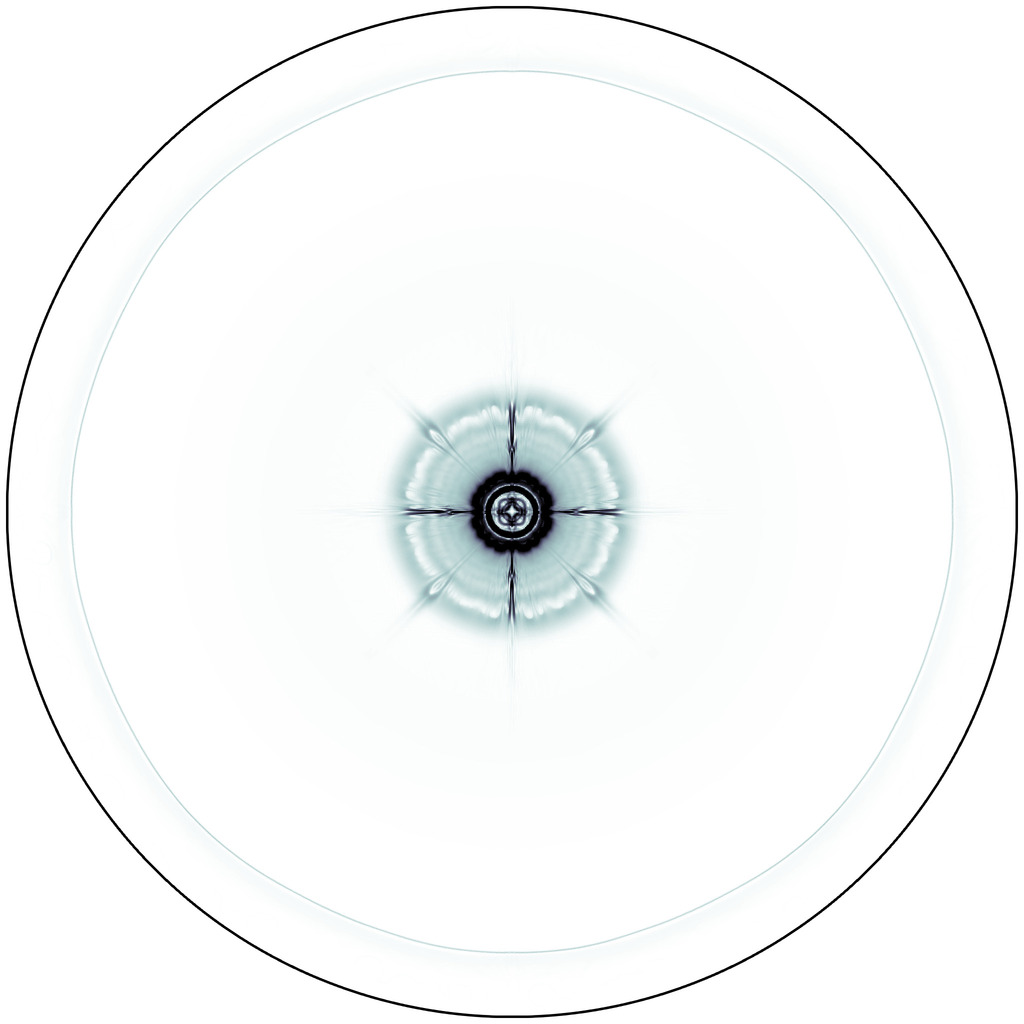}}}
    \vspace{-0.75em}

    \subfloat[$t=3/8\,\tf$] {\fbox{\includegraphics[width=40mm]{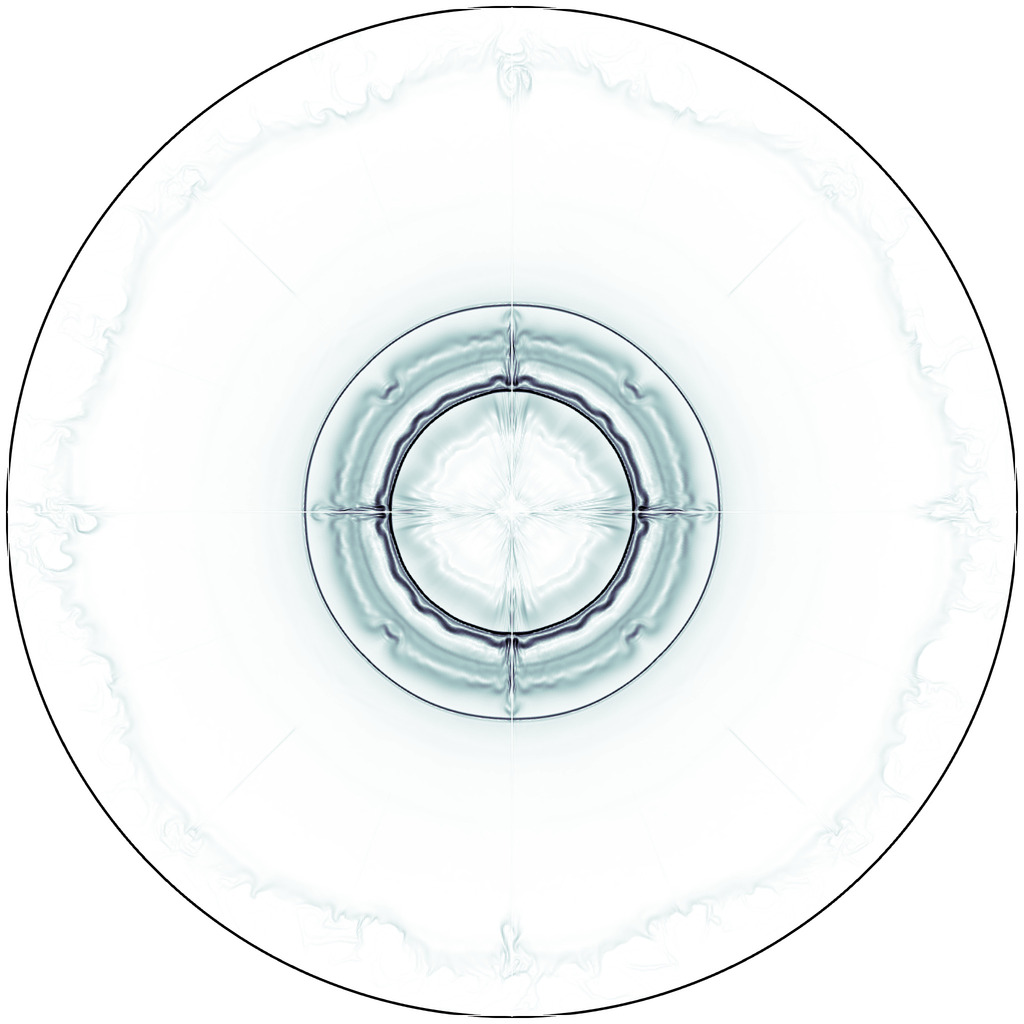}}}\;
    \subfloat[$t=4/8\,\tf$] {\fbox{\includegraphics[width=40mm]{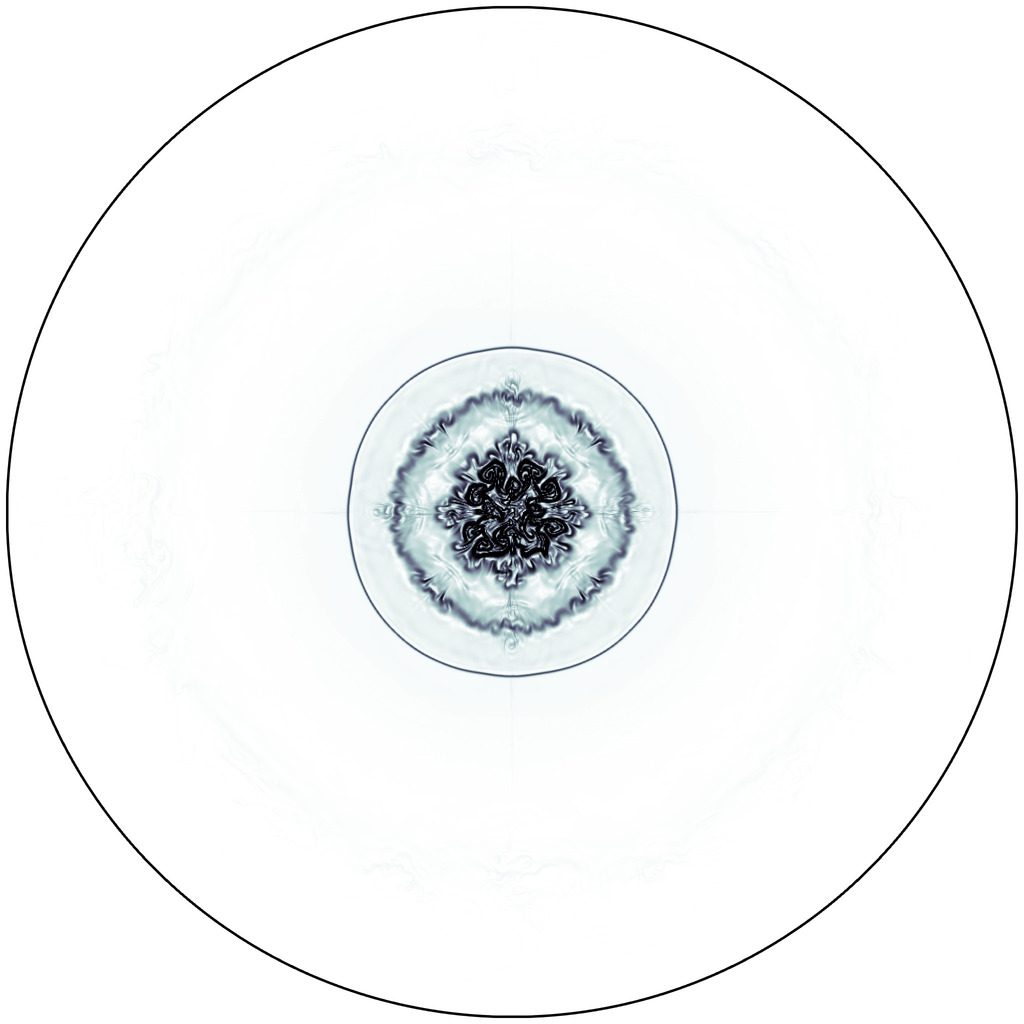}}}\;
    \subfloat[$t=5/8\,\tf$] {\fbox{\includegraphics[width=40mm]{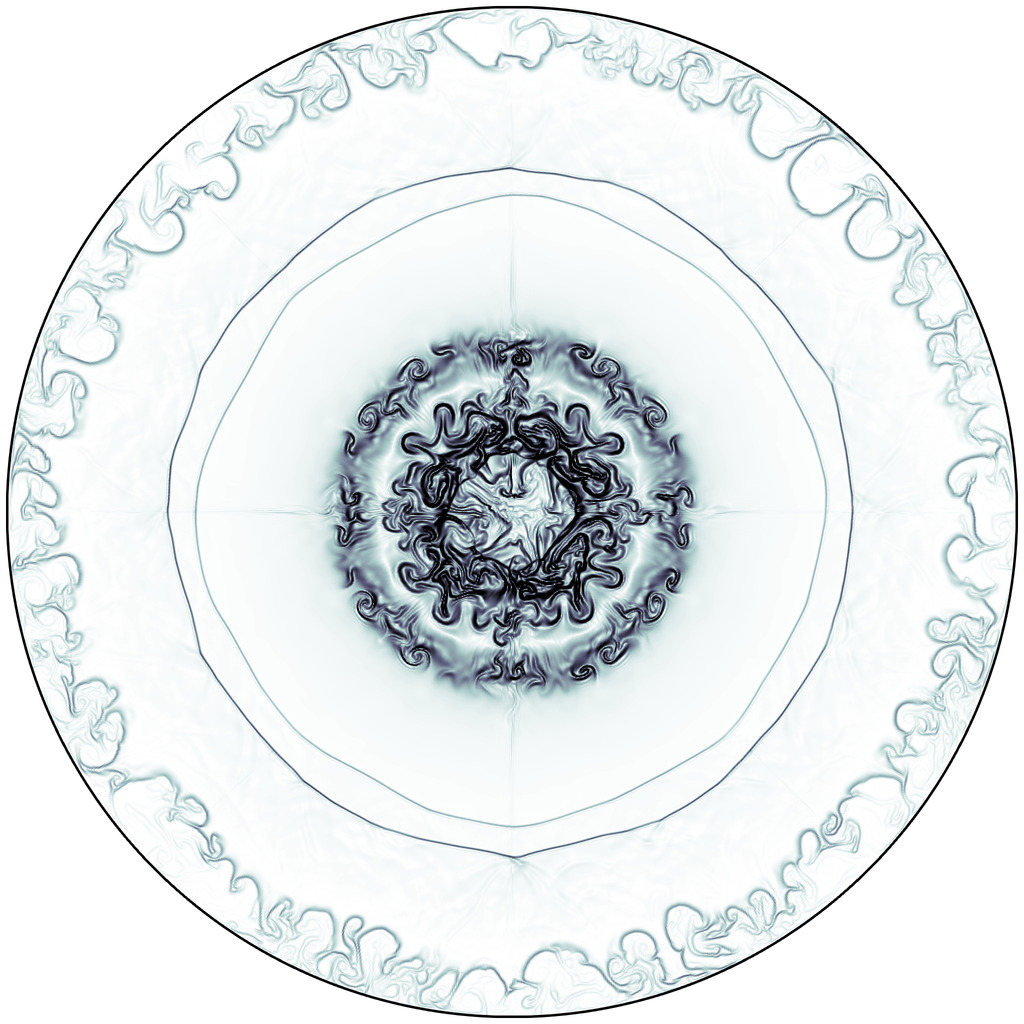}}}
    \vspace{-0.75em}

    \subfloat[$t=6/8\,\tf$] {\fbox{\includegraphics[width=40mm]{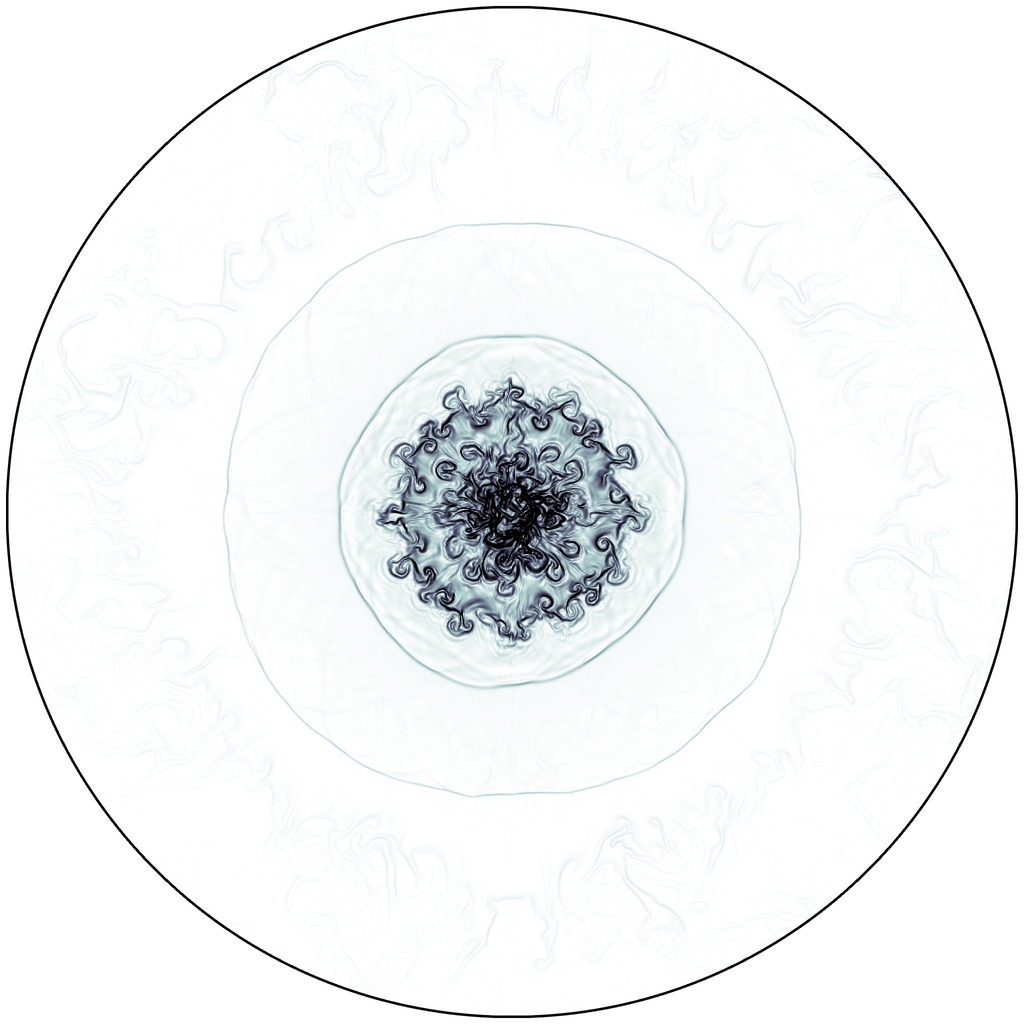}}}\;
    \subfloat[$t=7/8\,\tf$] {\fbox{\includegraphics[width=40mm]{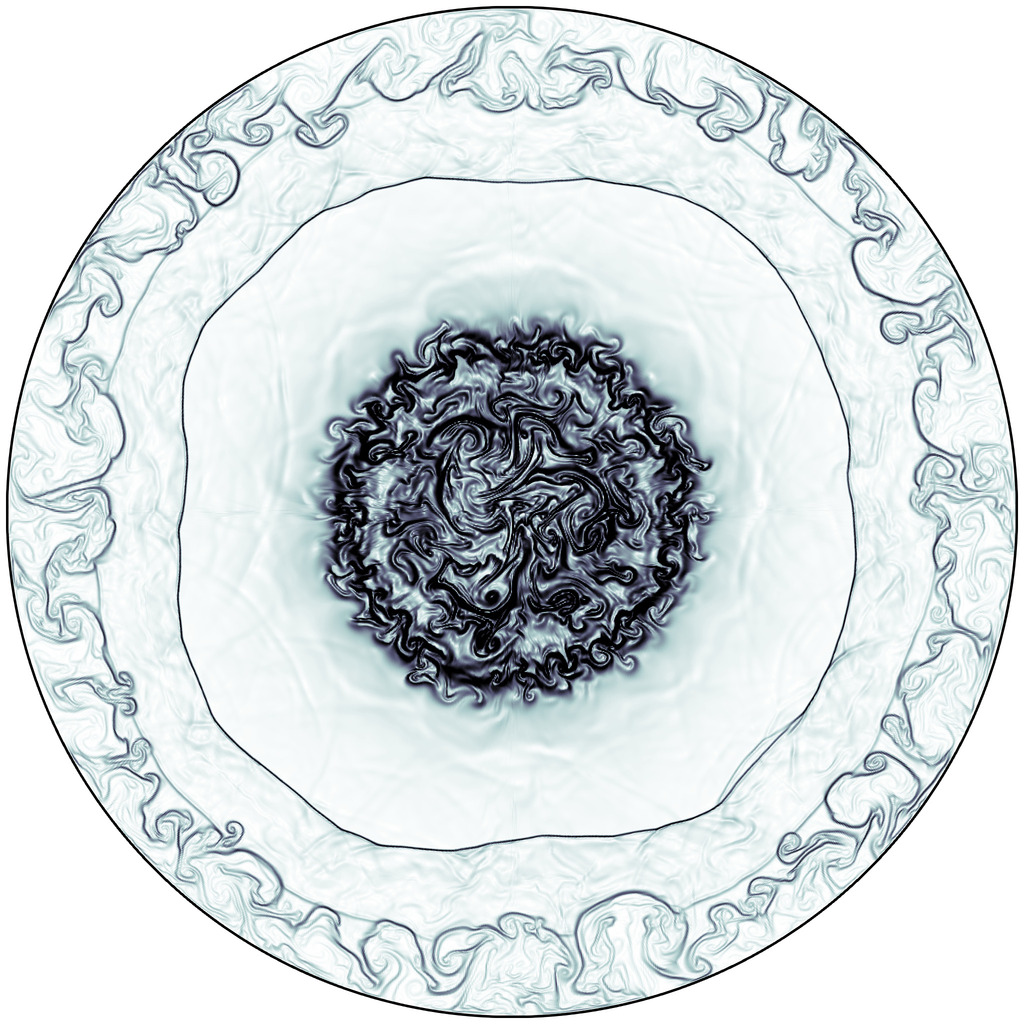}}}\;
    \subfloat[$t=F_f$]      {\fbox{\includegraphics[width=40mm]{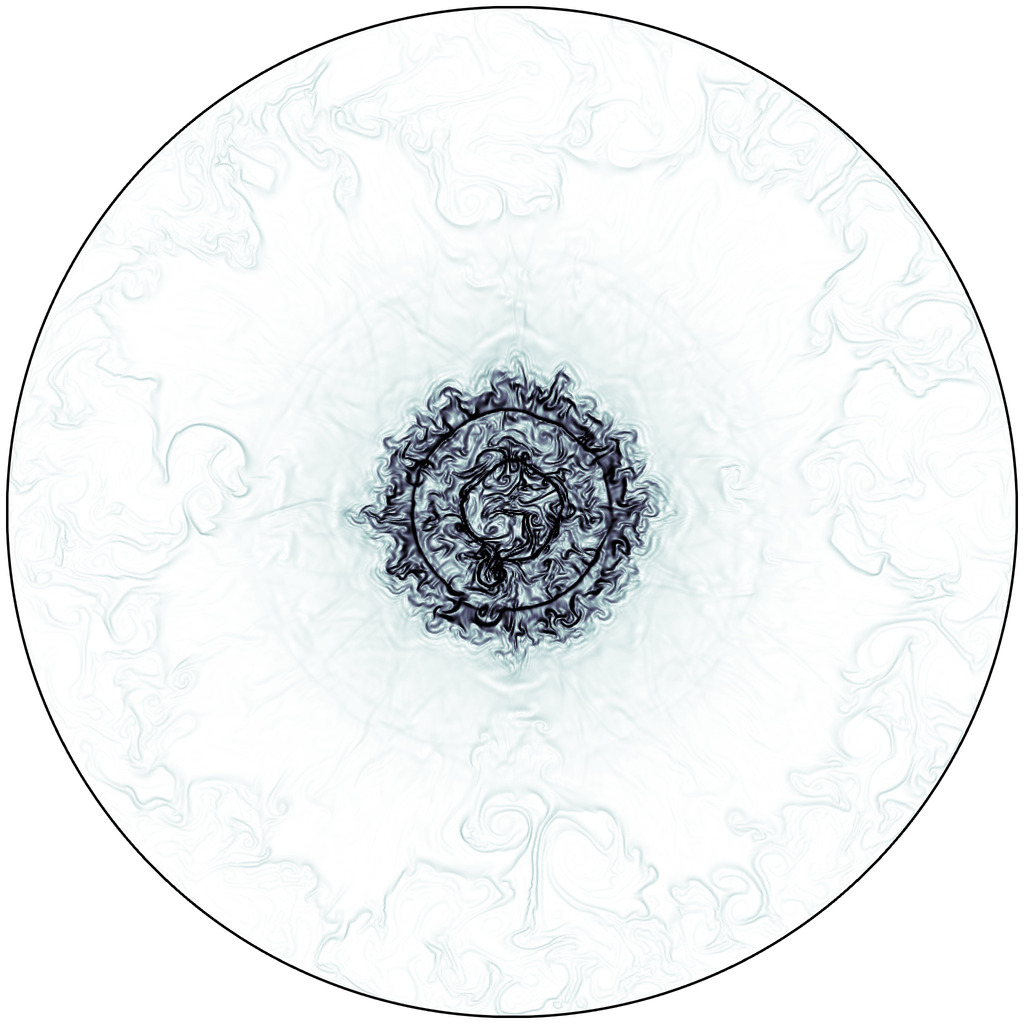}}}
  \end{center}
  \caption{\label{fig:implosion-resolved}%
    Temporal snapshots of a schlieren plot of the density profile of the
    electrostatic implosion test case. Reference computation with full
    restart and approximately 1M quadrilaterals.}

  \begin{center}
    \setlength\fboxsep{0pt}
    \setlength\fboxrule{0.5pt}
    \subfloat[no restart]{\fbox{\includegraphics[width=40mm]{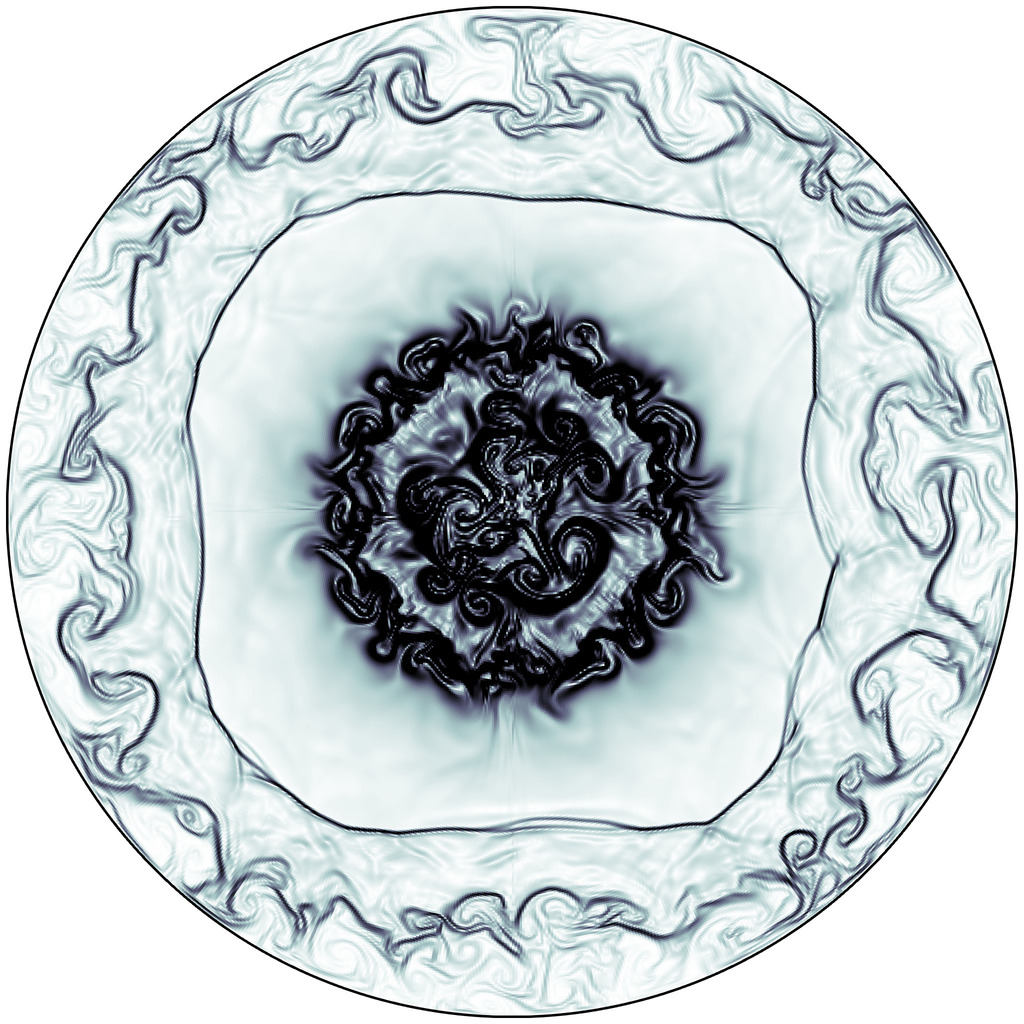}}}\;
    \subfloat[relaxation] {\fbox{\includegraphics[width=40mm]{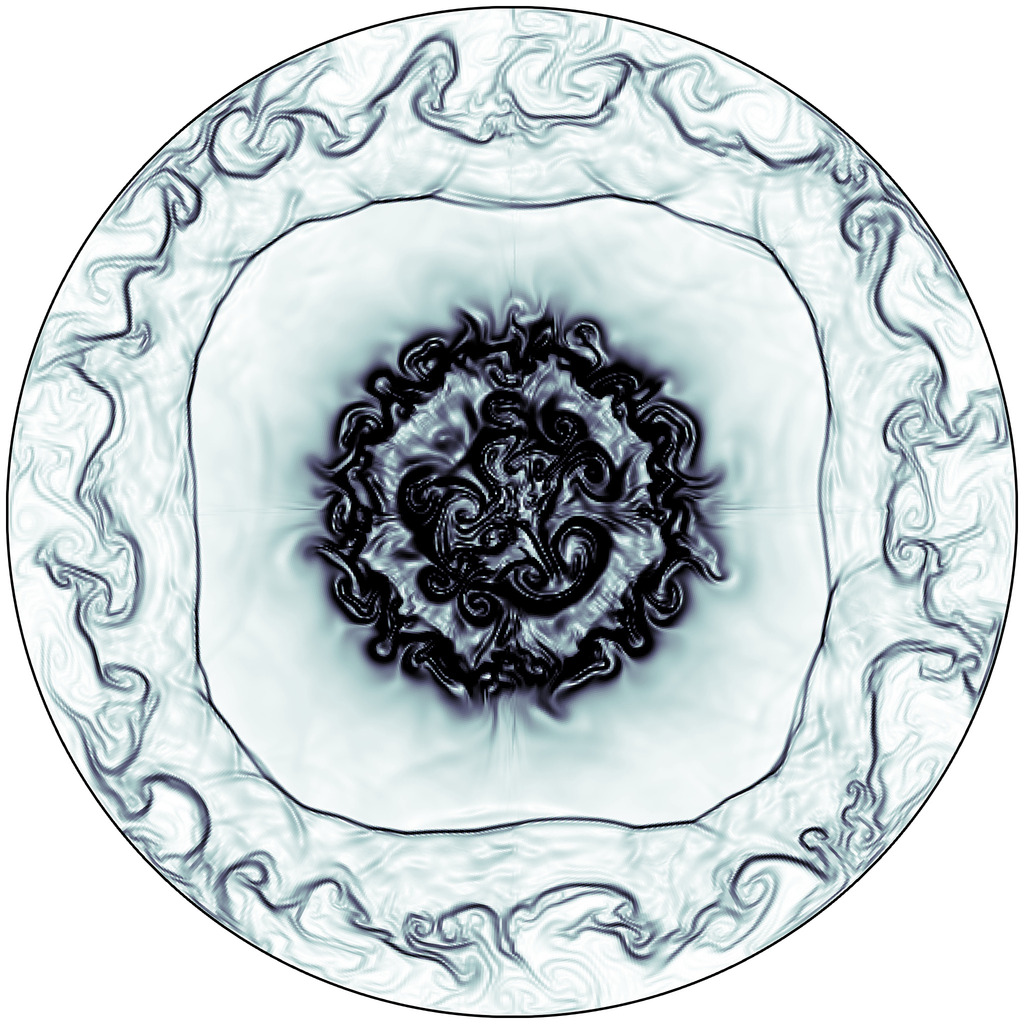}}}\;
    \subfloat[full restart] {\fbox{\includegraphics[width=40mm]{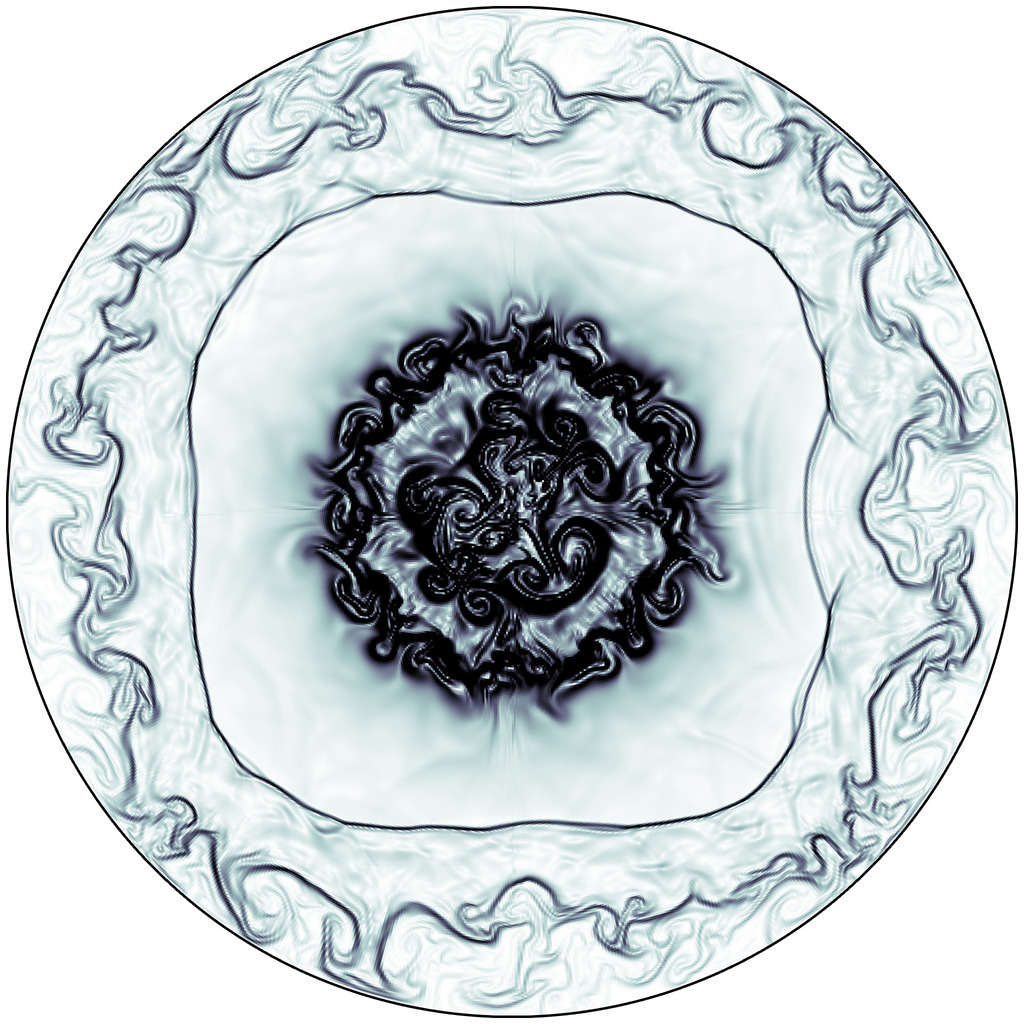}}}
  \end{center}
  \caption{\label{fig:implosion-comparison}%
    Comparison of the effects of different choices of restart on the
    electrostatic implosion test case. The snapshots are taken at $t =
    \tfrac{7}{8}\tf$ and show a schlieren plot of the density profile. From
    left to right we use: (a) no restart, (b) relaxation, and (c) full
    restart.}
\end{figure}
\begin{figure}[t!]
  \centering
  \includegraphics{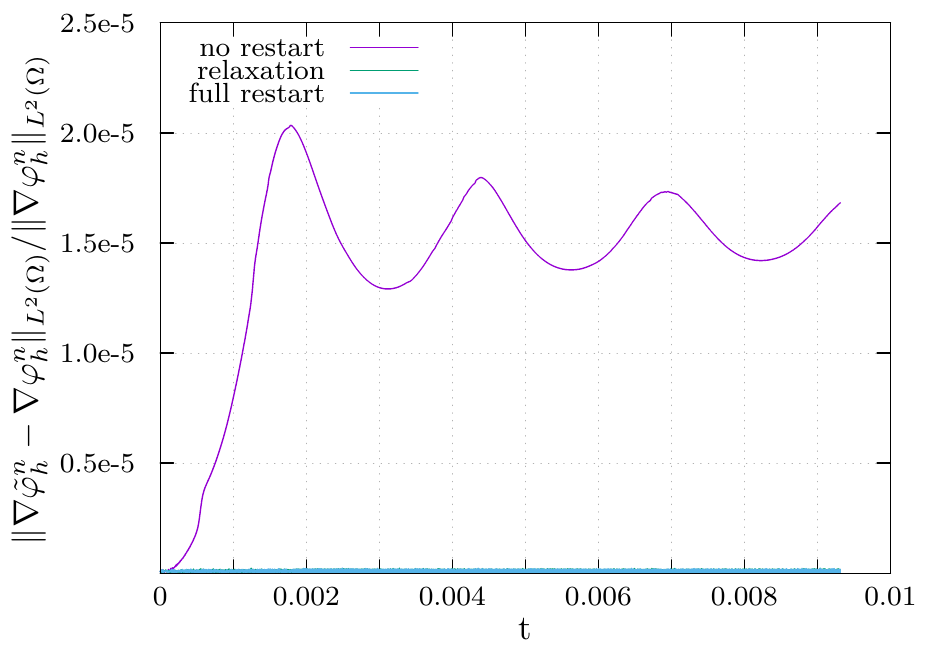}
  \caption{
    \label{fig:implosion-gauss-law}
    Norm of the Gauß law residual \eqref{eq:gauss_law_residual} as a
    function of time for three different choices of restart, (a) no
    restart, (b) relaxation, and (c) full restart. The curves for
    relaxation restart and full restart lie on top of each other.}

  \includegraphics{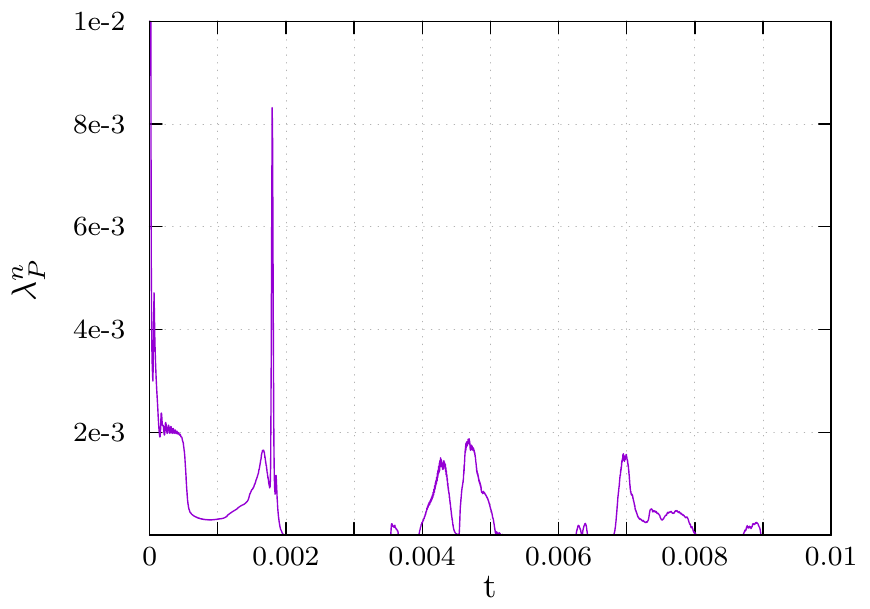}
  \caption{
    \label{fig:implosion-lambda}
    Relaxation parameter $\lambda_P^n$ as a function of time for the
    relaxation strategy. The scale had been clipped at 1.0e-2. The maximum
    $\lambda_P^N$ occurs during the first step and has a magnitude of
    around 2.0e-1.}
\end{figure}
As an additional figure of merit we report the Gauß-law residual
\eqref{eq:gauss_law_residual} as a function of time for all three choices
of restart in Figure~\ref{fig:implosion-gauss-law}. We highlight that the
Gauß-law residual for the no-restart strategy accumulates a maximal
relative deviation of around $\text{2e-5}$ after about about 47,000
time-steps. Bearing in mind that this is a complex simulation with a
non-smooth solution this is an excellent result. The residual for the other
two choices of restart remains in the $\text{1e-8}$ range and are not
distinguishable in Figure~\ref{fig:implosion-gauss-law}.

An important aspect of the relaxation technique is the prospect that it can
maintain the (discrete) Gauß-law exactly with only a minimal correction to
the kinetic energy. Figure~\ref{fig:implosion-lambda} indeed corroborates
this observation: The numerical value of the relaxation parameter
$\lambda^n_P$ is almost always well below 2e-3 throughout the computation.

%%%%%%%%%%%%%%%%%%%%%%%%%%%%%%%%%%%%%%%%%%%%%%%%%%%%%%%%%%%%%%%%%%%%%%%%%%%%%%%%
%%%%%%%%%%%%%%%%%%%%%%%%%%%%%%%%%%%%%%%%%%%%%%%%%%%%%%%%%%%%%%%%%%%%%%%%%%%%%%%%
%%%%%%%%%%%%%%%%%%%%%%%%%%%%%%%%%%%%%%%%%%%%%%%%%%%%%%%%%%%%%%%%%%%%%%%%%%%%%%%%

\section{Conclusion and outlook}
\label{sec:conclusion}

In this paper we have discussed a fully discrete numerical scheme for the
Euler-Poisson system that preserves the invariant domain of the Euler
subsystem, as well as energy stability properties. We have demonstrated
that the Gauss law can also be preserved at the same time with a simple
postprocessing technique that introduces a mild amount of artificial
relaxation to the system.

In order to satisfy energy stability as well as well-posedness of the
linear algebra system associated to each time step, the source update scheme 
requires special attention in regards to the choice of finite
element spaces and quadrature rules. In particular, we have made extensive
use of lumped quadratures. This presents no obstacle for the case of
simplicial affinely mapped elements. However, it is a delicate issue in the
context of quadrilateral/hexahedral meshes, where it is well known that
element distortion can degrade second-order accuracy. We have shown that
our proposed scheme exhibits second-order accuracy in the context of
asymptotically affine mesh sequences (of quadrilaterals) which are the
family of nested meshes most frequently used in practice.

In a sequence of numerical experiments we have verified and highlighted
qualitative and quantitative properties of the scheme. In addition to the
grid convergence results used to understand the effects of mesh distortion,
we have also demonstrated robustness of our scheme when simulating pure
plasma oscillation with minimal \emph{gas} effects, as well as an
electrostatic implosion test highlighting our capabilities in the
shock-hydrodynamics regime.

The results presented in this paper are first building block in a larger
effort of targeting the development of numerical methods for the full
Euler-Maxwell system. A future publication will discuss the application of
our scheme to systems incorporating magnetic field effects. A crucial
property for us to demonstrate in this context will be the ability to
overstep the plasma or cyclotron frequency similarly as we discussed in the
present manuscript.

%%%%%%%%%%%%%%%%%%%%%%%%%%%%%%%%%%%%%%%%%%%%%%%%%%%%%%%%%%%%%%%%%%%%%%%%%%%%%%%%
%%%%%%%%%%%%%%%%%%%%%%%%%%%%%%%%%%%%%%%%%%%%%%%%%%%%%%%%%%%%%%%%%%%%%%%%%%%%%%%%
%%%%%%%%%%%%%%%%%%%%%%%%%%%%%%%%%%%%%%%%%%%%%%%%%%%%%%%%%%%%%%%%%%%%%%%%%%%%%%%%

\section*{Acknowledgements}
M.M. acknowledges partial support from the National Science Foundation
grants DMS-1912847 and DMS-2045636. The work of I.T. and and J.S. was
partially supported by LDRD-express project \#223796 and LDRD-CIS project
\#226834 from Sandia National Laboratories, and by the U.S. Department of
Energy, Office of Science, Office of Advanced Scientific Computing
Research, Applied Mathematics Program and by the U.S. Department of Energy,
Office of Science, Office of Advanced Scientific Computing Research and
Office of Fusion Energy Sciences, Scientific Discovery through Advanced
Computing (SciDAC) program. The authors acknowledge the Texas Advanced
Computing Center (TACC) at The University of Texas at Austin for providing
HPC resources that have contributed to the research results reported within
this paper; \url{http://www.tacc.utexas.edu}.

%%%%%%%%%%%%%%%%%%%%%%%%%%%%%%%%%%%%%%%%%%%%%%%%%%%%%%%%%%%%%%%%%%%%%%%%%%%%%%%%
%%%%%%%%%%%%%%%%%%%%%%%%%%%%%%%%%%%%%%%%%%%%%%%%%%%%%%%%%%%%%%%%%%%%%%%%%%%%%%%%
%%%%%%%%%%%%%%%%%%%%%%%%%%%%%%%%%%%%%%%%%%%%%%%%%%%%%%%%%%%%%%%%%%%%%%%%%%%%%%%%

\appendix
\section{Graph-based hyperbolic solver}
\label{app:hyperbolic}
For the hyperbolic subsystem we use a framework of numerical schemes based
on a graph-viscosity stabilization and convex limiting \cite{GuePo2016,
Guermond2016, Guer2018, GuePo2019, Maier2021}. The framework is
discretization agnostic, meaning that it can in principle be used in
conjunction with continuous or discontinuous finite element, finite volume,
or finite difference formulations. In this paper, however, we use a
discontinuous finite element ansatz for reasons discussed in
Section~\ref{sec:numerical_approach} regarding energy stability. For the
sake of completeness and reproducibility we summarize some implementational
aspects in this appendix, for a complete overview of the methodology we
refer the reader to \cite{GuePo2019}. Numerical methods, that can preserve 
convex constraints of the solution (e.g. positivity properties), are also 
advanced in \cite{wu2021minimum, wu2021geometric, Zhang2010} and references 
therein.

%%%%%%%%%%%%%%%%%%%%%%%%%%%%%%%%%%%%%%%%%%%%%%%%%%%%%%%%%%%%%%%%%%%%%%%%%%%%%%%%
\subsection{Discrete divergence operator and stencil}

For every element $\element \in \triangulation$, we define the following
set of indices
\begin{align*}
  \mathcal{I}(\element) &:= \big\{ i \in \hypvertices \ | \
  \HypBasisScal_{i|\element} \not \equiv 0 \big\} \,,
\end{align*}
That is, $\mathcal{I}(\element)$ is the set of all shape functions with support
on the interior of the element $\element$. For every $i \in
\mathcal{I}(\element)$ and every $j \in \hypvertices$ we define the vector
$\bv{c}_{ij} \in \mathbb{R}^{d}$ as
\begin{equation*}
  \bv{c}_{ij} :=
  \begin{cases}
    \bv{c}_{ij}^\element-\bv{c}_{ij}^{\partial\element}\;
    &\text{if }j \in \mathcal{I}(\element), \\
    \bv{c}_{ij}^{\partial\element}
    &\text{if }j \in \mathcal{V} \backslash \mathcal{I}(\element),
  \end{cases}
\end{equation*}
\vspace{-1em}
where
\vspace{-1em}
\begin{equation*}
  \bv{c}_{ij}^{\element} := \int_{\element}\HypBasisScal_i \nabla\HypBasisScal_j
\dx,
  \qquad
  \bv{c}_{ij}^{\partial\element} :=
  \frac{1}{2} \int_{\partial K \backslash \partial\Omega} \HypBasisScal_{j}
  \HypBasisScal_i \normal_\element \ds.
\end{equation*}
where $\normal_\element$ is the outwards pointing normal of the element
$\element$. Note that: $\bv{c}_{ij}^{\partial\element}$ is necessarily zero
if $\HypBasisScal_j$ does not have support on the element $\element$ or on
one of its immediate neighbors. With this observation in mind, we define
the stencil at the node $i$ as follows:
% \vspace{-1em}
%
\begin{align*}
  \mathcal{I}(i) = \{j \in \hypvertices \;|\; \bv{c}_{ij} \not = 0 \}.
\end{align*}
The set of vectors $\{\bv{c}_{ij}\}_{j \in \mathcal{I}(i)}$ is used to
construct an approximation of the divergence operator at each node $i$ in
the spirit of a collocation scheme \cite{GuePo2019}. We highlight that this
approximation of the divergence operator is consistent with the polynomial
degree of the shape functions $\{\HypBasisScal_i\}_{i \in \hypvertices}$
and works with arbitrary meshes, see \cite{GuePo2019} for more
details.

%%%%%%%%%%%%%%%%%%%%%%%%%%%%%%%%%%%%%%%%%%%%%%%%%%%%%%%%%%%%%%%%%%%%%%%%%%%%%%%%
\subsection{Scheme}

For a given state $\bv{U}_i^{n}$ we define a low-order update
$\bv{U}_i^{n+1,\mathrm{L}}$ approximating the solution of
\eqref{EulerWithSources} as follows:
\begin{align}
  \label{LowOrderScheme}
  m_i \frac{\bv{U}_i^{n+1,\mathrm{L}} - \bv{U}_i^{n}}{\dt}
  + \sum_{j \in \mathcal{I}(i)} \mathbb{f}(\bv{U}_j^{n})\,\bv{c}_{ij}
  - d_{ij}^{n,\mathrm{L}} (\bv{U}_j^{n} - \bv{U}_i^{n}) = 0,
  \text{for all }i\in \hypvertices,
\end{align}
where we have set $m_i = \int_{\Omega} \HypBasisScal_i \dx$, and where
$\mathbb{f}(\bv{U}_j^{n}) \in \mathbb{R}^{(d+2) \times d}$ is the flux at
the node $j \in \mathcal{I}(i)$, and $d_{ij}^{n,\mathrm{L}} \in
\mathbb{R}^+$ is a viscosity coefficient defined as
\begin{align*}
  d_{ij}^{n,\mathrm{L}} := \max
  \{\lambda^{\text{max}}(\bv{U}_i^{n},\bv{U}_j^{n},\bv{n}_{ij})
  |\bv{c}_{ij}|_{\ell^2},
  \lambda^{\text{max}}(\bv{U}_j^{n},\bv{U}_i^{n},\bv{n}_{ji})
  |\bv{c}_{ji}|_{\ell^2}\}
\end{align*}
Here, $\lambda^{\text{max}}(\bv{U}_L,\bv{V}_R,\bv{n})$ is any upper-bound
on the maximum wavespeed of propagation of the projected-Riemann problem
(setting $x := \xcoord \cdot \normal$):
\begin{align*}
  \partial_{t}\bv{U} + \partial_x (\flux(\bv{U}) \cdot \normal) = 0
  \quad\text{with initial data}\quad
  \bv{U}_0 =
  \begin{cases}
    \bv{U}_{L} &\text{if } x \leq 0, \\
    \bv{U}_{R} &\text{if } x > 0, \\
  \end{cases}
\end{align*}
and we set $d_{ii}^{n,\mathrm{L}}= - \sum_{j\in\mathcal{I}{(i)}\backslash \{i\}}
d_{ij}^{n,\mathrm{L}}$.
Then under the \emph{hyperbolic CFL} condition
\begin{align}
  \label{eq:cfl}
  \dt_n := -\text{c}_{\text{cfl}}  \min_{i \in \hypvertices}
  \tfrac{m_i}{2 d_{ii}^{n,\mathrm{L}}}
\end{align}
the update $\bv{U}_i^{n+1,\mathrm{L}}$ as defined by \eqref{LowOrderScheme}
maintains the invariant domain $\mathcal{A}$ \cite{GuePo2016, Guermond2016,
Guer2018, GuePo2019}, viz.,
\begin{equation}
  \bv{U}_i^{n+1,\mathrm{L}}\;\in\;
  \mathcal{A}\;:=\;\big\{\state =(\den,\mom, \totme)
  \text{ such that } \rho>0, \ e>0,\ s(e,\rho)\ge s_{\min}\big\}.
\end{equation}

\paragraph{High-order update and convex limiting}
We also introduce a corresponding high-order method,
\begin{align*}
  m_i \frac{\bv{U}_i^{n+1,\mathrm{H}} - \bv{U}_i^{n}}{\dt_n}
  + \sum_{j \in \mathcal{I}(i)} \mathbb{f}(\bv{U}_j^{n})\,\bv{c}_{ij}
  - d_{ij}^{n,\mathrm{H}} (\bv{U}_j^{n} - \bv{U}_i^{n}) = 0,
\end{align*}
where the only difference with the low-order scheme \eqref{LowOrderScheme}
lies in the choice of a high-order viscosity $d_{ij}^{n,\mathrm{H}}$. The
high-order graph viscosities are typically constructed such that
$d_{ij}^{n,\mathrm{H}} \approx d_{ij}^{n,\mathrm{L}}$ near shocks and
discontinuities, but $d_{ij}^{n,\mathrm{H}} \approx 0$ in smooth regions of
the solution. A possible choice is to construct local indicators estimating
the entropy production or local smoothness of the solution and use those to
construct the high-order viscosity \cite{Guer2017,Guer2018}.
In the computations reported in this manuscript, however, we use a
very simple approach by setting
\begin{align*}
  d_{ij}^{n,\mathrm{H}} =
  \begin{cases}
    d_{ij}^{n,\mathrm{L}} &\text{if } \xcoord_i = \xcoord_j, \\
    0 &\text{otherwise}.
  \end{cases}
\end{align*}
This definition is equivalent to using the low-order viscosity only on the
faces of the elements. We observe numerically good convergence rates for
$\mathbb{P}^1$ and $\mathbb{Q}^1$ elements.

The high-order solution $\bv{U}_i^{n+1,\mathrm{H}}$ it is not
invariant domain preserving and can consequently not be used directly
\cite{Guer2018}. In order to maintain invariant domain preservation and the
high approximation order we blend the low-order solution and high-order
solution together in a postprocessing step by setting
\begin{align*}
  \bv{U}_i^{n+1} = \bv{U}_i^{n+1,\mathrm{L}}
  + \sum_{j \in \mathcal{I}(i)} \ell_{ij} \bv{A}_{ij}
  \quad\text{with}\quad
  \bv{A}_{ij} := \dt_n (d_{ij}^{n,\mathrm{H}} - d_{ij}^{n,\mathrm{L}})
  (\bv{U}_j^{n} -\bv{U}_i^{n}).
\end{align*}
Here, the limiter matrix $\ell_{ij}\in[0,1]$ is computed with a convex
limiter consisting of directional line searches that ensures that
$\bv{U}_i^{n+1}\in\mathcal{A}$ \cite{Guer2018,GuePo2019}.

\begin{remark}[Time-step size constraint]
  We note that the time-step constraint \eqref{eq:cfl} depends on the low
  order viscosities $\{ d_{ij}^{n,\mathrm{L}} \}_{j \in
  \mathcal{I}(i)\backslash \{i\}}$. The time-step size constraint
  \eqref{eq:cfl} is sufficient for the first-order scheme, but it is overly
  pessimistic when applied to a high-order method. It is possible to use
  sharper time-step size constraint estimates for high order
  methods/viscosities. However, for the sake of simplicity, and in order to
  avoid digression, in this manuscript we use \eqref{eq:cfl}, computed with
  the low-order viscosities, even when we are actually using a high order
  method.
\end{remark}

%%%%%%%%%%%%%%%%%%%%%%%%%%%%%%%%%%%%%%%%%%%%%%%%%%%%%%%%%%%%%%%%%%%%%%%%%%%%%%%%

\subsection{High-order time stepping}

The scheme introduced so far is high order in space and first order in
time. In order to recover higher order convergence in time the basic
forward Euler step is now used as a basic building block in a high-order
SSP Runge-Kutta method \cite{Shu1988}. For implementational details we
refer the reader to \cite{Guer2021,Guer2022,Maier2021}.

%%%%%%%%%%%%%%%%%%%%%%%%%%%%%%%%%%%%%%%%%%%%%%%%%%%%%%%%%%%%%%%%%%%%%%%%%%%%%%%%
%%%%%%%%%%%%%%%%%%%%%%%%%%%%%%%%%%%%%%%%%%%%%%%%%%%%%%%%%%%%%%%%%%%%%%%%%%%%%%%%
%%%%%%%%%%%%%%%%%%%%%%%%%%%%%%%%%%%%%%%%%%%%%%%%%%%%%%%%%%%%%%%%%%%%%%%%%%%%%%%%
\section{Partial restart of the potential with line search}
\label{app:line_search}

An alternative approach to the artificial relaxation discussed in
Section~\ref{subse:relax} is to perform a line search blending
$\Epot_h^{n+1}$ computed with the update procedure
(Algorithm~\ref{alg:yanenko}, or \ref{alg:strang} respectively) and
$\tilde\Epot_h^{n+1}$ computed by solving \eqref{eq:discrete_gauss} such
that the final update is as close to the restarted potential
$\tilde\Epot_h^{n+1}$ as possible while maintaining the energy inequality.
To this end we introduce a linear combination
\begin{align}
  \Epot^{n+1,\lambda}_{h}\in\FESpacePot,
  \qquad
  \EpotVect_i^{n+1,\lambda}\;=\;
  (1-\lambda_i)\EpotVect_i^{n+1}
  \;+\;
  \lambda_i
  \tilde\EpotVect_i^{n+1},
  \label{eq:line_search_potential}
\end{align}
where $\lambda_i\in[0,1]$, $i\in\potvertices$, is determined by an
optimization  problem,
\begin{align}
  \min_{\lambda \in [0,1]} \;
  \|\varphi_h^{n+1,\lambda}-\tilde\varphi_h^{n+1}\|_{\bv \Ltwo}^2
  \;\text{subject to}\;
  \big\|\nabla\Epot_h^{n+1,\lambda}\big\|_{\bvLtwo}^2 \le
  \big\|\nabla\Epot_h^{n+1}\big\|_{\bvLtwo}^2.
  \label{eq:line_search_opt_prob}
\end{align}
This ensures that setting $\varphi_h^{n+1} \leftarrow
\varphi_h^{n+1,\lambda}$ maintains the total energy balance
\eqref{TotalEnergyBalance} while closing the difference between
$\varphi_h^{n+1,\lambda}$ and the discrete Gauß law abiding
$\tilde\varphi_h^{n+1}$ as much as possible. In light of the heuristic
error estimate established in Lemma~\ref{lem:a_new_hope} the hope arises
that $\lambda_i$ can generally be chosen very close to $1$.

The proposed optimization problem can be modified in a number of ways. For
example, the computational cost of the optimization problem can be reduced
significantly by introducing a convenient lumping and allow for possible
overrelaxation. Suppose the inequality constraint in
\eqref{eq:line_search_opt_prob} is active. The Lagrange conditions then
read
\begin{multline*}
  \sum_{j\in\mathcal{I}(i)}
  \big(\tilde\EpotVect_i^{n+1} - \EpotVect_i^{n+1}\big)
  \,\big(m_{ij}+\mu\mathcal{K}_{ij}\big)\,
  \big(\tilde\EpotVect_j^{n+1} - \EpotVect_j^{n+1}\big)
  \,(1-\lambda_j)\;=\;
  \\
  \sum_{j\in\mathcal{I}(i)}
  \big(\tilde\EpotVect_i^{n+1} - \EpotVect_i^{n+1}\big)
  \,\big(\mu\mathcal{K}_{ij}\big)\, \tilde\EpotVect_j^{n+1},
\end{multline*}
where $\mathcal{K}$ is the stiffness matrix as introduced in
Section~\ref{subse:notation} and $\mu$ is a Lagrange multiplier. Lumping
the matrices on the left-hand side then leads to an algebraic condition of
the form:
\begin{equation*}
  \big(\tilde\EpotVect_i^{n+1} - \EpotVect_i^{n+1}\big)^2
  \,m_i\,(1-\lambda_i)\;=\;
  \\
  \sum_{j\in\mathcal{I}(i)}
  \big(\tilde\EpotVect_i^{n+1} - \EpotVect_i^{n+1}\big)
  \,\big(\mu\mathcal{K}_{ij}\big)\, \tilde\EpotVect_j^{n+1},
\end{equation*}
This can be efficiently solved by setting
\begin{gather*}
  \tilde\lambda_i\;:=\;
  \frac{\sum_{j\in\mathcal{I}{(i)}}\mathcal{K}_{ij}\tilde\EpotVect_j^{n+1}}
  {m_i\,\big(\tilde\EpotVect_i^{n+1} - \EpotVect_i^{n+1}\big)},
  \quad
  \mu \;=\; \frac
  {\sum_{i,j} 2\,\mathcal{K}_{ij} \tilde\EpotVect_i^{n+1}\,
  \tilde\lambda_j \big\{\EpotVect_j^{n+1} - \tilde\EpotVect_j^{n+1}\big\}}
  {\sum_{i,j} \mathcal{K}_{ij}
  \tilde\lambda_i \big\{\EpotVect_i^{n+1} - \tilde\EpotVect_i^{n+1}\big\}
  \tilde\lambda_j \big\{\EpotVect_j^{n+1} - \tilde\EpotVect_j^{n+1}\big\}}.
\end{gather*}

%%%%%%%%%%%%%%%%%%%%%%%%%%%%%%%%%%%%%%%%%%%%%%%%%%%%%%%%%%%%%%%%%%%%%%%%%%%%%%%%
%%%%%%%%%%%%%%%%%%%%%%%%%%%%%%%%%%%%%%%%%%%%%%%%%%%%%%%%%%%%%%%%%%%%%%%%%%%%%%%%
%%%%%%%%%%%%%%%%%%%%%%%%%%%%%%%%%%%%%%%%%%%%%%%%%%%%%%%%%%%%%%%%%%%%%%%%%%%%%%%%

\bibliographystyle{siamplain}

\end{document}